\documentclass[11pt,reqno]{amsart}
\usepackage{amsmath,amsfonts,amssymb,amsthm,amscd,comment,euscript}
\usepackage[all]{xy}
\usepackage{graphicx}
\usepackage{enumerate} 
\usepackage[colorlinks=true]{hyperref} 
\usepackage[usenames,dvipsnames]{xcolor}
\usepackage{enumitem}

\usepackage{bbm}

\xymatrixcolsep{1.9pc}                          
\xymatrixrowsep{1.9pc}
\newdir{ >}{{}*!/-5\pt/\dir{>}}                  

 \calclayout

\swapnumbers

\theoremstyle{plain}
\newtheorem{lem}{Lemma}[section]
\newtheorem{cor}[lem]{Corollary}
\newtheorem{prop}[lem]{Proposition}
\newtheorem{thm}[lem]{Theorem}
\newtheorem*{thmm}{Theorem}

\theoremstyle{definition}

\newtheorem{rem}[lem]{Remark}
\newtheorem{dfn}[lem]{Definition}

\newtheorem{defs}[lem]{Definition}

\renewcommand{\phi}{\varphi}
\renewcommand{\leq}{\leqslant}
\renewcommand{\geq}{\geqslant}
\renewcommand{\epsilon}{\varepsilon}

\renewcommand{\kappa}{\varkappa}

\DeclareMathOperator{\proj}{proj}

 \DeclareMathOperator{\glob}{proj}

\DeclareMathOperator{\Hom}{Hom} \DeclareMathOperator{\cof}{cof}

 \DeclareMathOperator{\Mor}{Mor}

\DeclareMathOperator{\Ab}{Ab}
\DeclareMathOperator{\Set}{Set}

 \DeclareMathOperator{\im}{Im}

\DeclareMathOperator{\coker}{Coker} \DeclareMathOperator{\nis}{\mathsf{nis}}

\DeclareMathOperator{\hocofib}{hocofib}
\DeclareMathOperator{\local}{local}

\newcommand{\cc}{\mathcal}
\newcommand{\bb}{\mathbb}

\newcommand{\eff}{\mathsf{eff}}

\newcommand{\pt}{\mathsf{pt}}

\newcommand{\Shv}{\mathsf{Shv}}
\newcommand{\Cor}{\mathcal{A}}
\newcommand{\Smk}{\mathsf{Sm}_k}
\newcommand{\Smkplus}{\mathsf{Sm}_{k,+} }
\newcommand{\Psh}{\mathsf{Psh}}
\newcommand{\Sm}{\cc Sm}
\newcommand{\ShvA}{\Shv(\Cor)}
\newcommand{\PshA}{\Psh(\Cor)}

\newcommand{\wgtc}{weakly generating trivial cofibrations}

\newcommand{\homShv}{\mathrm{map}^{\mathrm{Ch(\ShvA)}}}

\newcommand{\otimesPsh}{\underset{\mathrm{Psh}}{\otimes}}
\newcommand{\otimesShv}{\underset{\mathrm{Shv}}{\otimes}}
\newcommand{\otimesDay}{\underset{\mathrm{Day}}{\otimes}}
\newcommand{\otimesL}{\underset{\mathrm{Day}}{\otimes}^{\mathbf{L}}}

\newcommand{\Gmn}[1]{\mathbb{G}_m^{\wedge #1}}
\newcommand{\Gmt}[1]{\mathbb{G}_m^{\times #1}}
\newcommand{\embed}[1]{I(#1)}
\newcommand{\inthom}[3]{\underline{\mathrm{Hom}}_{#1}(#2,#3)}

\newcommand{\psSet}{\Delta^{op}\Set_*}
\newcommand{\Gm}{\mathbb{G}_m}
\newcommand{\colim}[1]{\underset{#1}{\mathsf{colim}}}

\newcommand{\Ch}{\mathsf{Ch}}
\newcommand{\Sp}{\mathsf{Sp}}
\newcommand{\SpGm}{\Sp_{\Gm}}

\newcommand{\motive}[1]{\mathcal{L}(\Sigma^{\infty}_{S^1,\Gm}#1_+)}

\newcommand{\mathscr}{\cc}
\begin{document}

\footskip30pt

\baselineskip=1.1\baselineskip

\title{Triangulated Categories of Big Motives via Enriched Functors}
\address{Department of Mathematics, Swansea University, Fabian Way, Swansea SA1 8EN, UK}
\email{peterjbonart@gmail.com}

\author{Peter Bonart}

\begin{abstract}
Based on homological algebra of Grothendieck categories of enriched functors, two models for 
Voevodsky's category of big motives with reasonable correspondences are given in this paper. 
\end{abstract}

\keywords{Triangulated categories of motives, enriched category theory}
\subjclass[2010]{14F42, 14F08, 18G80}
\thanks{Supported by the Swansea Science Doctoral Training Partnerships, and the Engineering
and Physical Sciences Research Council (Project Reference: 2484592)}

\maketitle
\thispagestyle{empty}
\pagestyle{plain}

\tableofcontents

\section{Introduction}
\label{section:introduction}

In his fundamental paper~\cite{Voe1} Voevodsky defined the triangulated category of motives $DM^{\eff}$ over a (perfect)
field $k$ as the full triangulated subcategory of the derived category $D(\Shv_{tr})$ of Nisnevich sheaves with transfers of those
complexes whose cohomology sheaves are $\bb A^1$-invariant, i.e. the $\bb A^1$-local complexes. The triangulated category of big motives
$DM$ is obtained from $DM^{\eff}$ by stabilisation in the $\bb G_m^{\wedge 1}$-direction.

Let $\Cor$ be a symmetric monoidal category of correspondences that satisfies the strict $V$-property
and cancellation, as defined in \cite{garkusha2019compositio}. Basic examples are given by the categories of finite
correspondences $Cor$ or Milnor--Witt correspondences $\widetilde{Cor}$.
The goal of this paper is to recover the triangulated category of big $\Cor$-motives $DM_{\Cor}$ out of
Grothendieck categories of enriched functors $[\cc B,\ShvA]$ in the sense of \cite{AlGG}, 
where $\cc B$ is either the $\ShvA$-category $\cc C$ of the powers 
$\bb G_m^{\times n}$ or the $\ShvA$-category $\Sm$ of all smooth $k$-schemes.
To this end, we use homological algebra of enriched Grothendieck 
categories developed in~\cite{garkusha2019derived,garkusha2022recollements}.

As we have mentioned above, Voevodsky's construction of $DM^{\eff}_{\cc A}$ is based on the $\bb A^1$-locality
of chain complexes $\Ch(\ShvA)$ of Nisnevich $\cc A$-sheaves. In our setting we consider two types of the $\bb A^1$-locality
of chain complexes in $\Ch([\cc B,\ShvA])$:
one for the contravariant $\bb A^1$-locality in the $\cc A$-direction (i.e. the usual one), denoted by $\bb A^1_1$, another for the covariant $\bb A^1$-locality
in the $\cc B$-direction, denoted by $\bb A^1_2$.
We also consider $\tau$-locality  in $\Ch([\cc B,\ShvA])$ with respect to the family 
   $$\tau = \{[\Gmn{n+1},-]\otimes_{\ShvA} \Gmn{1} \rightarrow [\Gmn{n},-]\mid n\geq 0\}$$
as well as $Nis$-locality in the covariant $\cc B$-direction associated to the elementary Nisnevich squares.
As we work with Grothendieck categories of $\ShvA$-enriched functors here, we say that the relevant chain complexes are
strictly local with respect to the specified family above.
We refer the reader to Section~\ref{ThmStatement} for details.
The relations are also counterparts of the axioms (2)-(5) for special motivic $\Gamma$-spaces in the sense of \cite{garkusha2019framed}
and framed spectral functors in the sense of~\cite[Section~6]{garkusha2018triangulated}. 

Our first reconstruction result states the following (see Theorem~\ref{cthm}).

\begin{thmm} 
Let $\cc C$ be the natural $\ShvA$-category represented by the $\cc A$-sheaves 
$\cc A(-,\bb G_m^{\times n})_{\nis}$, $n\geq 0$. Let $DM_{\Cor}[\cc C]$ be the full triangulated subcategory of the derived category $D([\cc C,\ShvA])$
consisting of the strictly $\bb A^1_1$-local and $\tau$-local complexes. Then the canonical evaluation functor
	$$ev_{\Gm} : DM_{\Cor}[\cc C] \rightarrow DM_{\Cor}$$
	is an equivalence of compactly generated triangulated categories.
\end{thmm}

Our second reconstruction result states the following (see Theorem~\ref{mainthm}).

\begin{thmm} 
Let $\Sm$ be the natural $\ShvA$-category represented by the $\cc A$-sheaves 
$\cc A(-,X)_{\nis}$, $X\in\Smk$. Let $DM_{\Cor}[\Sm]$ be the full triangulated subcategory of the derived category $D([\Sm,\ShvA])$
consisting of the strictly $\bb A^1_1$-, $\tau$-, $Nis$- and $\bb A_2^1$-local complexes. Then the canonical evaluation functor
	$$ev_{\Gm} : DM_{\Cor}[\Sm][1/p] \rightarrow DM_{\Cor}[1/p]$$
	is an equivalence of compactly generated triangulated categories, where $p$ is the exponential characteristic of the base field $k$.
\end{thmm}

It is worth mentioning that the latter result requires recollement theorems of~\cite{garkusha2022recollements}
as well as a generalization of Röndigs--\O stv\ae r's Theorem~\cite{rondigs2008modules} (see Section~\ref{ROSection}).

The results of the paper were first presented at the Conference on Motivic and Equivariant Topology in May 2023 (Swansea, UK).
The author expresses his gratitude to his supervisor Prof. Grigory Garkusha
whose patience and keen insight have been indispensable throughout this work.

\subsubsection*{\bf Notation}
Throughout the paper we use the following notation.\\

\begin{tabular}{l|l}
	$k$ & field of exponential characteristic $p$\\
	$\Smk$ & smooth separated schemes of finite type over $k$ \\
	$\Cor$ & symmetric monoidal additive $V$-category of correspondences\\ 
	$\PshA$ & presheaves of abelian groups on $\Cor$ \\
	$\ShvA$ & Nisnevich sheaves of abelian groups on $\Cor$ \\
	$DM_{\Cor}$ & triangulated category of big motives with $\Cor$-correspondences\\
	$SH(k)$ & stable motivic homotopy category over $k$\\
	$\Sm$ & enriched category of smooth schemes (see Section~\ref{ThmStatement})\\
	$\cc C$ & subcategory of $\Sm$ on $\Gmt{n}$ for $n \in \bb N$ (see Section~\ref{ThmStatement})\\
	$I$ & canonical embedding $\Sm \rightarrow \ShvA$, $X\mapsto\Cor(-,X)_{\nis}$\\
	$M_{\Cor}(X)$ & $\Cor$-motive of $X\in\Smk$\\
	$\mathscr M$ & category of motivic spaces\\
	$f\mathscr M$ & category of finitely presented motivic spaces
\end{tabular}

Also, we assume that $0$ is a natural number.

\section{Categories of correspondences} \label{defrecall}

In this section we recall the definition of a category of correspondeces $\Cor$ and the 
construction of the triangulated category of big motives with $\Cor$-correspondences $DM_{\Cor}$
in the sense of Voevodsky~\cite{Voe1}. 
We shall adhere to~\cite{garkusha2019compositio}.

\begin{defs}\label{tak}{\rm
		
		A \textit{preadditive category of correspondences} $\Cor$ consists of
		
		\begin{enumerate}
			\item a preadditive category $\Cor$ whose objects are those of $\Smk$, 
			called the \textit{underlying preadditive category},
			\item a functor $\Gamma : \Smk \rightarrow \Cor$, called the \textit{graph functor},
			\item a functor
			$\boxtimes:\cc A\times \Smk\to\cc A$
		\end{enumerate}
	such that the following axioms are satisfied:		
		\begin{enumerate}
			
			\item the functor $\Gamma:\Smk\to \Cor$ is the identity on objects;
			
			\item for every elementary Nisnevich square
			$$\xymatrix{U'\ar[r]\ar[d]&X'\ar[d]\\
				U\ar[r]&X}$$
			the sequence of Nisnevich sheaves
			$$0\to\cc A(-,U')_{\nis}\to\cc A(-,U)_{\nis}\oplus\cc A(-,X')_{\nis}\to\cc A(-,X)_{\nis}\to 0$$
			is exact. Moreover, we require $\cc A(-,\emptyset)_{\nis}=0$;
			
			\item for every $\cc A$-presheaf $\cc F$ (i.e. an additive contravariant
			functor from $\cc A$ to Abelian groups Ab)
			the associated Nisnevich sheaf $\cc F_{\nis}$ has a unique structure of an
			$\cc A$-presheaf for which the canonical morphism $\cc F\to\cc
			F_{\nis}$ is a morphism of $\cc A$-pre\-sheaves.
			
			\item the functor $\boxtimes : \Cor \times \Smk \to \Cor$ sends an object 
			$(X,U) \in \Smk\times \Smk$ to $X \times U \in \Smk$ and satisfies $1_X\boxtimes f=\Gamma(1_X\times f)$,
			$(u+v)\boxtimes f=u\boxtimes f+v\boxtimes f$ for all
			$f\in\Mor(Sm/k)$ and $u,v\in\Mor(\cc A)$.
		\end{enumerate}
}\end{defs}

\begin{defs}\label{corproperties}
	\begin{enumerate}
		\item A preadditive category of correspondences $\Cor$ is called an \textit{additive category of correspondences} if its underlying preadditive category is an additive category.
		\item A preadditive category of correspondences $\Cor$ is called a \textit{symmetric monoidal category of correspondences} if its underlying preadditive category $\Cor$ is also equipped with an $\Ab$-enriched symmetric monoidal structure, such that the graph functor $\Gamma : \Smk \rightarrow \Cor$ is a strong monoidal functor with respect to the cartesian monoidal structure on $\Smk$. This means in particular that for $X, Y \in \Smk$ the tensor product $X \otimes Y$ in $\Cor$ is isomorphic to the usual product of schemes $X \times Y$.
		\item A preadditive category of correspondences $\Cor$ is a {\it $V$-category\/} if it satisfies the {\it $V$-property}.
		The $V$-property says that for any $\bb A^1$-invariant $\Cor$-presheaf of abelian groups $\cc F$ the associated Nisnevich sheaf $\cc F_{\nis}$ is $\bb A^1$-invariant, in the sense that for all $X\in \Smk$ the map $$\cc F_{\nis}(X) \rightarrow  \cc F_{\nis}(X \times \bb A^1) $$ induced by the projection $X \times \bb A^1 \rightarrow X$ is an isomorphism.
		\item Recall from~\cite{Voe} that a Nisnevich sheaf $\cc F$ of abelian groups is {\it
strictly $\bb A^1$-invariant\/} if for any $X\in {Sm}/k$, the
canonical morphism
   $$H^*_{\nis}(X,\cc F)\to H^*_{\nis}(X\times\bb A^1,\cc F)$$
is an isomorphism.
A $V$-category of correspondences $\cc A$ is a {\it strict
$V$-category of correspondences\/} if for any $\bb A^1$-invariant
$\cc A$-presheaf of abelian groups $\cc F$ the associated Nisnevich
sheaf $\cc F_{\nis}$ is strictly $\bb A^1$-invariant.
\item 

For $i \leq k+1 \in \mathbb{N}$ let $\iota_{i,k} : \Gmt{k} \rightarrow \Gmt{k+1}$ be the inclusion map in $\Smk$ sending $(x_1,\dots,x_k)$ to $(x_1,\dots,x_{i-1},1,x_{i+1},\dots,x_k)$.
In $\ShvA$ define $$\Gmn{k} := \cc A(-,{\Gmt{k}})_{\nis} / \underset{i=1}{\overset{k}{\sum}} \im(\embed{\iota_{i,k-1}}).$$
Furthermore, let $$\Delta_k^n := \mathrm{Spec}(k[t_0,\dots,t_n]/(t_0+\dots+t_n-1)).$$
Similarly to \cite[Definition 3.5]{garkusha2019compositio} we can define bivariant $\Cor$-motivic cohomology groups by
$$H_{\Cor}^{p,q}(X,Y):=H^p_{\nis}(X,\Cor(-\times \Delta_k^\bullet,Y\wedge \Gmn{q})_{\nis}[-q]),$$
where the $H^p_{\nis}$ on the right hand side refers to Nisnevich hypercohomology groups.
We say that a strict $V$-category of correspondences $\Cor$ \textit{satisfies the cancellation property} if all the canonical maps
$$\beta^{p,q} : H_{\Cor}^{p,q}(X,Y) \rightarrow H_{\Cor}^{p+1,q+1}(X\wedge \Gmn{1},Y)$$
are isomorphisms.
	\end{enumerate}	
\end{defs}

{\it From now on, $\Cor$ is an additive symmetric monoidal strict $V$-category of correspondences. 
From Section \ref{ThmStatement} onwards we will furthermore assume that $\Cor$ satisfies the cancellation property.}
Non-trivial examples are given by finite correspondences $Cor$ in the sense of
Voevodsky~\cite{Voe1}, finite Milnor--Witt correspondences $\widetilde{Cor}$
in the sense of Calm\`es--Fasel~\cite{CF} or $K_0^{\oplus}$ in the
sense of Walker~\cite{Wlk}. Given a ring $R$ (not necessarily
commutative) which is flat as a $\bb Z$-algebra and a category of
correspondences $\cc A$, we can form an {\it additive category of
correspondences\/} $\cc A_R$ with coefficients in $R$. By
definition, $\cc A_R(X,Y):=\cc A(X,Y)\otimes R$ for all $X,Y\in
\Smk$.

We are now passing to the construction of Voevodsky's triangulated category of big motives with $\Cor$-correspondences $DM_{\Cor}$.

Let $\ShvA$ be the Grothendieck category of Nisnevich sheaves on $\Cor$ with values in abelian groups. 
The category $\ShvA$ of $\Ab$-valued Nisnevich sheaves on $\Cor$ is symmetric closed monoidal 
with the Day convolution product \cite{day1970closed} that is induced by the monoidal structure of $\Cor$. 
The internal hom of $\ShvA$ will be denoted sometimes by $\inthom{\ShvA}{-}{-}$, and sometimes 
by $[-,-]$ if there is no likelihood of confusion.

Let $D(\ShvA)$ be the derived category of $\ShvA$. 
Consider the localizing subcategory $\cc L$ in $D(\ShvA)$ that is compactly generated by the shifts of the complexes
$$\cdots \rightarrow 0 \rightarrow \Cor(-,X \times \bb A^1)_{\nis} \rightarrow \Cor(-,X)_{\nis} \rightarrow 0 \rightarrow \cdots$$
for all $X \in \Smk$.

By general localization theory for triangulated categories \cite{neeman2001triangulated} 
we can form the quotient triangulated category $D(\ShvA)/ \cc L$.

\begin{defs}
We call $DM_{\Cor}^{\eff} := D(\ShvA)/ \cc L$ the  \textit{triangulated category of effective motives with $\Cor$-correspondences}.  It can be identified with the full subcategory of  $D(\ShvA)$ of those objects that have $\bb A^1$-invariant cohomology sheaves.
\end{defs}

In $DM_{\Cor}^{\eff}$ we can $\otimes$-invert $\Gmn{1}$ using a procedure similar to~\cite[5.2]{morel131motivic}. Namely,
we define a $\Gmn{1}$-spectrum of chain complexes $C$ to be a collection $(C_n,\sigma_n)_{n\in \bb N}$ consisting for each $n \in \bb Z_{\geq 0}$ of a chain complex $C_n \in \Ch(\ShvA)$, and a morphism of chain complexes $\sigma_n: C_n \otimes \Gmn{1} \rightarrow C_{n+1}$. A morphism of $\Gmn{1}$-spectra of chain complexes is a graded morphism of complexes respecting the structure maps $\sigma_n$. The category of $\Gm$-spectra of chain complexes is denoted $\Sp_{\Gm}(\Ch(\ShvA))$.

\begin{defs} 
	\label{homologydef}
	\begin{enumerate}
		\item 
		Let $I : \Smk \rightarrow \ShvA$ be the obvious inclusion functor $I(X) := \Cor(-,X)_{\nis}$. For any $\Gmn{1}$-spectrum of chain complexes $C$ we define \textit{presheaves of homology groups} by assigning to each $U \in \Smk$ and $n, m \in \bb Z$ the group $\bb H_n(C)_m(U)$ as the colimit of the diagram
		$$\dots \rightarrow Hom_{DM_{\Cor}^{\eff}} (I(U)[n-m] \otimes \Gmn{m+r} , C_r) \rightarrow \dots$$
		ranging over $r \in \bb N$.
		\item
		A morphism of $\Gmn{1}$-spectra of chain complexes is called a \textit{stable motivic equivalence} if it induces isomorphisms on these homology presheaves.
		\item
		We define $DM_{\Cor}$ to be the category obtained from $\Sp_{\Gm}(\Ch(\ShvA))$ by inverting the stable motivic equivalences.
	\end{enumerate}
\end{defs}

\section{A model structure on $\Ch(\ShvA)$} \label{sectionChShvAModel}

Let $\Cor$ be a symmetric monoidal category of correspondences satisfying the $V$-property.
The goal of this section is to construct a monoidal model structure on $\Ch(\ShvA)$ that is weakly finitely generated (Definition \ref{defweakfingen}), satisfies the monoid axiom \cite[Definition 3.3]{schwede2000algebras}, and in which the weak equivalences are the quasi-isomorphisms. Once we have such a model structure we can use \cite[Theorem 5.5]{garkusha2019derived} to construct the projective model structure on the category of chain complexes $\Ch([\cc C, \ShvA])$ of the Grothendieck category of enriched functors $[\cc C, \ShvA]$ for any small $\ShvA$-enriched category $\cc C$.
The model structure will be useful for proving the reconstruction theorems of the next two chapters.

There is a finitely generated monoidal model structure on the category of unbounded chain compelxes of abelian groups $\Ch(\Ab)$, where weak equivalences are quasi-isomorphisms and fibrations are epimorphisms \cite{strickland2020model}.
This model structure also satisfies the monoid axiom in the sense of~\cite[Definition 3.3]{schwede2000algebras}.
For any abelian group $A$, let $S^n A$ be the chain complex that is $A$ in degree $n$ and $0$ everywhere else. Let $D^n A$ be the chain complex that is $A$ in degree $n$ and $n+1$, and $0$ everywhere else, and where the differential from degree $n+1$ to degree $n$ is the identity map on $A$. For every $n \in \bb Z$ there is a canonical map $S^n A \rightarrow D^n A$ which is $id_A$ in degree $n$.
A set of generating cofibrations of $\Ch(\Ab)$ is $I_{\Ch} := \{S^n \bb Z \rightarrow D^n \bb Z \mid n \in \bb Z \}$, and a set of generating trivial cofibrations is $J_{\Ch} := \{0 \rightarrow D^n \bb Z \mid n \in \bb Z \}$.

Let $\PshA$ be the category of $\Ab$-enriched functors $\cc A^{op} \rightarrow \Ab$.
We can then apply \cite[Theorem 5.5]{garkusha2019derived} to get a weakly finitely generated monoidal model structure on $\Ch(\PshA)$, where weak equivalences are sectionwise quasi-isomorphisms, and the fibrations are epimorphisms. We call it the \textit{standard projective model structure on presheaves}, or sometimes just the \textit{projective model structure on presheaves}.
The proof of \cite[Theorem 4.2]{dundas2003enriched} shows that the generating cofibrations and generating trivial cofibrations of this model structure are given by the sets $$I_{\glob} := \{ \Cor(-,X) \otimes S^n\bb Z \rightarrow \Cor(-,X) \otimes D^n \bb Z | X \in \Smk, n \in \bb Z \}$$ $$J_{\glob} := \{ 0 \rightarrow \Cor(-,X) \otimes D^n \bb Z | X \in \Smk, n \in \bb Z \}.$$
From \cite[Theorem 4.4]{dundas2003enriched} it also follows that this model structure on $\Ch(\PshA)$ satisfies the monoid axiom.

\begin{lem} \label{degreesplitmono}
	Every cofibration in the projective model structure on $\Ch(\PshA)$ is a degreewise split monomorphism with degreewise projective cokernel.
	
\end{lem}
\begin{proof} 
	Take a cofibration $f : A \rightarrow B$ in $\Ch(\PshA)$.
	Take an arbitrary $n \in \bb Z$.
	Define a morphism of complexes $\phi : A \rightarrow D^n(A_n)$ by means of the following diagram
	$$\xymatrix{\dots \ar[r]^{\partial_A^{n+3}} & A_{n+2} \ar[r]^{\partial_A^{n+2}} \ar[d] & A_{n+1} \ar[r]^{\partial_A^{n+1}} \ar[d]^{\partial_A^{n+1}} & A_{n} \ar[r]^{\partial_A^{n}} \ar[d]^{id} & A_{n-1} \ar[r]^{\partial_A^{n-1}} \ar[d] & \dots \\
		\dots \ar[r] & 0 \ar[r] & A_n \ar[r]^{id} & A_n \ar[r] & 0  \ar[r] & \dots }$$
	In the following commutative diagram in $\Ch(\PshA))$ the right hand side morphism is a surjective quasi-isomorphism, i.e. a projective trivial fibration
	$$\xymatrix{A \ar[d]_f \ar[r]^{\phi} & D^n(A_n) \ar[d] \\
		B \ar[r] \ar@{..>}^s[ur] & 0}$$
	So we get a lift $s : B \rightarrow D^n(A_n)$ with $s \circ f = \phi$. In particular $s_n \circ f_n = \phi_n = id_{A_n}$. Since $n \in \bb Z$ was arbitrary, $f$ is a degreewise split monomorphism.
	
	We have a pushout diagram:
	$$\xymatrix{A \ar[r]^{f}\ar[d]  & B \ar[d]\\
		0 \ar[r] & \coker(f)}$$
	Since the upper map is a cofibration, the lower map is a cofibration. So $\coker(f)$ is a cofibrant object.
	To show that $f$ is a degreewise split monomorphism with degreewise projective cokernel, we now just need to show that every cofibrant object in $\Ch(\PshA)$ is degreewise projective.
	
	Let $C$ be any cofibrant object in $\Ch(\PshA)$, and let $n \in \bb Z$. We claim that $C_n$ is projective in $\PshA$. Take an arbitrary epimorphism $p : X \rightarrow Y$ in $\PshA$ and an arbitrary map $g: C_n \rightarrow Y$ in $\PshA$. We need to find a lift in the diagram
	$$\xymatrix{& X \ar[d]^{p}\\
		C_n \ar@{..>}[ur] \ar[r]_g & Y}$$
	Just like at the begining of the lemma we can construct a morphism of chain complexes $\phi : C \rightarrow D^n(C_n)$ with $\phi_n = id_{C_n}$, $\phi_{n+1}=\partial_C^{n+1}$ and $\phi_k = 0$ for $k \notin\{n,n+1\}$.
	In $\Ch(\PshA)$ we then have a diagram
	$$\xymatrix{0 \ar[rr] \ar[d] & & D^n(X) \ar[d]^{D^n(p)} \\
		C \ar[r]_(0.4){\phi} \ar@{..>}[urr]^s & D^n(C_n) \ar[r]_{D^n(g)} & D^n(Y) } $$
	We claim that in this diagram a lift $s: C \rightarrow D^n(X)$ exists. This is true for the following reason:
	Since $p$ is an epimorphism, $D^n(p)$ is an epimorphism, so $D^n(p)$ is a projective fibration. Since $D^n(X)$ and $D^n(Y)$ are both acyclic, it follows that $D^n(p)$ is a quasi-isomorphism, so $D^n(p)$ is a trivial fibration. Since $0 \rightarrow C$ is a cofibration, it follows that the lift $s : C \rightarrow D^n(X)$ exists. 
	Then $s_n : C_n \rightarrow X$ satisfies $p \circ s_n = g$, and this then shows that $C_n$ is projective.    
\end{proof}

\begin{cor}
	The standard projective model structure on $\Ch(\PshA)$ is cellular, in the sense of \cite[Definition 12.1.1]{hirschhorn2003model}
\end{cor}

\begin{proof}
	The domains and codomains from $I_{\glob}$ and $J_{\glob}$ are compact.
	By Lemma \ref{degreesplitmono} every cofibration is a degreewise split monomorphism.
	Since $\Ch(\PshA)$ is an abelian category, every monomorphism is an effective monomorphism.
	So every cofibration is an effective monomorphism, and the projective model structure on $\Ch(\PshA)$ is cellular.
\end{proof}

We next apply a left Bousfield localization on the projective model structure on presheaves.

\begin{dfn} \label{localprojdef}
	
	Let $\cc Q$ be the set of all elementary Nisnevich squares in $\Smk$.
	We want to make the following class of maps in $\Ch(\PshA)$ into weak equivalences:
	\begin{enumerate}
		\item
		The morphism $0 \rightarrow \Cor(-,\emptyset)$ will be a weak equivalence.
		\item
		For every elementary Nisnevich square $Q \in \cc Q$ of the form $$\xymatrix{ U^\prime \ar[r]^\beta \ar[d]^\alpha & X^\prime \ar[d]^\gamma \\
			U \ar[r]^\delta & X} $$
		we get a square 
		$$ \xymatrix{ \Cor(-,U^\prime) \ar[r]^{\beta_*} \ar[d]^{\alpha_*} & \Cor(-,X^\prime) \ar[d]^{\gamma_*} \\
			\Cor(-,U) \ar[r]^{\delta_*}  & \Cor(-,X) } $$
		in $\Ch(\PshA)$ (we regard each entry of the square as a complex concentrated in zeroth degree).
		We take the mapping clyinder $C$ of the map $\Cor(-,U^\prime) \rightarrow \Cor(-,X^\prime)$. So the map factors as 
		$\xymatrix{\Cor(-,U^\prime) \: \ar@{>->}[r]& C \ar@{->>}[r]^(0.4){\sim} & \Cor(-,X^\prime)}$, where the first map is a cofibration, the second map is a trivial fibration and $C$ is finitely presented. Let $s_Q := \Cor(-,U) \underset{\Cor(-,U^{\prime})}{\coprod} C$. Then $s_Q$ is also finitely presented. 
		Notice that $s_Q$ is the homotopy pushout of $\Cor(-,U)$ and $\Cor(-,X^\prime)$ over $\Cor(-,U^\prime)$. Take the mapping cylinder $t_Q$ of the map $s_Q = \Cor(-,U) \underset{\Cor(-,U^{\prime})}{\coprod} C \rightarrow \Cor(-,X)$, so that it factors as a cofibration followed by a trivial fibration $\xymatrix{s_Q \: \ar@{>->}[r]^{p_Q}& t_Q \ar@{->>}[r]^(0.4){\sim} & \Cor(-,X) }$, and $t_Q$ is finitely presented.\\
		For every $Q \in \cc Q$ this cofibration $p_Q: s_Q \rightarrow t_Q$ will be a weak equivalence.
	\end{enumerate}
	
	Our notation here is similar to that of \cite[Notation 2.13]{dundas2003motivic}.
	Denote the set of all the shifts of these morphisms by $S = \{0 \rightarrow \Cor(-,\emptyset)[n] \mid n \in \bb Z \} \cup \{p_Q[n] | Q  \in \cc Q, n \in \bb Z \}$.
	We can apply \cite[Theorem 4.11]{hirschhorn2003model} to get the left Bousfied localization of the projective model structure of presheaves with respect to $S$. We call the resulting model structure the \textit{local projective model structure on presheaves}.
	We write $I_{\local}$, $J_{\local}$ for the generating cofibrations, generating trivial cofibrations and \wgtc\space of the local projective model structure on $\Ch(\PshA)$.
	
\end{dfn}

We will say that an object $F \in \Ch(\PshA)$ is \textit{locally fibrant}, if it is fibrant in the local projective model structure.

\begin{lem}\label{lemmaChPshALocalFibrant}
	An object $F \in \Ch(\PshA)$ is locally fibrant if and only if
	$F(\emptyset) \rightarrow 0$ is a quasi-isomorphism in $\Ch(\Ab)$, and $F$ sends elementary Nisnevich squares to homotopy pullback squares.
\end{lem}
\begin{proof}
	Let $\tau_{\geq 0} : \Ch(\PshA) \rightarrow \Ch_{\geq 0}(\PshA)$ be the good truncation functor, sending
	$$\dots \rightarrow A_1 \rightarrow A_0 \overset{\partial_A^0}{\rightarrow} A_{-1} \rightarrow \dots $$
	to
	$$\dots \rightarrow A_1 \rightarrow \ker(\partial_A^0).$$
	For $A, B \in \Ch(\PshA)$ let $\inthom{\Ch(\PshA)}{A}{B}$ be the internal hom of $\Ch(\PshA)$ and let $\mathrm{map}^{\Delta^{op}\Set}(A,B) \in \Delta^{op}\Set$ be the derived simplicial mapping space.
	Define $$\mathrm{map}^{\Ch_{\geq 0}(\Ab)}(A,B) := \tau_{\geq 0}(\inthom{\Ch(\PshA)}{A}{B}(pt)) \in \Ch_{\geq 0}(\Ab).$$
	If $A$ is cofibrant and $B$ is fibrant, then for every $n \geq 0$ we have an isomorphism of abelian groups
	$$H_n(\mathrm{map}^{\Ch_{\geq 0}(\Ab)}(A,B)) \cong \pi_n(\mathrm{map}^{\Delta^{op}\Set}(A,B)).$$
	
	By \cite[Definition 3.1.4]{hirschhorn2003model} an object $F \in \Ch(\PshA)$ is locally fibrant if and only if
	for every $s : A \rightarrow B, $ with $s \in S$ the map
	$$s^* : \mathrm{map}^{\Delta^{op}\Set}(B,F) \rightarrow \mathrm{map}^{\Delta^{op}\Set}(A,F) $$
	is a weak equivalence of simplicial sets.
	Since $s$ is a cofibration between cofibrant objects, and every object in $\Ch(\PshA)$ in the standard projective model structure is fibrant, it follows that $F$ is locally fibrant if and only if
	$$s^* : \mathrm{map}^{\Ch_{\geq 0}(\Ab)}(B,F) \rightarrow \mathrm{map}^{\Ch_{\geq 0}(\Ab)}(A,F) $$
	is a quasi-isomorphism in $\Ch_{\geq 0}(\Ab)$.
	If $s$ is of the form $0 \rightarrow \Cor(-,\emptyset)[n]$, this means that the map
	$$\tau_{\geq 0}(F(\emptyset)[-n]) \rightarrow 0$$
	is a quasi-isomorphism.
	This holds for every $n \in \bb Z$ if and only if
	$0 \rightarrow F(\emptyset)$
	is a quasi-isomorphism.
	If $s$ is of the form $p_Q: s_Q \rightarrow t_Q$ for an elementary Nisnevich square $Q$ of the form
	$$\xymatrix{ U^\prime \ar[r]^\beta \ar[d]^\alpha & X^\prime \ar[d]^\gamma \\
		U \ar[r]^\delta & X} $$then this means that the map
	$$\tau_{\geq 0}(F(X)[-n]) \rightarrow \tau_{\geq 0}((F(X^\prime)  \underset{F(U^\prime)}{\times^h}F(U) )[-n]) $$
	is a quasi-isomorphism in $\Ch(\Ab)$, where $F(X^\prime)\underset{F(U^\prime)}{\times^h}F(U)$ is the homotopy pullback of $F(U) \rightarrow F(U^\prime) \leftarrow F(X^\prime)$.
	This holds for every $n \in \bb Z$ if and only if
	$$F(X) \rightarrow F(X^\prime)  \underset{F(U^\prime)}{\times^h}F(U) $$
	is a quasi-isomorphism in $\Ch(\Ab)$, which is the case if and only if $F$ sends $Q$ to a homotopy pullback square.
\end{proof}
The property of sending elementary Nisnevich squares to homotopy pullback squares is also called the B.G.-property in \cite{MV}.
We now prove basic facts about the local projective model structure.
\begin{lem}\label{stalkwiseqis}
	A morphism $f : A \rightarrow B$ in $\Ch(\PshA)$ is a weak equivalence in the local projective model structure if and only if 
	it is a local quasi-isomorphism, in the sense that it is a stalkwise quasi-isomorphism with respect to the Nisnevich topology.
\end{lem}
\begin{proof}
	This follows using a similar argument as in \cite[ C.2.1]{kahn2010motives}. They
	use finite correspondences, but all the arguments of \cite[\S C.2]{kahn2010motives} work 
	for an arbitrary additive symmetric monoidal category of correspondences satisfying the strict $V$-property.
\end{proof}

\begin{lem} \label{projectiveimpliesflat}
	Let $C \in \PshA$ be projective.
	Then $C$ is flat, in the sense that $$C \underset{\Psh}{\otimes} - : \PshA \rightarrow \PshA$$ is an exact functor.
\end{lem}
\begin{proof}
	Since $\PshA$ is an abelian category with enough projectives, we know that for every $A \in \PshA$ the tensor product functor $A \underset{\Psh}{\otimes} - $ has  left derived functors $$\mathrm{Tor}^{\Psh}_i(A,-) : \PshA \rightarrow \PshA$$ for $i \geq 0$.
	By \cite[Corollary 2.4.2]{weibel1996introduction}, if $C$ is projective, then $$\mathrm{Tor}^{\Psh}_i(A,C) = 0$$ for all $i \neq 0$ and all $A \in \PshA$.
	Since $\mathrm{Tor}^{\Psh}_i$ is symmetric we therefore also get $\mathrm{Tor}^{\Psh}_i(C,A) = 0$.
	But this then means that the functor $C \underset{\Psh}{\otimes} - : \PshA \rightarrow \PshA$ is exact.
\end{proof}

\begin{lem}\label{degreewiseflat}
	Let $C \in \Ch(\PshA)$ be a degreewise flat chain complex. Then $C$ is a flat chain complex in the sense that $$C \otimes - : \Ch(\PshA) \rightarrow \Ch(\PshA) $$ is an exact functor.
\end{lem}
\begin{proof}
	Since the functor $C \otimes -$ is right exact, we just need to show that $C \otimes -$ preserves monomorphisms.
	Let $\iota : A \rightarrow B$ be a monomorphism in $\Ch(\PshA)$.
	For every $n\in \bb Z$ we have
	$$ (C \otimes \iota)_n = \underset{p + q = n}{\bigoplus} C_p \otimes \iota_q.$$
	Since each $C_p$ is flat and each $\iota_q$ is a monomorphism, each $C_p \otimes \iota_q$ is a monomorphism. Then $(C \otimes \iota)_n$ is a monomorphism because it is a direct sum of monomorphisms. So $C \otimes \iota$ is a monomorphism, and therefore $C$ is flat in $\Ch(\PshA)$.
\end{proof}

There is an adjunction $L_{\nis}: \PshA \rightleftarrows \ShvA : U_{\nis}$, where the left adjoint $L_{\nis}$ is Nisnevich sheafification and the right adjoint $U_{\nis}$ is the forgetful functor. The sheafification functor $L_{\nis}$ is well-defined because one of the axioms of the category of correspondences $\Cor$ states that for every $\Cor$-presheaf the associated sheaf with respect to the Nisnevich topology on $\Smk$ has a unique strucutre of an $\Cor$-presheaf.
This adjunction extends to an adjunction on chain complexes 
$$L_{\nis}: \Ch(\PshA) \rightleftarrows \Ch(\ShvA) : U_{\nis}.$$

\begin{lem}
	The local projective model structure on $\Ch(\PshA)$ is monoidal.
\end{lem}
\begin{proof}
	We use \cite[Theorem B]{white2014monoidal}.
	Cofibrant objects in the local projective model structure are also cofibrant in the standard projective model structure. By Lemma \ref{degreesplitmono} they are degreewise projective, and therefore degreewise flat by Lemma \ref{projectiveimpliesflat}, and therefore flat by Lemma \ref{degreewiseflat}.
	We now need to show for every elementary Nisnevich square $Q$ and cofibrant object $K$ that the morphism $$K \otimes p_Q : K \otimes s_Q \rightarrow K \otimes t_Q$$ is a local quasi-isomorphism.
	For this it suffices to show that the sheafification $L_{\nis}(K \otimes p_Q)$ is a local quasi-isomorphism.
	Since $L_{\nis} : \Ch(\PshA) \rightarrow \Ch(\ShvA)$ is a strong monoidal functor we have
	$$L_{\nis}(K \otimes p_Q) \cong L_{\nis}(K) \otimes L_{\nis}(p_Q).$$
	
	Since $K$ is a cofibrant object in $\Ch(\PshA)$, it follows that $K$ is flat in $\Ch(\PshA)$. This then also implies that the sheafification $L_{\nis}K$ of $K$ is flat in $\Ch(\ShvA)$, and this implies that the functor
	$L_{\nis}(K) \otimes - : \Ch(\ShvA) \rightarrow \Ch(\ShvA)$ preserves local quasi-isomorphisms.
	Since $p_Q$ is a local quasi-iso\-mor\-phism, it follows that $L_{\nis}(K) \otimes L_{\nis}(p_Q)$ is a local quasi-iso\-mor\-phism. So $K \otimes p_Q$ is a local quasi-iso\-mor\-phism. Similarly $0 \rightarrow K \otimes \Cor(-,\emptyset)$ is a local quasi-iso\-mor\-phism. With this we have proved the lemma.
\end{proof}

We want to show that the local projective model structure is weakly finitely generated in the sense of \cite[Definition 3.4]{dundas2003enriched}. For the convenience of the reader we recall this notion here.

\begin{dfn}\label{defweakfingen}
	A cofibrantly generated model category $M$ is said to be \textit{weakly finitely generated}, if it is 
	cofibrantly generated and the generating cofibrations $I$ and generating trivial cofibrations $J$ can be chosen such that
	\begin{enumerate}
		\item The domains and codomains of maps in $I$ are finitely presented.
		\item The domains of maps in $J$ are small.
		\item
		There exists a subset $J^\prime \subseteq J$ of maps with finitely presented 
		domains and codomains, such that for every map $f : A \rightarrow B$, 
		if $B$ is fibrant and $f$ has the right lifting property with respect to $J^\prime$, then $f$ is a fibration.
	\end{enumerate}
	We will call $J^\prime$ the set of \textit{\wgtc}.
\end{dfn}

Let $I_{\Ch_{\geq 0}} = \{S^n \bb Z \rightarrow D^n \bb Z \mid n \geq 0 \} \cup \{0 \rightarrow S^0 \bb Z\}$ be a set of generating cofibrations for the standard projective model structure on the category of connective chain complexes $\Ch_{\geq 0}(\Ab)$.
Let $S \square I_{\Ch_{\geq 0}}$ denote the set of all maps which are pushout-products of maps in $S$ and $I_{\Ch_{\geq 0}}$.

\begin{lem}\label{SsquareI}
	An object $F \in \Ch(\PshA)$ is fibrant in the local projective model structure if and only if the map $F \rightarrow 0$ has the right lifting property with respect to $S \square I_{\Ch_{\geq 0}}$.
\end{lem}
\begin{proof}
	For $A, B \in \Ch(\PshA)$ let $\mathrm{map}^{\Ch(\Ab)}(A,B) \in \Ch_{\geq 0}(\Ab)$ denote the good truncation of the chain complex of morphisms $A \rightarrow B$, just like in the proof of Lemma \ref{lemmaChPshALocalFibrant}.
	An object $F \in \Ch(\PshA)$ is $S$-local if and only if for every $s : X \rightarrow Y$, $s \in S$ the map
	$$s^* : \mathrm{map}^{\Ch(\Ab)}(Y,F) \rightarrow \mathrm{map}^{\Ch(\Ab)}(X,F)$$
	is a quasi-isomorphism. Since $s$ is a cofibration and $F$ is fibrant, the map $s^*$ is a fibration in $\Ch(\Ab)$.
	So $s^*$ is a quasi-isomorphism in $\Ch_{\geq 0}(\Ab)$ if and only if $s^*$ is trivial fibration in $\Ch_{\geq 0}(\Ab)$, and that is the case if and only if $s^*$ has the right lifting property with respect to $I_{\Ch_{\geq 0}}$.
	For every $\iota : A\rightarrow B$ in $I_{\Ch_{\geq 0}}$ we have that the following diagram has a lift
	$$\xymatrix{A \ar[d]_{\iota} \ar[r] & \mathrm{map}^{\Ch(\Ab)}(Y,F) \ar[d]^{s^*} \\
		B \ar[r] \ar@{..>}[ur] &  \mathrm{map}^{\Ch(\Ab)}(X,F)} $$
	in $\Ch_{\geq 0}(\Ab)$ if and only if the following diagram has a lift	$$\xymatrix{A\otimes Y \underset{A \otimes X}{\coprod} B \otimes X \ar[d]_{\iota \square s} \ar[r] & F \ar[d] \\
		B \otimes Y \ar[r] \ar@{..>}[ur] &  0 }$$
	in $\Ch(\PshA)$. So $F$ is fibrant in the local projective model structure if and only if $F \rightarrow 0$ has the right lifting property with respect to $S \square I_{\Ch_{\geq 0}}$.
\end{proof}

\begin{lem}
	The local model structure on $\Ch(\PshA)$ is weakly finitely generated. A set of weakly generating trivial cofibrations is given by $J_{\local}^{\prime} := J_{\proj} \cup (S \square I_{\Ch_{\geq 0}})$.
\end{lem}
\begin{proof}
	The domains and codomains from $J_{\local}^{\prime}$ are clearly finitely presented.
	
	All morphisms from $J_{\proj}$ are local projective trivial cofibrations. Since $S$ consists out of cofibrations that are $S$-local equivalences, it consists out of local projective trivial cofibrations. Since the local projective model structure is monoidal, it follows that $S \square I_{\Ch_{\geq 0}}$ consists out of local projective trivial cofibrations. So all morphisms from $J_{\local}^{\prime}$ are trivial cofibrations in the local projective model structure, so $J_{\local}^{\prime} \subseteq J_{\local}$ for a suitable choice of $J_{\local}$. 
	
	Let $f : A \rightarrow B$ be a map in $\Ch(\PshA)$, where $B$ is fibrant in the local projective model structure and $f$ satisfies the right lifting property with respect to $J_{\local}^{\prime} = J_{\proj} \cup (S \square I_{\Ch_{\geq 0}})$.
	Then $f$ satisfies the right lifting property with respect to $J_{\proj}$, so $f$ is a fibration in the standard projective model structure. Since $f: A \rightarrow B$ and $B \rightarrow 0$ satisfy the right lifting property with respect to $S \square I_{\Ch_{\geq 0}}$, also the composition $A \rightarrow 0$ satisfies the right lifting property with respect to $S \square I_{\Ch_{\geq 0}}$. By Lemma \ref{SsquareI} it follows that $A$ is fibrant in the local projective model structure.
	From \cite[Proposition 3.3.16]{hirschhorn2003model} it follows that $f$ is a fibration in the local projective model structure. So the local projective model structure on $\Ch(\PshA)$ is weakly finitely generated with $J_{\local}^{\prime}$ as the set of \wgtc.
\end{proof}

We next want to transfer the local projective model structure along the adjunction
$$L_{\nis}: \Ch(\PshA) \rightleftarrows \Ch(\ShvA) : U_{\nis}.$$

\begin{dfn}
	Given a model category $M$ and an adjunction $L : M \rightleftarrows N: R$ we say that the \textit{left transferred model structure along $L$ exists} if there is a model structure on $N$ such that a morphism $f$ in $N$ is a weak equivalence (resp. fibration) if and only if $R(f)$ is a weak equivalence (resp. fibration) in $M$.
\end{dfn}
\begin{rem}
	Let $M$ be a model category and $L : M \rightleftarrows N: R$ an adjunction.
	If the left transferred model structure along $L$ exists, then the adjunction $L : M \rightleftarrows N: R$ is a Quillen adjunction.
	If $M$ is cofibrantly generated with generating cofibrations $I$ and generating trivial cofibrations $J$ and if $L(I)$ and $L(J)$ permit the small object argument in $N$, then $L(I)$ is a set of generating cofibrations and $L(J)$ is a set of generating trivial cofibrations for $N$.
\end{rem}

We next want to show that the left transferred model structure along $L_{\nis} : \Ch(\PshA) \rightarrow \Ch(\ShvA)$ exists.

\begin{lem} \label{Uaccessible}
	The forgetful functor $U_{\nis} : \Ch(\ShvA) \rightarrow \Ch(\PshA) $ preserves filtered colimits.
\end{lem}
\begin{proof}
	This follows from the fact that every covering in the Nisnevich topology has a finite subcovering.
	To spell it out in more detail, let $I$ be a filtered diagram and $A_{(-)} : I \rightarrow \ShvA$ a functor.
	Let $A := \colim{i \in I} U_{\nis}(A_i)$.
	We need to show that the canonical map $$A \rightarrow U_{\nis}(\colim{i \in I} A_i)$$
	is an isomorphism. If we apply $L_{\nis}$ to this map then it clearly becomes an isomorphism in $\ShvA$. Also the presheaf $U_{\nis}(\colim{i \in I} A_i)$ is a sheaf. To prove the lemma, 
	it now suffices to show that the presheaf $A$ is a sheaf.
	
	Take a Nisnevich covering $\{Y_j \rightarrow X\}_{j \in J}$, and compatible sections $s_j \in A(Y_j)$. 
	Since every covering has a finite subcovering we can assume without loss of generality that the index set $J$ is finite.
	Now for each $j \in J$, there exists some $i_j \in I$ so that $s_j$ is the restriction of some section $t_{i,j} \in U_{\nis}(A_{i_j})(Y_j)$ along the canonical map $U_{\nis}(A_{i_j}) \rightarrow A$. Since $I$ is a filtered category, we can find a single $k \in I$ such that every $s_j$ is the restriction of some section $t_j \in U_{\nis}(A_k)(Y_j)$ along the map $U_{\nis}(A_k) \rightarrow A$.
	Since $A_k$ is a sheaf we can glue together all the sections $t_j$ into a single section $t \in U_{\nis}(A_k)(X)$. If we include $t$ into the colimit $\colim{i \in I} U_{\nis}(A_i)(Y_j)$ then we get a section $s \in A(X)$ which is a unique gluing of all the $s_j$. So $A$ is a sheaf, and $U_{\nis}$ preserves filtered colimits.
\end{proof}

\begin{cor}\label{Lnisfinpres}
	$L_{\nis} : \Ch(\PshA) \rightarrow \Ch(\ShvA)$ preserves finitely presented objects.
\end{cor}

\begin{proof}
	Let $X \in \Ch(\PshA)$ be finitely presented. Let $I$ be a filtered diagram, and let $A_{(-)} : I \rightarrow \Ch(\ShvA)$ be a functor. Then using Lemma \ref{Uaccessible} we get
	\begin{align*}
		\Hom_{\Ch(\ShvA)}(L_{\nis}X, \colim{i \in I} A_i) & \cong \Hom_{\Ch(\PshA)}(X, U_{\nis} \colim{i \in I} A_i) \overset{\ref{Uaccessible}}{\cong} \\ \Hom_{\Ch(\PshA)}(X, \colim{i \in I} U_{\nis} A_i) & \cong \colim{i \in I} \Hom_{\Ch(\PshA)}(X,  U_{\nis} A_i) \cong\\
		\colim{i \in I} \Hom_{\Ch(\ShvA)}(L_{\nis}X,  A_i) &
	\end{align*}
	so $L_{\nis}X$ is finitely presented.
\end{proof}

\begin{lem}\label{localmodelstructureexists}
	For the local projective model structure on $\Ch(\PshA)$, the left transferred model structure along $L_{\nis} : \Ch(\PshA) \rightarrow \Ch(\ShvA)$ exists.
\end{lem}
\begin{proof}
	We use \cite[Theorem 11.3.2]{hirschhorn2003model}.
	Since $\Ch(\ShvA)$ is a Grothendieck category~\cite[Proposition~3.4]{AlGG}, 
	every object is small, so $I_{\local}$ and $J_{\local}$ permit the small object argument. 
	
	Next, we need to show that $U_{\nis}$ takes relative $L_{\nis}(J_{\local})$-complexes to stalkwise quasi-isomorphisms in $\Ch(\PshA)$.
	Since $U_{\nis}$ preserves filtered colimits, it commutes with transfinite compositions. Also, 
	stalkwise quasi-isomorphisms are closed under transfinite composition. It therefore suffices to show that $U_{\nis}$ takes any pushout of a map from $L_{\nis}(J_{\local})$ to a stalkwise quasi-isomorphism.
	
	Let $f : A \rightarrow B$ be a map in $J_{\local}$, and consider a pushout of the form
	$$\xymatrix{L_{\nis}A \ar[r]^{L_{\nis}f} \ar[d] & L_{\nis}B \ar[d]\\
		X \ar[r]^g & Y } $$
	We need to show that $U_{\nis}g$ is a stalkwise quasi-isomorphism.
	Since $\Ch(\ShvA)$ is an abelian category, this pushout gives rise to a short exact sequence in $\Ch(\ShvA)$
	$$0 \rightarrow L_{\nis}A \rightarrow L_{\nis}B \oplus X \rightarrow Y \rightarrow 0.$$
	For every point $x$ of the Nisnevich site, we get a short exact sequence on stalks
	$$0 \rightarrow A_x \rightarrow B_x \oplus X_x \rightarrow Y_x \rightarrow 0$$
	in $\Ch(\Ab)$. This short exact sequence of chain complexes induces a long exact sequence on homology groups
	$$\cdots \rightarrow H_{n+1}(Y_x) \rightarrow H_n(A_x) \rightarrow H_n(B_x) \oplus H_n(X_x) \rightarrow H_n(Y_x) \rightarrow H_{n-1}(A_x) \rightarrow \cdots $$
	Since $f$ is in $J_{\local}$, it is a stalkwise quasi-isomorphism, so the map $H_n(A_x) \rightarrow H_n(B_x)$ is an isomorphism. This then implies that $H_n(X_x) \rightarrow H_n(Y_x)$ is also an isomorphism, so $g : X \rightarrow Y$ is a stalkwise quasi-isomorphism.
	
	Therefore the transferred model structure on $\Ch(\ShvA)$ exists, with 
	generating cofibrations $L_{\nis}(I_{\local})$ and generating trivial 
	cofibrations $L_{\nis}(J_{\local})$, and the adjunction $L_{\nis} : \Ch(\PshA) \leftrightarrows \Ch(\ShvA) : U_{\nis}$ is a Quillen adjunction.
\end{proof}

\begin{lem}\label{weakfingentransfer}
	Let $M$ be a model category that is weakly finitely generated with weakly generating trivial cofibrations $J^\prime_M$, and let $L: M \rightleftarrows N : R$ be an adjunction, such that the left transferred model structure along $L$ exists. Assume that $L$ preserves small objects and finitely presented objects. Then the transferred model structure on $N$ is weakly finitely generated, and $L(J^\prime_M)$ is a set of weakly generating trivial cofibrations for $N$.
\end{lem}
\begin{proof}
	Let $I_M$ denote a set of generating cofibrations and $J_M$ denote a set of generating trivial cofibrations for $M$.
	Then by definition of the transferred model structure, $L(I_M)$ is a set of generating cofibrations and $L(J_M)$ is a set of generating trivial cofibrations for $N$.
	
	Since $L$ preserves small objects and finitely presented objects, the domains and codomains from $L(I_M)$ and $L(J^\prime_M)$ are finitely presented, and the domains from $L(J_M)$ are small.
	
	Take $f : A \rightarrow B$ in $N$ with $B$ fibrant and $f$ having the right lifting property with respect to $L(J^\prime_M)$. To show the lemma we now just have to show that $f$ is a fibration in $N$. By adjunction $R(f)$ has the right lifting property with respect to $J^\prime_M$. Since $R: N \rightarrow M$ is a right Quillen functor and $B$ is fibrant in $N$ we know that $R(B)$ is fibrant in $M$. Since $J^\prime_M$ is a set of weakly generating trivial cofibrations for $M$ it now follows that $R(f)$ is a fibration in $M$. From the definition of the transferred model structure it follows that $f$ is a fibration in $N$. Therefore $L(J^\prime_M)$ is a set of weakly generating trivial cofibrations for $N$.
\end{proof}

\begin{cor}\label{transfereddweaklyfingen}
	The model category $\Ch(\ShvA)$ is weakly finitely generated, with $L_{\nis}(J^\prime_{\local})$ as a set of weakly generating trivial cofibrations.
\end{cor}
\begin{proof}
	By Lemma \ref{Lnisfinpres} we know that $L_{\nis}$ preserves finitely presented objects. It also preserves small objects, because all objects in $\Ch(\ShvA)$ are small.
	The result now follows from Lemma \ref{weakfingentransfer}.
\end{proof}

There is a symmetric monoidal 
structure on $\Ch(\ShvA)$ defined by $X \otimes Y := L_{\nis} (U_{\nis}(X) \otimes U_{\nis}(Y))$. With respect to this monoidal structure the adjunction $L_{\nis}: \Ch(\PshA) \rightleftarrows \Ch(\ShvA) : U_{\nis}$ 
is a monoidal adjunction. This means that the left adjoint $L_{\nis}$ is strong monoidal, 
while the right adjoint $U_{\nis}$ is lax monoidal.
We use the following lemma to make $\Ch(\ShvA)$ into a monoidal model category in the sense of \cite[Definition 3.1]{schwede2000algebras}.

\begin{lem}
	Let $M, N$ be closed symmetric monoidal categories,
	and let $L: M \leftrightarrows N : R$ be a monoidal adjunction.
	Let $M$ be equipped with a cofibrantly generated monoidal model structure with generating cofibrations $I$ and generating trivial cofibrations $J$. Assume that the left transferred model structure along $L: M \rightarrow N$ exists and that $L(I)$ and $L(J)$ permit the small object argument. Furthermore assume that the monoidal unit $\mathbbm{1}_M$ is cofibrant in $M$. Then the left
	transferred model structure on $N$ is a monoidal model structure and the unit $\mathbbm{1}_N$ is cofibrant.
\end{lem}

\begin{proof}
	Let $I$ be the generating cofibrations of $M$, and let $J$ be the generating trivial cofibrations of $M$.
	Then $L(I)$ is a set of generating cofibrations and $L(J)$ is a set of generating trivial cofibrations for $N$.
	Given two morphisms $f, g$, we write $f \square g$ to denote the pushout-product of $f$ and $g$.
	To verify the pushout-product axiom for the transferred model structure on $N$, it suffices by \cite[Corollary 4.2.5]{hovey2007model}  to show that $L(I) \square L(I)$ consists  out of cofibrations, and $L(J) \square L(I)$ consists out of trivial cofibrations.
	
	Since $L$ is a strong monoidal left adjoint functor, it preserves pushout products, in the sense that for all morphisms $f: A \rightarrow B$ and $g: C \rightarrow D$ in $M$ we have a commutative diagram in which the vertical maps are isomorphisms:
	$$ \xymatrix{ L(A \otimes D \underset{A \otimes C}{\coprod} B \otimes C) \ar[rr]^{L(f \square g)} \ar[d]^{\sim}  & & L(B \otimes D) \ar[d]^{\sim} \\ 
		L(A) \otimes L(D) \underset{L(A) \otimes L(C)}{\coprod} L(B) \otimes L(C) \ar[rr]^(0.65){L(f) \square L(g)} & &L(B) \otimes L(D)  }$$
	This can also be expressed by saying that $L(f \square g) \cong L(f) \square L(g)$ in the arrow category $\mathsf{Arr}(N)$.
	
	So any morphism in $L(I) \square L(I)$, respectively $L(J) \square L(I)$, is isomorphic to a morphism in $L(I \square I)$, respectively $L(J \square I)$, in the arrow category  $\mathsf{Arr}(N)$.
	Since $M$ is a monoidal model category, all morphisms from $I \square I$, respectively $J \square I$, are cofibrations, respectively trivial cofibrations.
	Since $L: M\rightarrow N$ is a left Quillen functor it preserves cofibrations and trivial cofibrations.
	Since cofibrations and trivial cofibrations are closed under isomorphisms in $\mathsf{Arr}(N)$ it follows that $L(I) \square L(I)$ consists out of cofibrations and $L(J) \square L(I)$ consists out of trivial cofibrations.
	So $N$ satisfies the pushout-product axiom.
	
	Since $\mathbbm{1}_M$ is cofibrant in $M$ and $L$ is a left Quillen functor, $L(\mathbbm{1}_M)$ is cofibrant in $N$. Since $L$ is strong monoidal $L(\mathbbm{1}_M) \cong \mathbbm{1}_N$, so $\mathbbm{1}_N$ is cofibrant in $N$.
	This in particular implies that $N$ is a monoidal model category.
\end{proof}

We will now prove some lemmas to show that $\Ch(\ShvA)$ satisfies the monoid axiom.

\begin{lem} \label{cokerfree}
	If $f \in J^\prime_{\local}$ then $\coker(f) \in \Ch(\PshA)$ is a bounded chain complex and degreewise free.
\end{lem}
\begin{proof}
	Take $f \in J^\prime_{\local}$. Then $f \in J_{\proj}$ or $f \in S \square I_{\Ch_{\geq 0}}$. If $f \in J_{\proj}$, then $$\coker(f) = \Cor(-,X) \otimes D^n\bb Z$$ for some $X \in \Smk, n \in \bb Z$, and that is clearly bounded and free.
	If $f \in S \square I_{\Ch_{\geq 0}}$, then $f=g \square h$ for some $g \in I_{\Ch_{\geq 0}}$ and some $h \in S$. Since $g$ is just a map of the form $S^n \bb Z \rightarrow D^n \bb Z$ for some $n \geq 0$, it suffices to show that $h$ has a bounded and degreewise free cokernel. Up to a shift, $h$ is either the morphism $0 \rightarrow \Cor(-,\emptyset)$ or $h$ is a morphism of the form $s_Q \rightarrow t_Q$ for some Nisnevich square $Q \in \mathcal{Q}$. The cokernel of $0 \rightarrow \Cor(-,\emptyset)$ is clearly bounded and free.
	So assume now that $h$ is of the form $s_Q \rightarrow t_Q$ for some Nisnevich square $Q \in \mathcal{Q}$, of the form
	$$\xymatrix{U^\prime \ar[r] \ar[d] & X^\prime \ar[d] \\ 
		U \ar[r]  & X} $$
	Recall from Definition \ref{localprojdef} that $s_Q$ is defined via the pushout square
	$$\xymatrix{\Cor(-,U^\prime) \ar[r] \ar[d]& C \ar[d]\\
		\Cor(-,U) \ar[r] & s_Q }$$
	where $C$ is the mapping cylinder of $\Cor(-,U^\prime) \rightarrow \Cor(-,X^\prime)$.
	By the usual construction of mapping cylinders \cite[1.5.5]{weibel1996introduction} we have in each individual degree $n$ an equality $$C_n = \Cor(-,U^\prime)_n \oplus \Cor(-,U^\prime)_{n-1} \oplus \Cor(-,X^\prime)_n$$ and the canonical map $\Cor(-,U^\prime) \rightarrow C$ is in each individual degree $n$ a coproduct inclusion.
	
	Thus the pushout defining $s_Q$ is a pushout of bounded and degreewise free complexes along a morphism which is degreewise a coproduct inclusion. This then implies that $s_Q$ is bounded and degreewise free.
	
	Next, recall that $t_Q$ is defined as the mapping cylinder of $s_Q \rightarrow \Cor(-,X)$. Thus the canonical map $h : s_Q \rightarrow t_Q$ is also a degreewise coproduct inclusion between bounded and degreewise free objects. This then implies that $\coker(h)$ is bounded and degreewise free.
	
	And then it follows that $\coker(f)$ is bounded and degreewise free.
\end{proof}

\begin{lem}\label{ftimesZinjectiveqis}
	If $f \in J^\prime_{\local}$ and
	$Z \in \Ch(\PshA)$, then $f \otimes Z$ is a local quasi-isomorphism and a monomorphism in $\Ch(\PshA)$.
	
\end{lem}
\begin{proof}
	We can calculate $f \otimes Z$ in degree $n \in \bb Z$ by
	$$(f \otimes Z)_n = \underset{i+j = n}{\bigoplus} f_i \otimes Z_j.$$
	By Lemma \ref{degreesplitmono} each $f_i$ is a split monomorphism.
	Then also every $f_i \otimes Z_j$ is a split monomorphism, so their direct sum is a split monomorphism.
	So $f \otimes Z$ is a monomorphism.
	We now just need to show that $f \otimes Z$ is a local quasi-isomorphism. Since it is already a monomorphism, we now just need to show that $\coker(f \otimes Z)$ is locally acyclic.
	Let $C := \coker(f)$.
	By Lemma \ref{cokerfree} the complex $C$ is bounded and degreewise free. Since $f$ is a local quasi-isomorphism, we know that $C$ is locally acyclic.
	Also we have an isomorphism $\coker(f \otimes Z) \cong \coker(f) \otimes Z = C \otimes Z$.
	So to prove the lemma we now just need to show the following claim:
	
	If $C \in \Ch(\PshA)$ is bounded, degreewise free and locally acyclic, then $C \otimes Z$ is locally acyclic.
	
	We will first show this claim for the case where $Z$ is concentrated in degree 0. So we assume $Z \in \PshA$.
	We claim that $C \otimes Z$ is locally acyclic.

	Take a free resolution of $Z$ in $\PshA$
	$$\dots \rightarrow F_2  \rightarrow F_1 \rightarrow F_0 \rightarrow Z \rightarrow 0. $$
	We can tensor this resolution with $C$ to get the following double complex
	$$\xymatrix{ &   \dots\ar[d] &  \dots\ar[d] & \dots \ar[d] \\
		\dots \ar[r] & F_1 \otimes C_1 \ar[r] \ar[d] & F_0 \otimes C_1 \ar[r] \ar[d] & Z \otimes C_1 \ar[d] \\
		\dots\ar[r]  & F_1 \otimes C_0  \ar[r]\ar[d]  & F_0 \otimes C_0 \ar[r]\ar[d] & Z \otimes C_0 \ar[d] \\
		&  \dots& \dots & \dots } $$
	Denote this double complex by $D_{\bullet,\bullet}$.
	
	Since $C$ is degreewise free, by Lemma \ref{projectiveimpliesflat} each $C_i$ is also flat, so each row is exact. This then means that the horizontal homology of $D_{\bullet,\bullet}$ vanishes. So we have for all $q \in \bb Z$, $$H_{\mathrm{hor},q}(D_{\bullet,\bullet}) = 0$$ in $\Ch(\PshA)$.
	
	Associated to the double complex $D$ we have a spectral sequence in $\PshA$ computing the homology of the total complex \cite{nlab:spectral_sequence}.
	$$E^2_{p,q} = H_{\mathrm{vert},p}(H_{\mathrm{hor},q}(D_{\bullet,\bullet} ) ) \implies H_{p+q}(\mathrm{Tot}(D_{\bullet,\bullet})) $$
	Since $H_{\mathrm{hor},q}(D_{\bullet,\bullet}) = 0$ it follows that $H_{p+q}(\mathrm{Tot}(D_{\bullet,\bullet})) = 0$.
	
	If this homology vanishes, then it also locally vanishes. So if $L_{\nis}(D_{\bullet,\bullet})$ denotes the sheafification of $D_{\bullet,\bullet}$, and if $H^{\nis}$ denotes Nisnevich homology sheaves in $\ShvA$, then we have for all $p.q \in \bb Z$ that $H^{\nis}_{p+q}(\mathrm{Tot}(L_{\nis}(D_{\bullet,\bullet})) = 0$.
	
	By mirroring the double complex $L_{\nis}(D_{\bullet,\bullet})$ and then using the double complex spectral sequence in the Grothendieck category $\ShvA$, we get another spectral sequence computing the same homology
	$$E^2_{p,q} = H^{\nis}_{\mathrm{hor},p}(H^{\nis}_{\mathrm{vert},q}(L_{\nis}(D_{\bullet,\bullet}) ) ) \implies H^{\nis}_{p+q}(\mathrm{Tot}(L_{\nis}(D_{\bullet,\bullet}))).$$
	Since $C$ is bounded, degreewise free and locally acyclic, and since each $F_i$ is free, we can use an argument similar to \cite[Corollary 2.3]{suslin2000bloch} to show for every $q\geq 0$ that $$H^{\nis}(L_{\nis}(F_{q} \otimes C)) = 0.$$
	This then means that the Nisnevich homology of all vertical columns of $L_{\nis}(D_{\bullet,\bullet})$ in positive degree vanishes. So for $q \neq 0$ and $p \in \bb Z$ we have
	$$ H^{\nis}_{\mathrm{vert},q}(L_{\nis}(D_{\bullet,\bullet}) )_p = H^{\nis}_p(L_{\nis}(F_{q-1} \otimes C )) =  0.$$
	Here we consider the $L_{\nis}(Z\otimes C_i)$ column of $L_{\nis}(D_{\bullet,\bullet})$ to be in degree 0.
	
	Thus the spectral sequence $E^2_{p,q} = H^{\nis}_{\mathrm{hor},p}(H^{\nis}_{\mathrm{vert},q}(L_{\nis}(D_{\bullet,\bullet})) ) $ stabilizes at the second page, and consists only of a single column whose terms are $H^{\nis}_p(L_{\nis}(Z\otimes C))$.
	Since the spectral sequence converges against $H^{\nis}_{p+q}(\mathrm{Tot}(L_{\nis}(D_{\bullet,\bullet}))) = 0$ it follows that  $H^{\nis}_p(L_{\nis}(Z\otimes C)) = 0 $ for every $p$, so the chain complex $Z \otimes C$ is locally acyclic.
	
	So we have now shown the lemma in the case where $Z$ is concentrated in degree $0$.
	Let us show the lemma in full generality. Namely, let $C$ be bounded, degreewise free and locally acyclic, and let  $Z\in \Ch(\PshA)$ be any chain complex.
	We claim that $C \otimes Z$ is locally acyclic.
	
	For every $k \in \bb Z$, let $\tau_k(Z)$ denote the following truncated chain complex
	$$\dots \rightarrow Z_{k+3} \overset{\partial_Z^{k+3}}{\rightarrow} Z_{k+2} \overset{\partial_Z^{k+2}}{\rightarrow} Z_{k+1} \rightarrow \ker(\partial_Z^{k}) \rightarrow 0,$$
	where $\ker(\partial_Z^{k})$ is in degree $k$. The chain complex $\tau_k(Z)$ is $k$-connected.
	
	For every $k \in \bb Z$ there is a canonical map $\phi_k: \tau_k(Z) \rightarrow \tau_{k-1}(Z)$ with $\phi_{k,i} = id_{Z_i}$ for all $i \geq k+1$, as shown in this diagram
	$$\xymatrix{\dots \ar[r] \ar[d]& Z_{k+2} \ar[r]^{\partial_Z^{k+2}} \ar[d]& Z_{k+1} \ar[r]^{\partial_Z^{k+1}} \ar[d]& \ker(\partial_Z^{k}) \ar[r] \ar[d]& \ar[d] 0\\
		\dots \ar[r]& Z_{k+2} \ar[r]^{\partial_Z^{k+2}} & Z_{k+1} \ar[r]^{\partial_Z^{k+1}} & Z_{k} \ar[r]^{\partial_Z^{k}}& \ker(\partial_Z^{k-1})    } $$
	In $\Ch(\PshA)$ we can consider the $\bb Z$-indexed diagram
	$$\dots \rightarrow \tau_{k+1}(Z) \rightarrow \tau_k(Z) \rightarrow \tau_{k-1}(Z) \rightarrow \cdots $$
	The colimit of this diagram is obviously $Z$.
	In particular $C \otimes Z \cong \colim{k \in \bb Z} (C \otimes \tau_k(Z))$.
	
	Since filtered colimits in $\Ch(\PshA)$ preserve local quasi-isomorphisms, we know that filtered colimits of locally acyclic objects are locally acyclic. So to show that $C \otimes Z$ is locally acyclic, we now just need to show that each $C \otimes \tau_k(Z)$ is locally acyclic. Let $k \in \bb Z$ be arbitrary.
	We have a distinguished triangle in $\Ch(\PshA)$ 
	$$\tau_{k+1}(Z)[-k] \rightarrow \tau_k(Z)[-k] \rightarrow H_k(Z) \rightarrow  \tau_{k+1}(Z)[1-k]$$
	where $H_k(Z) \in \PshA$ is regarded as a chain complex concentrated in degree $0$.
	So if we consider the following diagram in $D(\PshA)$
	$$ \dots \rightarrow \tau_{k+i}(Z)[-k] \rightarrow \dots  \rightarrow \tau_{k+1}(Z)[-k] \rightarrow \tau_k(Z)[-k] $$
	then for every $i \in \bb N$, the $i$-th morphism in the sequence has a cofiber isomorphic to $H_{k+i}(Z)[i]$.
	Also the $i$-th term in the sequence $\tau_{k+i}(Z)[-k]$ is $i$-connected.
	By Lemma \ref{projectiveimpliesflat} we know that $C$ is degreewise flat. So if we tensor the above diagram with $C$ we get a diagram
	$$ \dots \rightarrow C \otimes \tau_{k+i}(Z)[-k] \rightarrow \dots \rightarrow C \otimes \tau_{k+1}(Z)[-k] \rightarrow C \otimes \tau_k(Z)[-k]$$
	in which the $i$-th morphism has a cofiber isomorphic to $C \otimes H_{k+i}(Z)[i]$.
	From \cite[Corollary 6.1.1]{friedlander2002spectral} we get a strongly convergent spectral sequence
	$$E^2_{pq} = H^{\nis}_{p+q}(C \otimes H_{k+q}(Z)[q]) \implies H_{p+q}^{\nis}(C \otimes \tau_k(Z)[-k]).$$
	Since $H_{k+q}(Z)[q]$ is concentrated in a single degree, we know that $C \otimes H_{k+q}(Z)[q]$ is locally acyclic.
	So $H^{\nis}_{p+q}(C \otimes H_{k+q}(Z)[q])= 0$, and then the spectral sequence implies that $H_{p+q}^{\nis}(C \otimes \tau_k(Z)[-k])=0$, hence $C \otimes \tau_k(Z)[-k]$ is locally acyclic.
	Then also $C \otimes \tau_k(Z)$ is locally acyclic, and then also the colimit $C \otimes Z \cong \colim{k \in \bb Z} (C \otimes \tau_k(Z))$ is locally acyclic, which then proves the entire lemma.
\end{proof}

\begin{lem}\label{checkmonoidonJ}
	Let $M$ be a monoidal model category that is weakly finitely generated. Denote the set of weakly generating trivial cofibrations by $J^\prime$.
	
	Then the monoid axiom for $M$ can be checked on $J^\prime$.
	This means with the notations from \cite{schwede2000algebras}, that if every element of $(J^\prime \otimes M)\mathrm{-cof}_{\mathrm{reg}}$ is a weak equivalence then $M$ satisfies the monoid axiom.
\end{lem}
\begin{proof}
	Before verifying the monoid axiom we first show that every trivial cofibration with fibrant codomain lies in  $J^\prime\mathrm{-cof}$.
	
	Let $f : A \overset{\sim}{ \rightarrowtail} B$ be a trivial cofibration with fibrant codomain $B$. We claim that $f$ lies in $J^\prime\mathrm{-cof}$.
	According to the small object argument \cite[Lemma 2.1]{schwede2000algebras} we can factor $f$ as $f = qi$ with $q \in \mathrm{RLP}(J^\prime)$ and $i \in J^\prime\mathrm{-cof}_{\mathrm{reg}}$.
	$$\xymatrix{A \ar[dr]_i \ar[rr]^f & & B \\
		& Z \ar[ur]_{q} &}  $$
	Since $q$ has a fibrant codomain and $q \in \mathrm{RLP}(J^\prime)$ it follows that $q$ is a fibration.
	Then $f$ has the left lifting property against $q$ so by \cite[Lemma 1.1.9]{hovey2007model} $f$ is a retract of $i$. Since $i \in J^\prime\mathrm{-cof}_{\mathrm{reg}}$ this implies $f \in J^\prime\mathrm{-cof}$.
	
	Now we start verifying the monoid axiom. Assume every element of $(J^\prime \otimes M)\mathrm{-cof}_{\mathrm{reg}}$ is a weak equivalence.
	Let $f : A \overset{\sim}{ \rightarrowtail} B$ be any trivial cofibration, let $Z \in M$ be any object and consider an arbitrary pushout diagram of the form
	$$\xymatrix{A \otimes Z \ar[d] \ar[r]^{f \otimes Z} &\ar[d] B \otimes Z \\
		X \ar[r]^h & Y} $$
	We claim that $h$ is a weak equivalence.
	Since $M$ is weakly finitely generated, we know by \cite[Lemma 3.5]{dundas2003enriched} that transfinite compositions of weak equivalences are weak equivalences in $M$. So if we show that $h$ is a weak equivalence, then this immediately implies the monoid axiom.
	
	Denote the terminal object of $M$ by $1$. Factor the map $B \rightarrow 1$ into a trivial cofibration followed by a fibration.
	We then have a trivial cofibration $g : B \overset{\sim}{ \rightarrowtail} B^f$ with $B^f$ fibrant.
	Then both $g : B \rightarrow B^f$ and $gf : A \rightarrow B^f$ are trivial cofibrations with fibrant codomain.
	So $g$ and $gf$ both lie in $J^\prime\mathrm{-cof}$.
	Then $Z \otimes g$ and $Z \otimes gf$ lie in $Z \otimes (J^\prime\mathrm{-cof})$.
	By a simple argument using the adjunction $- \otimes Z \dashv \mathrm{Hom}(Z,-)$  one can show that $Z \otimes (J^\prime\mathrm{-cof}) \subseteq (Z \otimes J^\prime)\mathrm{-cof}$.
	So $Z \otimes g$ and $Z \otimes gf$ lie in $(Z \otimes J^\prime)\mathrm{-cof}$, and thus also in $(M \otimes J^\prime)\mathrm{-cof}$.
	
	Consider the pushout diagram
	$$\xymatrix{A \otimes Z \ar[d] \ar[r]^{f \otimes Z} &\ar[d] B \otimes Z \ar[r]^{g \otimes Z} & B^f \otimes Z \ar[d]  \\
		X \ar[r]^h & Y \ar[r]^(0.4)k & (B^f \otimes Z) \underset{B \otimes Z}{\coprod} Y  } $$
	Since $g \otimes Z$ and $gf \otimes Z$ lie in $( J^\prime \otimes M)\mathrm{-cof}$, and
	since $(J^\prime\otimes M)\mathrm{-cof}$ is stable under pushouts, it follows that $k$ and $kh$ also lie in  $(J^\prime\otimes M)\mathrm{-cof}$. 
	By \cite[Lemma 2.1]{schwede2000algebras} this means that $k$ and $kh$ are retracts of morphisms from $(J^\prime\otimes M)\mathrm{-cof}_{\mathrm{reg}}$. Since we assume that all morphisms from $(J^\prime\otimes M)\mathrm{-cof}_{\mathrm{reg}}$ are weak equivalences, and since weak equivalences are stable under retracts, it follows that $k$ and $kh$ are weak equivalences. Then by $2$-of-$3$ also $h$ is a weak equivalence.
	This then proves the monoid axiom for $M$.
\end{proof}

\begin{lem}\label{lemmaChShvAMonoidAxiom}
	$\Ch(\ShvA)$ satisfies the monoid axiom in the sense of~\cite{schwede2000algebras}.
\end{lem}
\begin{proof}
	By Lemmas \ref{checkmonoidonJ} and \ref{transfereddweaklyfingen} it suffices to check the monoid axiom on the set $L_{\nis}(J^\prime_{\local})$.
	
	Take $f: A \rightarrow B$, with $f \in L_{\nis}(J^\prime_{\local})$ and take $Z \in \Ch(\ShvA)$. We claim that $f \otimesShv Z$ is an injective quasi-isomorphism.
	Since $\ShvA$ is a Grothendieck category, we know that injective quasi-isomorphisms in $\Ch(\ShvA)$ are stable under pushouts and transfinite compositions. So if we show that $f \otimesShv Z$ is an injective quasi-isomorphism, then this proves the entire monoid axiom. 
	
	If $f \in L_{\nis}(J^\prime_{\local})$, then there exists $f^\prime : A^\prime \rightarrow B^\prime$ with $f^\prime \in J^\prime_{\local}$ and $L_{\nis}(f^\prime) = f$.
	By Lemma \ref{ftimesZinjectiveqis} we know $f^\prime \otimesPsh U_{\nis}Z$ is an injective local quasi-isomorphism in $\Ch(\PshA)$.
	Since $L_{\nis}$ is strongly monoidal we have an isomorphism of arrows
	$$L_{\nis}(f^\prime \otimesPsh U_{\nis}Z) \cong L_{\nis}(f^\prime) \otimesShv L_{\nis}U_{\nis}Z \cong f \otimesShv Z$$
	So we just need to show that $L_{\nis}(f^\prime \otimesPsh U_{\nis}Z)$ is an injective quasi-isomorphism.
	
	Since $f^\prime \otimesPsh U_{\nis}Z$ is injective, and the sheafification functor $L_{\nis}$ is exact, we know that $L_{\nis}(f^\prime \otimesPsh U_{\nis}Z)$ is injective. So we now just need to show that $L_{\nis}(f^\prime \otimesPsh U_{\nis}Z)$ is a quasi-isomorphism.
	By definition of the transferred model structure on $\Ch(\ShvA)$ we thus need to show that $U_{\nis}L_{\nis}(f^\prime \otimesPsh U_{\nis}Z)$ is a local quasi-isomorphism in $\Ch(\PshA)$.
	
	We have a commutative diagram, where $\eta$ is the unit of the adjunction $L_{\nis} \dashv U_{\nis}$:
	$$\xymatrix{U_{\nis}L_{\nis}(A^\prime \otimesPsh U_{\nis}Z) \ar[rrr]^{U_{\nis}L_{\nis}(f^\prime \otimesPsh U_{\nis}Z)} & & & U_{\nis}L_{\nis}(B^\prime \otimesPsh U_{\nis}Z) \\ \ar[u]^{\eta}
		A^\prime \otimesPsh U_{\nis}Z \ar[rrr]^{f^\prime \otimesPsh U_{\nis}Z}  & & &  \ar[u]^{\eta} B^\prime \otimesPsh U_{\nis}Z
	}  $$
	The diagram commutes by the naturality of $\eta$.
	Since $\eta$ is stalkwise an isomorphism, it is by Lemma \ref{stalkwiseqis} in particular a local quasi-isomorphism in $\Ch(\PshA)$.
	
	Since $f^\prime \otimesPsh U_{\nis}Z$ is also a local quasi-isomorphism, it follows from the $2$-of-$3$-property that $U_{\nis}L_{\nis}(f^\prime \otimesPsh U_{\nis}Z)$ is a local quasi-isomorphism. So $f \otimes Z \cong L_{\nis}(f^\prime \otimesPsh U_{\nis}Z)$ is an injective quasi-isomorphism, and this concludes the proof of the lemma.
\end{proof}

\begin{lem}\label{lemmaChShvAProper}
	$\Ch(\ShvA)$ is strongly left proper in the sense of \cite[Definition 4.6]{dundas2003enriched}.
\end{lem}
\begin{proof}
	For any Grothendieck category $\cc B$, quasi-isomorphisms in $\Ch(\cc B)$ are stable under pushouts along degreewise monomorphisms. So to show that $\Ch(\ShvA)$ is strongly left proper we just need to show that for any cofibration $f$ and any object $Z \in \Ch(\ShvA)$ the map $Z \otimes f$ is a degreewise monomorphism.
	The set $L_{\nis}(I_{\proj})$ is a set of generating cofibrations for $\Ch(\ShvA)$ so we have $f \in L_{\nis}(I_{\proj})-\cof$.
	Then $$Z \otimes f \in (Z \otimes L_{\nis}(I_{\glob}))-\cof.$$
	All morphisms from $L_{\nis}(I_{\glob})$ are degreewise split monomorphisms, so all morphisms from $Z \otimes L_{\nis}(I_{\glob})$ are degreewise split monomorphisms, and this implies that all morphisms from $(Z \otimes L_{\nis}(I_{\glob}))-\cof$ are degreewise split monomorphisms. So $Z \otimes f$ is a degreewise split monomorphism.
	Therefore $\Ch(\ShvA)$ is strongly left proper.
\end{proof}

\section{Statements of the Main theorems} \label{ThmStatement}

From now on we will additionally assume that $\Cor$ satisfies the cancellation property in the sense of Definition \ref{corproperties}.
We define a $\ShvA$-enriched category $\Sm$, by letting the objects of $\Sm$ be smooth schemes over $k$, and by defining $$\Sm(X,Y) := \inthom{\ShvA}{\Cor(-,X)_{\nis}}{\Cor(-,Y)_{\nis}}.$$

We have a $\ShvA$-enriched inclusion functor $I : \Sm \rightarrow \ShvA$ defined on objects by $I(X) := \Cor(-,X)_{\nis}$, and which is defined on morphism objects as the identity $\Sm(X,Y) = \inthom{\ShvA}{I(X)}{I(Y)}$.

Let $\mathscr{C}$ be the full enriched subcategory of $\Sm$ consisting of the objects $\Gmt{n}$ where $n \in \bb Z_{\geq 0}$.

We write $\otimesShv$ for the tensor product of $\ShvA$,
and $\otimesDay$ for the Day convolution product on $[\Sm,\ShvA]$ or $[\cc C, \ShvA]$, as defined in \cite{day1970closed}:
$$(F\otimesDay G)(c) = \overset{(a,b) \in \Sm \otimes \Sm}{\int} \Sm(a \times b, c) \otimesShv F(a) \otimesShv G(b).  $$

The Grothendieck category of enriched functors $[\Sm, \ShvA]$ is tensored and cotensored over $\ShvA$ by $\otimesShv$. Given any enriched functor $F : \Sm \rightarrow \ShvA$ and $X \in \ShvA$ we can form an enriched functor $F \otimesShv X$, given by $$F \otimesShv X (U) := F(U) \otimesShv X.$$
If $X$ is representable by a scheme $U$, so that $X = \Cor(-,U)_{\nis}$, then we write $F \otimesShv U$ for $F \otimesShv X$.

The monoidal structure on $\ShvA$ induces a monoidal structure on $\Sm$ via the following easy lemma.

\begin{lem}
	Let $\mathscr{V}$ be a symmetric monoidal closed category. Let $\mathscr{C}$ be a full $\mathscr{V}$-subcategory of $\mathscr{V}$, such that $\mathbbm{1}_{\mathscr{V}}$ is isomorphic to an object of $\mathscr{C}$, and for every $X, Y \in \mathscr{C}$ the monoidal product $X \otimes Y$ is isomorphic to an object of $\mathscr{C}$.
	Then $\mathscr{C}$ can be made into a symmetric monoidal $\mathscr{V}$-category such that the inclusion functor $\mathscr{C} \rightarrow \mathscr{V}$ is strong monoidal.
\end{lem}
\begin{proof}
	Let $\overline{\cc C}$ be the full $\cc V$-subcategory of $\cc V$ on all those objects which have the property of being isomorphic to some object of $\cc C$. Then $\mathbbm{1} \in \overline{\cc C}$, and for all $X, Y \in \overline{\cc C}$ we have $X \otimes Y \in \overline{\cc C}$.
	So the functor $\otimes : \cc V \times \cc V \rightarrow \cc V$ restricts to a functor $\otimes : \overline{\cc C} \times \overline{\cc C} \rightarrow \overline{\cc C}$.
	For all $X, Y, Z \in \overline{\cc C}$ we have coherence isomorphisms $$\ell_X : \mathbbm{1} \otimes X \overset{\sim}{\rightarrow} X$$
	$$\rho_X : X \otimes \mathbbm{1}  \overset{\sim}{\rightarrow} X$$
	$$\phi_{X,Y} : X \otimes Y  \overset{\sim}{\rightarrow} Y \otimes X$$
	$$\alpha_{X,Y,Z} : (X \otimes Y) \otimes Z  \overset{\sim}{\rightarrow} X \otimes (Y \otimes Z) $$
	in $\mathscr{V}$.
	The domains and codomains of all these coherence isomorphisms lie in $\overline{\cc C}$. Since $\overline{\cc C}$ is a full subcategory of $\mathscr{V}$, all these coherence isomorphisms lie in $\overline{\cc C}$. Obviously these coherences isomorphisms in $\overline{\cc C}$ still make exactly the same diagrams commute as in $\mathscr{V}$. So $\overline{\cc C}$ is a symmetric monoidal $\mathscr{V}$-category, and the inclusion $\overline{\cc C} \rightarrow \cc V$ is a strict monoidal $\cc V$-functor.
	
	We have an inclusion $\cc V$-functor $\cc C \rightarrow \overline{\cc C}$. This functor is essentially surjective, and it is the identity on morphism objects. This then implies that $\cc C \rightarrow \overline{\cc C}$ is an equivalence in the $2$-category $\cc V-CAT$, and we then get an induced symmetric monoidal $\cc V$-category structure on $\cc C$.
\end{proof}

\begin{cor}
	$\Sm$ and $\cc C$ are symmetric monoidal $\ShvA$-categories.
\end{cor}
\begin{proof}
	The unit of $\ShvA$ is isomorphic to $\Cor(-,pt)_{\nis}$.
	We claim that for all $X, Y \in \Sm$ we have an isomorphism $$\Cor(-,X)_{\nis} \otimesShv \Cor(-,Y)_{\nis} \cong \Cor(-,X \times Y)_{\nis}.$$
	This isomorphism is constructed as follows.
	The sheafification functor $(-)_{\nis}: \PshA\rightarrow \ShvA$ is strongly monoidal, so if $\otimesPsh$ denotes the presheaf tensor product, then we have a natural isomorphism
	$\Cor(-,X)_{\nis} \otimesShv \Cor(-,Y)_{\nis} \cong (\Cor(-,X) \otimesPsh \Cor(-,Y))_{\nis}.$
	The presheaf tensor product $\otimesPsh$ is a Day convolution with respect to the monoidal structure on $\Cor$. The monoidal structure on $\Cor$ is given on objects by the cartesian product on $\Smk$.
	By general properties of Day convolution we have an isomorphism of presheaves
	$\Cor(-,X) \otimesPsh \Cor(-,Y) \cong \Cor(-,X \times Y)$
	and thus an isomorphism of sheaves $\Cor(-,X)_{\nis} \otimesShv \Cor(-,Y)_{\nis} \cong \Cor(-,X \times Y)_{\nis}.$
	The previous lemma now implies that $\Sm$ is a symmetric monoidal $\ShvA$-category.
	Since $\Cor(-,pt)_{\nis} = \Cor(-,\Gmt{0})_{\nis}$ and $\Cor(-,\Gmt{n})_{\nis} \otimesShv \Cor(-,\Gmt{m})_{\nis} \cong \Cor(-,\Gmt{n+m})_{\nis}$ it also follows that $\cc C$ is a symmetric monoidal $\ShvA$-category.
\end{proof}

Since $\ShvA$ is a closed symmetric monoidal Grothendieck category, and $\Sm$ is a monoidal $\ShvA$-category, we can apply \cite[Theorem 5.5]{garkusha2019derived} to get a weakly finitely generated monoidal model structure on $\Ch([\Sm,\ShvA])$, where the weak equivalences are the pointwise quasi-isomorphisms and the fibrations are the pointwise fibrations.
We will say that an object $F \in \Ch([\Sm,\ShvA])$ is \textit{locally fibrant} if it is fibrant in this model category.
The homotopy category of this model category is the derived category $D([\Sm, \ShvA])$ of the Grothendieck category $[\Sm, \ShvA]$.

We write $\otimesL$ for the derived tensor product on $D([\Sm, \ShvA])$. Since the model structure on $\Ch([\Sm,\ShvA])$ is monoidal by \cite[Theorem 5.5]{garkusha2019derived}, we can compute this derived tensor product by using cofibrant replacements in $\Ch([\Sm,\ShvA])$.
Also note that every representable functor $\Sm(X,-) : \Sm \rightarrow \ShvA$ is cofibrant in $\Ch([\Sm,\ShvA])$, because it is isomorphic to the cofibrant object $\Sm(X,-)\otimesShv pt$.
We similarly have a weakly finitely generated monoidal model structure on $\Ch([\cc C, \ShvA])$, whose homotopy category is $D([\cc C, \ShvA])$. 

We now define two families of morphisms in the enriched functor category $[\cc C, \ShvA]$.
The first family of morphisms we call $\mathbb{A}^1_1$, and it consists of the morphisms $$\cc C(\Gmt{n},-) \otimesShv \bb A^1 \rightarrow \cc C(\Gmt{n},-) \otimesShv pt$$ induced by the projection map $\bb A^1 \rightarrow pt$ for every $n \in \bb Z_{\geq 0}$.

The second family of morphisms, denoted by $\tau$, consists for every $n \in \bb N$ of the morphism $$\tau_n : [\Gmn{n+1},\embed{-}]\otimesShv \Gmn{1} \rightarrow [\Gmn{n},I(-)]$$ where for every $U \in \Smk$ the map $[\Gmn{n+1},\embed{U}]\otimesShv \Gmn{n+1} \overset{\tau_n}{\rightarrow} I(U)$ in $\ShvA$ is given by the counit of the adjunction $- \otimesShv \Gmn{n+1} \dashv [\Gmn{n+1},-]$.
We also sometimes write $\Sm(\Gmn{n+1},-)$ or $\cc C(\Gmn{n+1},-)$ for $[\Gmn{n+1},\embed{-}]$, even though 
$\Gmn{n+1}$ is not in $\Sm$ or $\cc C$ strictly speaking.

The domains and codomains of all these morphisms are compact in the derived category $D([\cc C, \ShvA])$ according to \cite[Theorem 6.2]{garkusha2019derived}.

Let $\sim_{\mathscr{C}}$ be the union of both of these classes of morphisms \label{simCdef} $$\sim_{\mathscr{C}} = \mathbb{A}^1_1 + \tau$$ considered as a class of morphisms in $[\mathscr{C}, \ShvA]$.

\begin{dfn}\label{strongSlocal}
	Let $\cc B$ be any small $\ShvA$-enriched category.
	
	We can consider $\Ch([\cc B, \ShvA])$ to be a $\Ch(\ShvA)$-enriched category, and denote the morphism objects by $\homShv(A,B) \in \Ch(\ShvA)$. These morphism objects are defined on $Z \in \Smk$ by
	$$\homShv(A,B)(Z) := \mathrm{map}^{\mathrm{Ch(Ab)}}(A \otimesShv Z, B) \in \Ch(\Ab)$$
	where $\mathrm{map}^{\mathrm{Ch(Ab)}}$ refers to morphism objects of the $\Ch(\Ab)$-enriched category $\Ch([\cc B, \ShvA])$. 
	Given an object $F \in \Ch([\cc B, \ShvA])$ and a class of morphisms $S$ in $\Ch([\cc B, \ShvA])$, we say that $F$ is \textit{enriched $S$-local} if for every $f : A \rightarrow B$ in $S$ we have a quasi-isomorphism of complexes of sheaves
	$$\homShv(B,F) \rightarrow \homShv(A,F) $$
	in $\Ch(\ShvA)$.
	Furthermore say that $F \in \Ch([\cc B, \ShvA])$ is \textit{strictly $S$-local} if its pointwise locally fibrant replacement $F^f$ in $\Ch([\cc B, \ShvA])$ is enriched $S$-local.
\end{dfn}

\begin{lem} \label{localobjectcomparison}
	Let $\cc B$ be a small monoidal $\ShvA$-enriched category, and $S$ a set of morphisms in $\Ch([\cc B, \ShvA])$.
	Define a new set of morphisms $$\widehat{S} :=\{ (f \otimesShv Z)[n] \mid n \in \bb Z, Z \in \Smk , f \in S \} $$
	in $D([\cc B, \ShvA])$.
	
	Let $F \in \Ch([\cc B, \ShvA])$ be locally fibrant, and assume that all domains and codomains from $S$ are cofibrant.in the local model structure.
	Then $F$ is strictly $S$-local in the sense of Definition \ref{strongSlocal} if and only if $F$ is $\widehat{S}$-local in $D([\cc B, \ShvA])$ in the usual sense, i.e. if and only if for all $g: C \rightarrow D, g \in \widehat{S}$ we have an isomorphism of abelian groups $$g^*: \mathrm{Hom}_{D([\cc B, \ShvA])}(D,F) \rightarrow \mathrm{Hom}_{D([\cc B, \ShvA])}(C,F).$$
\end{lem}
\begin{proof}
	Suppose $F$ is strictly $S$-local. Then for every $f : A \rightarrow B$, $f \in S$ we have a quasi-isomorphism of complexes of sheaves
	$$f^*: \homShv(B,F) \rightarrow \homShv(A,F) $$
	in $\Ch(\ShvA)$.
	
	We claim  that $\homShv(B,F)$ is locally fibrant.
	In fact if we have a local trivial cofibration $h : X \rightarrow Y$, then a diagram
	$$\xymatrix{X \ar[d]_h \ar[r] & \homShv(B,F) \ar[d] \\ Y \ar[r] \ar@{..>}[ur]  & 0} $$
	has a lift, by adjunction if and only if
	$$\xymatrix{B \otimesShv X \ar[d]_{B \otimesShv h} \ar[r] & F \ar[d] \\ B \otimesShv Y \ar[r] \ar@{..>}[ur]  & 0} $$
	has a lift.
	But since $B$ is cofibrant, then $B \otimesShv h$ is still a trivial cofibration. Since $F$ is locally fibrant the map $F \rightarrow 0$ is a local fibration, so the lift exists.
	Therefore $\homShv(B,F)$ and similarly $\homShv(A,F)$ are 
	locally fibrant. We see that the quasi-isomorphism 
	$$f^*: \homShv(B,F) \rightarrow \homShv(A,F)$$ is sectionwise a quasi-isomorphism.
	
	This means that for every $n \in \bb Z$ we have an isomorphism of homology presheaves
	$$H_n(\homShv(B,F))\rightarrow H_n(\homShv(A,F)).$$
	Therefore for every $Z \in \Smk$ one has 
	$$H_n(\homShv(B,F))(Z) \cong \mathrm{Hom}_{D([\cc B, \ShvA])}((B \otimesShv Z)[-n],F).$$
	It follows that $F$ is $\widehat{S}$-local in $D([\cc B, \ShvA])$.
	
	Conversely, assume that $F$ is  $\widehat{S}$-local in $D([\cc B, \ShvA])$.
	Then for every $f: A \rightarrow B$ in $S$ the map
	$$f^*: \homShv(B,F) \rightarrow \homShv(A,F) $$ is a sectionwise quasi-isomorphism, because for every $n \in \bb Z$ and $Z \in \Smk$ the map
	$$H_n(f)(Z) : H_n(\homShv(B,F))(Z) \rightarrow H_n(\homShv(A,F))(Z)  $$
	is isomorphic to the map
	\footnotesize
	$$(f \otimesShv Z)[-n]^* : \mathrm{Hom}_{D([\cc B, \ShvA])}((B \otimesShv Z)[-n],F) \rightarrow \mathrm{Hom}_{D([\cc B, \ShvA])}((A \otimesShv Z)[-n],F)  $$
	\normalsize
	and since $(f \otimesShv Z)[-n] \in \widehat{S}$ and $F$ is $\widehat{S}$-local this map is an isomorphism.
	So $F$ is strictly $S$-local if and only if $F$ is $\widehat{S}$-local in $D([\cc B, \ShvA])$.
\end{proof}

We can localize the compactly generated triangulated category $D([\mathscr{C}, \ShvA])$ with respect to the family of morphisms between compact objects $\widehat{\sim_{\cc C}}$. 
\begin{defs}\label{simcdef}
	We write $D([\mathscr{C}, \ShvA])/ \sim_{\mathscr{C}}$ for the localized compactly generated triangulated category.
	Furthermore we write $DM_{\Cor}[\cc C]$ for the full triangulated subcategory of $D([\cc C, \ShvA])$ consisting of the strictly  $\sim_{\mathscr{C}}$-local objects.
	
	It follows from Lemma \ref{localobjectcomparison} that the category $D([\mathscr{C}, \ShvA])/ \sim_{\mathscr{C}}$ is equivalent to $DM_{\Cor}[\cc C]$.
\end{defs}

\begin{dfn}\label{defcancellation}
	An enriched functor $F : \cc C \rightarrow \Ch(\ShvA)$ or $F : \Sm \rightarrow \Ch(\ShvA)$ is said to \textit{satisfy cancellation}, if for every $ n \geq 0$ the canonical map $F(\Gmn{n}) \rightarrow [\Gmn{1},F(\Gmn{n+1})]$ is a local quasi-isomorphism.
\end{dfn}

Note that an enriched functor $F$ satisfies cancellation if and only if it is enriched $\tau$-local.

\begin{defs}
	Let $F \in \Ch([\cc C, \ShvA])$. We say that $F$ is \textit{$\sim_{\mathscr{C}}$-fibrant} if it is pointwise locally fibrant in $\Ch([\cc C, \ShvA])$ and strictly  $\sim_{\mathscr{C}}$-local.
\end{defs}
Note that $F$ is strictly $\sim_{\mathscr{C}}$-local if and only if it is strictly $\bb A^1_1$-local and satisfies cancellation. 

Our first theorem is that there is a canonical equivalence of compactly generated triangulated categories
$$D([\mathscr{C}, \ShvA])/  \sim_{\mathscr{C}} \overset{\sim}{\rightarrow} DM_{\Cor}.$$
The equivalence is constructed as follows.
For an enriched functor $F : \mathscr{C} \rightarrow \Ch(\ShvA)$ and $k \in \mathbb{N}$ define $$F(\Gmn{k}) := F(\Gmt{k}) / \underset{i=0}{\overset{k+1}{\sum}} \im(F(\iota_{i,k})).$$
There is an isomorphism of categories $\Ch([\cc C, \ShvA]) \cong [\cc C, \Ch(\ShvA)]$ by \cite[Theorem 5.4]{garkusha2019derived}. For this reason we will often implicitly pass back and forth between those categories without mentioning it.

Let $\SpGm(\ShvA)$ be the category of $\Gmn{1}$-spectra in $\ShvA$.
Define $$ev_{\Gm} : \Ch([\mathscr{C}, \ShvA]) \rightarrow \SpGm(\Ch(\ShvA))$$ by taking $F \in \Ch([\mathscr{C}, \ShvA])$
(regarding it as an enriched functor $F : \cc C \rightarrow \Ch(\ShvA)$) to the $\Gmn{1}$-spectrum $(F(\Gmn{n}))_{n\in\mathbb{N}}$. 
We construct the structure maps $$F(\Gmn{k}) \otimesShv \Gmn{1} \rightarrow F(\Gmn{k+1})$$ by applying the tensor-hom adjunction to
$$\Gmn{1} \rightarrow [\Gmn{n},\Gmn{n+1}] \rightarrow [F(\Gmn{n}), F(\Gmn{n+1})].$$
This functor sends quasi-isomorphisms in $\Ch([\mathscr{C}, \ShvA])$ to stable motivic equivalences in $\SpGm(\Ch(\ShvA))$, so it induces a functor $ev_{\Gm} :  D([\mathscr{C}, \ShvA]) \rightarrow DM_{\Cor}$.
This functor can then be restricted to the full triangulated subcategory $DM_{\Cor}[\cc C] \subseteq D([\mathscr{C}, \ShvA])$.
We are now in a position to formulate the following theorem.

\begin{thm} \label{cthm}
	The functor
	$$ev_{\Gm} : DM_{\Cor}[\cc C] \rightarrow DM_{\Cor}$$
	is an equivalence of compactly generated triangulated categories.
	In particular there is an equivalence
	$$D([\mathscr{C}, \ShvA])/\sim_{\cc{C}} \overset{\sim}{\rightarrow} DM_{\Cor}.$$
\end{thm}
The proof of this theorem is given in Section \ref{FirstStep}.
To state our next result we now define some additional classes of morphisms in $D([\Sm,\ShvA])$.
Firstly, in $\Ch([\Sm,\ShvA])$ let $\bb A^1_1$ denote the class of morphisms
$$\Sm(U,-) \otimesShv \bb A^1 \rightarrow \Sm(U,-) $$
for $U \in \Sm$,
and let $\tau$ denote the class of morphisms
$$ \tau_n : [\Gmn{n+1},I(-)] \otimesShv \Gmn{1} \rightarrow [\Gmn{n},-]$$
just like in $\Ch([\cc C, \ShvA])$.
By $\mathbb{A}^1_2$ we mean the family consisting for every $Y \in \Smk$ of the morphism $$\Sm(Y,-) \rightarrow \Sm(Y \times \mathbb{A}^1,-). $$ 
The family of morphisms $Nis$ is defined as follows.
For every elementary Nisnevich square $$\xymatrix{ U^\prime \ar[r]_{\beta} \ar[d]^{\alpha} & X^\prime \ar[d]_{\gamma}\\
	U \ar[r]^{\delta} & X }$$ 
in $\Smk$, we have a square  in $\Ch([\Sm,\ShvA])$ $$ \xymatrix{\Sm(U^\prime,-) & \ar[l]^{\beta^*} \Sm(X^\prime,-) \\
	\Sm(U,-) \ar[u]_{\alpha^*}& \ar[l]_{\delta^*} \Sm(X,-) \ar[u]^{\gamma^*}}$$
It induces a map of chain complexes $p :  \hocofib(\gamma^*) \rightarrow \hocofib(\alpha^*)$, where $\hocofib$ refers to the naive mapping cone chain complex.
The family $Nis$ consists of all the morphisms $p$ for every elementary Nisnevich square.
Denote by $\sim$ the union of all the four morphism sets defined above. Namely, \label{simdef}
$$\sim := \mathbb{A}^1_1 + \tau + \mathbb{A}^1_2 + Nis.$$

\begin{dfn} \label{nisdef}
	A functor $F \in \Ch([\Sm,\ShvA])$ is said to \textit{satisfy Nisnevich excision} if it sends elementary Nisnevich squares in	$\Smk$ to homotopy cartesian squares in $\Ch(\ShvA)$.
	
	Note that we consider here covariant Nisnevich excision in the $\Sm$-variable, rather than contravariant Nisnevich excision in the $\Cor$-variable.
\end{dfn}

\begin{lem}\label{NisLemma}
	Let $F \in \Ch([\Sm,\ShvA])$ be a functor.
	Then $F$ satisfies Nisnevich excision if and only if $F$ is enriched $Nis$-local.
\end{lem}
\begin{proof}	
	By the $\Ch(\ShvA)$-enriched Yoneda lemma there is a natural isomorphism in $\Ch(\ShvA)$
	$$F(X) \cong \homShv(\Sm(X,-),F).$$ 
	So $$\xymatrix{ F(U^\prime) \ar[r]_{F(\beta)} \ar[d]^{F(\alpha)} & F(X^\prime) \ar[d]_{F(\gamma)}\\
		F(U) \ar[r]^{F(\delta)} & F(X) }$$ is homotopy cartesian if and only if  $$\xymatrix{ \homShv(\Sm(U^{\prime},-),F) \ar[r]_{\beta^{**}} \ar[d]^{\alpha^{**}} & \homShv(\Sm(X^{\prime},-),F) \ar[d]_{\gamma^{**}}\\
		\homShv(\Sm(U,-),F) \ar[r]^{\delta^{**}} & \homShv(\Sm(X,-),F) }$$ is homotopy cartesian. This is the case if and only if $\mathrm{hocofib}(\alpha^{**}) \rightarrow \mathrm{hocofib}(\gamma^*)$ is a local quasi-isomorphism. The latter 
	holds if and only if the induced morphism $$p^* : \homShv(\mathrm{hocofib}(\alpha^*),F) \rightarrow \homShv(\mathrm{hocofib}(\gamma^*),F)$$ is a local quasi-isomorphism, which means that $F$ is enriched $Nis$-local.
\end{proof}

\begin{defs}\label{defsimfibrant}
	Let $F \in \Ch([\Sm, \ShvA])$. We say that $F$ is \textit{$\sim$-fibrant} if it is pointwise locally fibrant in $\Ch([\Sm, \ShvA])$ and strictly $\sim$-local.
\end{defs}

\begin{dfn}
	Let $DM_{\Cor}[\Sm]$ be the full subcategory of $D([\Sm,\ShvA])$ of those complexes which satisfy the following properties:
	\begin{enumerate}
		\item For every $U \in \Sm$, the complex of sheaves $F(U)$ has $\bb A^1$-invariant cohomology sheaves.
		\item $F$ satisfies cancellation.
		\item $F$ is covariantly $\bb A^1$-invariant, in the sense that $F(U \times \bb A^1) \rightarrow F(U)$ is a local quasi-isomorphism.
		\item $F$ satisfies Nisnevich excision.
	\end{enumerate} 
	These properties are similar to the axioms (2)-(5) for special motivic $\Gamma$-spaces defined in \cite{garkusha2019framed}
	and axioms for framed spectral functors in the sense of~\cite[Section~6]{garkusha2018triangulated}.
	
\end{dfn}	

\begin{prop} \label{DMSmcomparison}
	The category $DM_{\Cor}[\Sm]$ is equal to the full subcategory of $D([\Sm,\ShvA])$ of those complexes $F$ which are strictly $\sim$-local.
	In particular, the inclusion from $DM_{\Cor}[\Sm]$ to $D([\Sm,\ShvA])$ induces an equivalence of triangulated categories $$DM_{\Cor}[\Sm] \overset{\sim}{\rightarrow} D([\Sm,\ShvA])/\sim.$$
\end{prop}
\begin{proof}
	The proposition follows from the following four claims:
	\begin{enumerate}
		\item A functor $F$ is strictly $\bb A^1_1$-local if and only if for every $U \in \Smk$, the complex $F(U)$ has $\bb A^1$-invariant cohomology sheaves.
		\item A strictly $\bb A^1_1$-local functor $F$ satisfies cancellation if and only if it is strictly $\tau$-local.
		\item A functor $F$ is covariantly $\bb A^1$-invariant if and only if it is strictly $\bb A^1_2$-local.
		\item A functor $F$ satisfies Nisnevich excision if and only if it is strictly $Nis$-local.
	\end{enumerate}
	Here are the proofs for those claims.
	\begin{enumerate}
		\item $F$ is strictly $\bb A^1_1$-invariant if and only if for every $U \in \Smk$ the canonical map
		$$F^f(U) \rightarrow F^f(U)(\bb A^1 \times -)$$
		is a local quasi-isomorphism in $\Ch(\ShvA)$.
		Since $F^f(U)$ and $F^f(U)(\bb A^1 \times -)$ are locally fibrant in $\Ch(\ShvA)$, it follows that the above map is a local quasi-isomorphism if and only if it is a sectionwise quasi-isomorphism in $\Ch(\PshA)$. This is the case if and only if $F^f$ has $\bb A^1$-invariant cohomology presheaves in the sense that for each $n \in \bb Z$ the map
		$$H_n(F^f(U)) \rightarrow H_n(F^f(U) (\bb A^1\times -) ) = H_n(F^f(U))(\bb A^1\times -) $$ is an isomorphism in $\PshA$.
		This means that $F^f(U)$ is motivically fibrant, which is the case if and only if $F(U)$ is $\bb A^1$-local.
		By \cite[Theorem 6.2.7]{morel2005stable} this is the case if and only if $F(U)$ has $\bb A^1$-invariant cohomology sheaves.
		\item The Yoneda lemma implies that a functor $F$ satisfies cancellation if and only if it is enriched $\tau$-local. We now claim that a strictly $\bb A^1_1$-local functor $F$ is enriched $\tau$-local if and only if it is strictly $\tau$-local.
		Let $F$ be a strictly $\bb A^1_1$-local functor, and let $F^f$ be its pointwise local fibrant replacement.
		For every $U \in \Smk$ and $n \in \bb Z$, consider the following diagram in $\ShvA$
		$$ \xymatrix{H^{\nis}_n(\inthom{\Ch(\ShvA)}{\Gmn{1}}{F(U)}) \ar[r] \ar[d]&  \inthom{\ShvA}{\Gmn{1}}{H^{\nis}_n(F(U))}  \ar[d]
			\\
			H^{\nis}_n(\inthom{\Ch(\ShvA)}{\Gmn{1}}{F^f(U)})
			\ar[r] & \inthom{\ShvA}{\Gmn{1}}{H^{\nis}_n(F^f(U))}   } $$

		Since $F(U)$ and $F^f(U)$ have $\bb A^1$-invariant cohomology sheaves, it follows from \cite[Lemma 4.3.11]{morel2003introduction} that the two horizontal maps in the diagram are isomorphisms.
		Since the canonical map $F(U) \rightarrow F^f(U)$ is a local quasi-isomorphism, the map $H^{\nis}_n(F(U)) \rightarrow H^{\nis}_n(F^f(U))$ is an isomorphism in $\ShvA$, so the right vertical map in the above diagram is also an isomorphism.
		This implies the left vertical map in the diagram
		$$H^{\nis}_n(\inthom{\Ch(\ShvA)}{\Gmn{1}}{F(U)}) \rightarrow H^{\nis}_n(\inthom{\Ch(\ShvA)}{\Gmn{1}}{F^f(U)})$$
		is an isomorphism in $\ShvA$. Hence
		$$\inthom{\Ch(\ShvA)}{\Gmn{1}}{F(U)} \rightarrow \inthom{\Ch(\ShvA)}{\Gmn{1}}{F^f(U)}$$
		is a local quasi-isomorphism in $\Ch(\ShvA)$.
		
		Now consider the diagram in $\Ch(\ShvA)$.
		$$\xymatrix{ F(\Gmn{n}) \ar[r]\ar[d]& \inthom{\Ch(\ShvA)}{\Gmn{1}}{F(\Gmn{n+1})})  \ar[d] \\F^f(\Gmn{n}) \ar[r]& \inthom{\Ch(\ShvA)}{\Gmn{1}}{F^f(\Gmn{n+1})}) } $$
		The two vertical maps are local quasi-isomorphisms.
		
		$F$ is enriched $\tau$-local if and only if the upper horizontal map is a local quasi-isomorphism. This is the case if and only if the lower horizontal map is a quasi-isomorphism, and that is true if and only if $F$ is strictly $\tau$-local.
		\item From the Yoneda lemma it follows that a functor $F$ is covariantly $\bb A^1$-invariant if and only if it is enriched $\bb A^1_2$-local.
		And every functor $F$ is enriched $\bb A^1_2$-local if and only if it is strictly $\bb A^1_2$-local, because the relation $\bb A^1_2$ only affects the covariant $\Sm$-variable and is thus not affected by pointwise local fibrant replacement.
		More precisely, consider the following diagram, in which the vertical maps are local quasi-isomorphisms:
		$$\xymatrix{ F(X \times \bb A^1) \ar[d]_{\sim} \ar[r] & F(X) \ar[d]^{\sim} \\
			F^f(X \times \bb A^1) \ar[r] & F^f(X) }. $$
		$F$ is enriched $\bb A^1_2$-local if and only if the upper morphism is a local quasi-isomorphism, which is the case if and only if the lower morphism is a quasi-isomorphism, which is the case if and only if $F^f$ is enriched $\bb A^1_2$-local, which means that $F$ is strictly $\bb A^1_2$-local.
		\item
		By Lemma \ref{NisLemma} a functor $F$ satisfies Nisnevich excision if and only if it is enriched $Nis$-local. Just like for $\bb A^1_2$, since the relation $Nis$ only affects the covariant argument, it is not affected by pointwise local fibrant replacement, so that a functor $F$ is enriched $Nis$-local if and only if it is strictly $Nis$-local.
	\end{enumerate}
	This completes the proof.
\end{proof}

Next, the evaluation functor $$ev_{\Gm}: D([\Sm,\ShvA])/\sim \rightarrow DM_{\Cor}$$is defined as follows. We send $F \in D([\Sm,\ShvA])/\sim$ to $ev_{\Gm}(F^\prime)$, where the functor $ev_{\Gm} : D([\Sm,\ShvA]) \rightarrow DM_{\Cor}$ is the evaluation functor defined just like the one in Theorem \ref{cthm}, and $F^\prime$ is a $\sim$-fibrant replacement of $F$ in $\Ch([\Sm,\ShvA])$.

When $ev_{\Gm}$ is restricted to the subcategory $DM_{\Cor}[\Sm]$, it is the naive $\Gm$-evaluation functor $$ev_{\Gm}: DM_{\Cor}[\Sm] \rightarrow DM_{\Cor} $$ that sends $F$ to the $\Gm$-spectrum $(F(\Gmn{k}))_{k \geq 0}$.

For any pre-additive category $\mathcal{B}$ we denote by $\mathcal{B}[1/p]$ the pre-additive category where all hom-sets get tensored with $\bb Z[1/p]$. Explicitly, for $x, y \in \mathcal{B}$ we define
$$\mathcal{B}[1/p](x,y) := \mathcal{B}(x,y) \otimes \bb Z[1/p]. $$

Another main result of this thesis is as follows.
\begin{thm} \label{mainthm}
	Let $p$ be the exponential characteristic of $k$. After inverting $p$ the functor $ev_{\Gm}$ is an equivalence of compactly generated triangulated categories	$$ev_{\Gm}: (D([\Sm,\ShvA])/\sim)[1/p] \overset{\sim}{\rightarrow} DM_{\Cor}[1/p].$$
	In particular the naive $\Gm$-evaluation functor $$ev_{\Gm}: DM_{\Cor}[\Sm][1/p] \rightarrow DM_{\Cor}[1/p]$$ is an equivalence of compactly generated triangulated categories.
\end{thm}
The proof of this theorem is given at the end of Section \ref{SecondStep}.

\section{Proof of Theorem \ref{cthm}} \label{FirstStep}

In this section we prove Theorem \ref{cthm}.

We will sometimes write $\mathscr{C}(\Gmn{k},-)$ for $[\Gmn{k},I(-)] = \inthom{\Ch(\ShvA)}{\Gmn{k}}{I(-)}$.
\begin{lem} \label{binomial}
	In $\ShvA$ we have an isomorphism
	$$I(\Gmt{k}) \cong \underset{i=0}{\overset{k}{\bigoplus}} \binom{k}{i} \Gmn{i}$$
	where $\binom{k}{i}$ is the binomial coefficient, and $\binom{k}{i} \Gmn{i} := \underset{j=1}{\overset{\binom{k}{i}}{\bigoplus}} \Gmn{i}$.
	
	In particular we have an isomorphism in $\Ch([\cc C, \ShvA])$ 
	$$\mathscr{C}(\Gmt{k},-) \cong \underset{i=0}{\overset{k}{\bigoplus}} \binom{k}{i}\mathscr{C}(\Gmn{k},-).$$ 
	
\end{lem}
\begin{proof}
	First note that $\Gmn{k} \otimes \Gmn{1} \cong \Gmn{k+1}$, so $\Gmn{k} \cong (\Gmn{1})^{\otimes k}$.
	Also since the map $pt \overset{\iota_{1,1}}{\rightarrow} \Gmt{1}$ splits, the splitting lemma for abelian categories implies $I(\Gmt{1})\cong \Gmn{1} \oplus I(pt)$. The binomial theorem, applied to the semi-ring of isomorphism classes of the symmetric monoidal closed category $\ShvA$, then yields an isomorphism $$I(\Gmt{k}) \cong (\Gmn{1} \oplus pt)^{\otimes k} \cong \underset{i=0}{\overset{k}{\bigoplus}} \binom{k}{i} (\Gmn{1})^{\otimes i} \otimes pt^{\otimes k-i} \cong \underset{i=0}{\overset{k}{\bigoplus}} \binom{k}{i} \Gmn{i}$$
	as required.
\end{proof}

\begin{defs} \label{motiveDef}
	\begin{enumerate}
		\item 
		We define the \textit{Suslin complex functor} $$C_* : \Ch(\ShvA) \rightarrow \Ch(\ShvA)$$ by sending $F_\bullet \in \Ch(\ShvA)$ and $U \in \Smk$ to $$C_*(F_\bullet)(U) := \mathrm{Tot}(F_\bullet(\Delta_k^\bullet \times U)) \in \Ch(\Ab).$$
		Here $\mathrm{Tot}$ is the total complex functor and $\Delta^n_k = \mathrm{Spec}(k[t_0,\dots,t_n]/(t_0+\dots+t_n-1))$ is the algebraic simplex.
		\item
		For $X \in \Smk$ we define the \textit{$\Cor$-motive of $X$} to be $$M_{\Cor}(X) := C_*(I(X)) = C_*(\Cor(-,X)_{\nis}) $$ in $\Ch(\ShvA)$.
		\item
		The enriched functor $\cc M_{\Cor}(X) : \cc C \rightarrow \Ch(\ShvA)$ defined by $$\cc M_{\Cor}(X)(U):= M_{\Cor}(X \times U)$$ will be called the \textit{enriched $\Cor$-motive of $X$}.
		\item
		For $X \in \Smk$ we define its \textit{$\Gmn{1}$-suspension spectrum} $\Sigma^\infty_{\Gm} X_+ \in DM_{\Cor}$, by defining it in weight $n$ as $$ (\Sigma^\infty_{\Gm} X_+)(n) := \Gmn{n} \otimesShv I(X)$$ and equipping it with the obvious structure maps.
	\end{enumerate}
\end{defs}

If $F : \cc C \rightarrow \Ch(\ShvA)$ is an enriched functor, then we define $C_*F : \cc C \rightarrow \Ch(\ShvA)$ by $(C_*F)(U):=C_*(F(U))$.
The endofunctor $C_* : \Ch([\cc C,\ShvA]) \rightarrow \Ch([\cc C,\ShvA])$ preserves pointwise local quasi-isomorphisms, because $\Cor$ satisfies the strict $V$-property. Thus $C_*$ induces an endofunctor on the derived category $$C_* : D([\cc C,\ShvA]) \rightarrow D([\cc C,\ShvA]).$$
For $X \in \Smk$ we have the zero inclusion map $X \rightarrow \bb A^1_X$.
Let $\bb A^1_X/X \in \Ch(\PshA)$ denote the cokernel of the induced morphism $$\Cor(-,X) \rightarrow \Cor(-,\bb A^1_X).$$
Then $\bb A^1_X/X$ is cofibrant in $\Ch(\PshA)$ because it is a direct summand of the cofibrant object $\Cor(-,\bb A^1_X)$.
We write $(\bb A^1_X/X)_{\nis} \in \Ch(\ShvA)$ for the sheafification of $\bb A^1_X/X$.
Let $T_{\bb A^1_1} = \langle \cc C(U,-) \otimesShv \bb (\bb A^1_X/X)_{\nis} \mid U \in \cc C, X \in \Smk \rangle$ be the full triangulated subcategory of $D([\cc C,\ShvA])$ that is compactly generated by $\cc C(U,-) \otimesShv(\bb A^1_X/X)_{\nis}$.

\begin{lem} \label{kerC*}
	In $D([\cc C,\ShvA])$ we have that
	$\ker(C_*) = T_{\bb A^1_1}.$
\end{lem}
\begin{proof}
	Consider a generator $\cc C(U,-) \otimesShv (\bb A^1_X/X)_{\nis}$ of $T_{\bb A^1_1}$. We claim that it is in $\ker(C_*)$. 
	For this we need to show for every $V \in \cc C$ that $C_*(\cc C(U,V) \otimesShv \bb (\bb A^1_X/X)_{\nis})$ is locally quasi-isomorphic to $0$.
	Take a free resolution of $\cc C(U,V)$ in $\Ch(\PshA)$:
	$$\dots \rightarrow F_1 \rightarrow F_0 \rightarrow \cc C(U,V) \rightarrow 0$$
	
	The presheaf $\bb A^1_X/X$ is projective because it is a direct summand of $\Cor(-,\bb A^1_X)$, and hence it is also flat by Lemma \ref{projectiveimpliesflat}. 
	Thus the following sequence is exact
	$$\dots \rightarrow F_1 \otimesPsh \bb A^1_X/X \rightarrow F_0 \otimesPsh  \bb A^1_X/X \rightarrow \cc C(U,V) \otimesPsh \bb A^1_X/X  \rightarrow 0.$$
	It then also follows that the sequence is exact in $\Ch(\PshA)$ after applying $C_*$
	$$\dots\rightarrow  C_*( F_1 \otimesPsh \bb A^1_X/X) \rightarrow  C_*(F_0 \otimesPsh  \bb A^1_X/X) \rightarrow  C_*(\cc C(U,V) \otimesPsh \bb A^1_X/X)  \rightarrow 0.$$
	
	Since each individual entry of this sequence is a chain complex, we can regard it as a double complex. Let $D_{\bullet, \bullet}$ be the double complex 
	$$D_{p,q} := \begin{cases}
		C_*( F_{p-1} \otimesPsh \bb A^1_X/X)_q & p >0\\
		C_*( \cc C(U,V) \otimesPsh \bb A^1_X/X)_q& p=0\\
		0 & p < 0
	\end{cases} $$
	Then all horizontal homology groups of $D_{\bullet, \bullet}$ are zero. The double complex spectral sequence 
	$$E^2_{p,q} = H_{\mathrm{vert},p}(H_{\mathrm{hor},q}(D_{\bullet,\bullet} ) ) \implies H_{p+q}(\mathrm{Tot}(D_{\bullet,\bullet})) $$
	implies that $H_n(\mathrm{Tot}(D_{\bullet, \bullet})) = 0$.
	
	One can now check that $C_*( \bb A^1_X/X)$ is locally quasi-isomorphic to $0$ similarly to \cite[Proposition 1.11(1)]{suslin2000bloch}.
	It follows that every $C_*( F_q \otimesPsh \bb A^1_X/X)$ is locally quasi-isomorphic to $0$, because the $F_q$ are free and for all $Y \in \cc C$ we have $\Cor(-,Y) \otimesPsh \bb A^1_X/X \cong \bb A^1_{Y \times X}/Y \times X$. 
	
	By mirroring the double complex $D_{\bullet, \bullet}$, the double complex spectral sequence for sheaves and 
	the fact that $H_n(\mathrm{Tot}(D_{\bullet, \bullet})) = 0$ imply that $C_*( \cc C(U,V) \otimesPsh \bb A^1_X/X)$ is locally quasi-isomorphic to $0$.
	We argue here similarly to the proof of Lemma \ref{ftimesZinjectiveqis}.
	Then $C_*(( \cc C(U,V) \otimesPsh \bb A^1_X/X)_{\nis}) \cong C_*( \cc C(U,V) \otimesShv \bb (\bb A^1_X/X)_{\nis})$ is locally quasi-isomorphic to $0$.
	So $\cc C(U,-) \otimesShv \bb (\bb A^1_X/X)_{\nis}$ is in $\ker(C_*)$, as claimed. 
	
	Since $\ker(C_*)$ is a full triangulated subcategory and $T_{\bb A^1_1}$ is compactly
	generated by the $\cc C(U,-) \otimesShv \bb (\bb A^1_X/X)_{\nis}$ it follows that $T_{\bb A^1_1} \subseteq \ker(C_*)$.
	
	Now show the other inclusion.
	Let $X \in \ker(C_*)$.
	Using \cite[Section 5.6]{krause2008localization} and \cite[Proposition 4.9.1]{krause2008localization} we can construct a triangle in $D([\cc C,\ShvA])$
	$$Y \rightarrow X \rightarrow LX$$
	with $Y \in T_{\bb A^1_1}$ and $LX$ orthogonal to $T_{\bb A^1_1}$.
	Apply $C_*$ to the triangle to get
	$$C_*Y \rightarrow C_*X \rightarrow C_*LX.$$
	Since $X, Y \in \ker(C_*)$, we see that $C_*X = C_*Y = 0$, hence
	$C_*LX = 0$.
	
	Since $LX$ is orthogonal to $T_{\bb A^1_1}$, we can deduce that $LX$ is strictly $\bb A^1_1$-local, so that for all $U \in \cc C$ we have a quasi-isomorphism $LX(U)(\bb A^1 \times -) \rightarrow LX(U)$ in $\Ch(\ShvA)$.
	From this property it follows that the canonical map $LX(U) \rightarrow C_*LX(U)$ is a quasi-isomorphism in $\Ch(\ShvA)$.
	Since $C_*LX = 0$ this implies $LX = 0$ in $D([\cc C,\ShvA])$.
	But if $LX = 0$, then the map $Y \rightarrow X$ is an isomorphism in $D([\cc C,\ShvA])$ and then $X \in T_{\bb A^1_1}$.
	So $T_{\bb A^1_1}= \ker(C_*)$.
\end{proof}

Let $D([\cc C, \ShvA])/T_{\bb A^1_1}$ denote the quotient of $D([\cc C, \ShvA])$ by the triangulated subcategory $T_{\bb A^1_1}$.
By Lemma \ref{localobjectcomparison} $D([\cc C, \ShvA])/T_{\bb A^1_1}$ is  equivalent to the full subcategory of $D([\cc C, \ShvA])$ 
consisting of strictly $\bb A^1_1$-local objects.

\begin{lem}
	Let $L : D([\cc C, \ShvA]) \rightarrow D([\cc C, \ShvA])$ be the $T_{\bb A^1_1}$-localization endofunctor, which is the composite of the quotient functor
	$D([\cc C, \ShvA]) \rightarrow D([\cc C, \ShvA])/T_{\bb A^1_1}$
	and the inclusion of $T_{\bb A^1_1}$-local objects
	$D([\cc C, \ShvA])/T_{\bb A^1_1} \rightarrow D([\cc C, \ShvA]).$
	Then the functor $L$ is naturally isomorphic to the endofunctor $C_*: D([\cc C, \ShvA]) \rightarrow D([\cc C, \ShvA])$.	
\end{lem}
\begin{proof}
	For every $X \in D([\cc C, \ShvA])$ we have an exact triangle in $D([\cc C, \ShvA])$
	$$Y \rightarrow X \rightarrow LX$$
	with $Y \in \ker(L) = T_{\bb A^1_1}$.
	We can apply $C_*$ to this triangle to get another triangle in $D([\cc C, \ShvA])$
	$$C_*Y \rightarrow C_*X \rightarrow C_*LX.$$
	Since $Y \in T_{\bb A^1_1}$ and by Lemma \ref{kerC*} $T_{\bb A^1_1} = \ker(C_*)$ we know that $C_*Y = 0$ in $D([\cc C, \ShvA])$. So we get an isomorphism $$C_*X \cong C_*LX$$ in $D([\cc C, \ShvA])$.
	Since the map $X \rightarrow LX$ is functorial in $X \in D([\cc C, \ShvA])$, it follows that also the map $C_*X \rightarrow C_*LX$ is functorial in $X$.
	Therefore the isomorphism $C_*X \cong C_*LX$ is functorial in $X$.
	Since $LX$ is strictly $\bb A^1_1$-invariant we have a natural quasi-isomorphism 
	$LX \cong LX(\bb A^1 \times -) $ in $\Ch(\ShvA)$.
	This then implies that for every $n \in \bb N$ we also have a natural quasi-isomorphism $LX \cong LX(\Delta_k^n \times -) .$
	It now follows from the definition of $C_*$ that we have a natural isomorphism 
	$LX \cong C_*LX $
	in $D([\cc C, \ShvA])$.
	And then we have isomorphisms
	$$C_*X \cong C_*LX  \cong LX$$
	natural in $X$, which proves the lemma.
\end{proof}

\begin{defs}\label{motivicequivdef}
	We say that a morphism $f: X \rightarrow Y$ in $\Ch(\ShvA)$ is a \textit{motivic equivalence} if and only if $f$ is an isomorphism in $DM_{\Cor}^{\eff}$.
	Note that $f$ in $\Ch(\ShvA)$ is a motivic equivalence if and only if $C_*(f)$ is a local quasi-isomorphism in $\Ch(\ShvA)$.
	
	Similarly, we say that a morphism $f: \mathcal{X} \rightarrow \mathcal{Y}$ in $\Ch([\cc C,\ShvA])$ is a \textit{motivic equivalence} if it is an isomorphism in $D([\cc C,\ShvA])/T_{\bb A^1_1}$.
\end{defs}

From the previous lemma we can deduce:
\begin{cor} \label{moteqC}
	A morphism $f: \mathcal{X} \rightarrow \mathcal{Y}$ in $\Ch([\cc C,\ShvA])$ is a motivic equivalence if and only if $C_*(f)$ is a pointwise local quasi-isomorphism in $\Ch([\cc C, \ShvA])$.
\end{cor}

\begin{lem}
	For every $X \in \Smk$ the canonical map
	$I(X \times -) \rightarrow \cc M_{\Cor}(X)$
	is a motivic equivalence in $\Ch([\cc C,\ShvA])$.
	This means it is an isomorphism in $D([\cc C, \ShvA])/T_{\bb A^1_1}$. In particular it is also an isomorphism in $D([\cc C, \ShvA])/\sim_{\cc{C}}.$
\end{lem}
\begin{proof}
	By Corollary \ref{moteqC} we just need to show for every $U \in \Smk$ that $C_*(I(X \times U)) \rightarrow C_*(M_{\Cor}(X \times U))$ 
	is a local quasi-isomorphism in $\Ch(\ShvA)$.
	From the definition of $M_{\Cor}$ we know that $M_{\Cor}(X \times U) = C_*(I(X \times U))$.
	So the above map is equal to the canonical map
	$C_*(I(X \times U)) \rightarrow C_*C_*(I(X \times U))$
	and this is clearly an isomorphism.
\end{proof}

\begin{lem}\label{enrichedmotivelocal}
	The enriched motive functor $\cc M_{\Cor}(X)$ is strictly $\bb A^1_1$-local and strictly $\tau$-local. So $\cc M_{\Cor}(X)$ is an object of $DM_{\Cor}[\cc C]$.
\end{lem}
\begin{proof}
	The strict $\bb A^1_1$-locality follows from the $\bb A^1$-invariance of $C_*(\Cor(-,X)_{\nis})$.
	The cancellation property of $\Cor$ (see Definition \ref{corproperties}) implies that $M_{\Cor}(X \times -)$ satisfies cancellation.
	Similarly to item $(2)$ of the proof of Proposition \ref{DMSmcomparison}, this implies $M_{\Cor}(X \times -)$ is strictly $\tau$-local.
\end{proof}

The previous two lemmas together imply that $\cc M_{\Cor}(X)$ is a strictly $\sim_{\cc{C}}$-local replacement of $I(X \times -)$ in $\Ch([\cc C,\ShvA])$.

\begin{lem}\label{Gmpreserveslocal}
	If $f: X \rightarrow Y$ is a local quasi-isomorphism in $\Ch(\ShvA)$, and $X, Y \in \Ch(\ShvA)$ have $\bb A^1$-invariant cohomology sheaves, then the map
	$$f_*: \inthom{\Ch(\ShvA)}{\Gmn{k}}{X} \rightarrow \inthom{\Ch(\ShvA)}{\Gmn{k}}{Y} $$
	is also a local quasi-isomorphism in $\Ch(\ShvA)$.
	In particular, the functor
	$$\inthom{\Ch(\ShvA)}{\Gmn{k}}{-} : \Ch([\cc C, \ShvA]) \rightarrow \Ch([\cc C, \ShvA])$$
	preserves pointwise local quasi-isomorphisms between strictly $\bb A^1_1$-local objects.
\end{lem}
\begin{proof}
	It follows from \cite[Lemma 4.3.11]{morel2003introduction} that for every $X$ with $\bb A^1$-invariant cohomology sheaves and for every $n \in \bb Z$, we have a natural isomorphism $$H_n^{\nis}(\inthom{\Ch(\ShvA)}{\Gmn{k}}{X}) \cong \inthom{\ShvA}{\Gmn{k}}{H_n^{\nis}(X)}$$
	in $\ShvA$.
	So if $f : X \rightarrow Y$ is a local quasi-isomorphism between objects with $\bb A^1$-invariant cohomology sheaves, then we have for every $n \in \bb Z$ a commutative diagram
	$$\xymatrix{H_n^{\nis}(\inthom{\Ch(\ShvA)}{\Gmn{k}}{X}) \ar[rr]^{H_n^{\nis}(f_*)} \ar[d]^\sim & & H_n^{\nis}(\inthom{\Ch(\ShvA)}{\Gmn{k}}{Y}) \ar[d]^\sim \\ \inthom{\ShvA}{\Gmn{k}}{H_n^{\nis}(X)} \ar[rr]^{H_n^{\nis}(f)_*} & &  \inthom{\ShvA}{\Gmn{k}}{H_n^{\nis}(Y)}  }$$
	in $\ShvA$. Since $f$ is a local quasi-isomorphism, the lower horizontal map is an isomorphism. Therefore the upper horizontal map is an isomorphism.
	Then $f_*: \inthom{\Ch(\ShvA)}{\Gmn{k}}{X} \rightarrow \inthom{\Ch(\ShvA)}{\Gmn{k}}{Y} $
	is also a local quasi-isomorphism in $\Ch(\ShvA)$.
\end{proof}

\begin{lem} \label{GmWeq}
	The functors
	$$\inthom{\Ch(\ShvA)}{\Gmn{k}}{-} : \Ch(\ShvA) \rightarrow \Ch(\ShvA)$$ and
	$$\inthom{\Ch(\ShvA)}{\Gmt{k}}{-} : \Ch(\ShvA) \rightarrow \Ch(\ShvA)$$ preserve motivic equivalences.
\end{lem}
\begin{proof} 
	Let $f : A \rightarrow B$ be a motivic equivalence in $\Ch(\ShvA)$. Consider the diagram
	$$\xymatrix{  C_*\inthom{\Ch(\ShvA)}{\Gmn{k}}{A} \ar[r]^{C_*(f_*)} \ar[d]^{\sim}& C_*\inthom{\Ch(\ShvA)}{\Gmn{k}}{B} \ar[d]^{\sim} \\
		\inthom{\Ch(\ShvA)}{\Gmn{k}}{C_*A} \ar[r]^{(C_*f)_*} & \inthom{\Ch(\ShvA)}{\Gmn{k}}{C_*B}  }. $$
	The vertical maps are isomorphisms.
	Since $f$ is a motivic equivalence we know that $C_*(f)$ is a local equivalence. Since $C_*A$ and $C_*B$ have $\bb A^1$-invariant cohomology sheaves it follows by Lemma \ref{Gmpreserveslocal} that
	the bottom horizontal map
	$(C_*f)_*$ is a local equivalence.
	This implies that the upper horizontal map $C_*(f_*)$ is a local equivalence, and hence
	$f_* : \inthom{\Ch(\ShvA)}{\Gmn{k}}{A} \rightarrow \inthom{\Ch(\ShvA)}{\Gmn{k}}{B}$
	is a motivic equivalence.
	The second claim for $\Gmt{k}$ can be deduced from the claim for $\Gmn{k}$ by using Lemma \ref{binomial}.
\end{proof}

Let $D([\cc C, \ShvA])/\tau$ denote the localization of $D([\cc C, \ShvA])$ at the family of morphisms $\widehat{\tau}$. By Lemma \ref{localobjectcomparison} it is equivalent to the full subcategory of $D([\cc C, \ShvA])$ of those functors which are strictly $\tau$-local.

We will now prove some lemmas about $D([\mathscr{C}, \ShvA])/\tau$, which show that $\cc C(\Gmn{k},-)$ is a strongly dualizable object.

The model category $\Ch([\cc C, \ShvA])$ can be Bousfield localized along the family of morphisms $\hat{\tau}$, where just like 
Lemma \ref{localobjectcomparison}, the family $\hat{\tau}$  is defined as
$$\hat{\tau} := \{ (f \otimes Z)[n] | f \in \tau, Z \in \Smk, n \in \bb Z\}.$$ 
The homotopy category of this Bousfield localization is the derived category $D([\mathscr{C}, \ShvA])/\tau$.

\begin{lem}\label{taumonoidal}
	The left Bousfield localization of $\Ch([\cc C, \ShvA])$ along $\widehat{\tau}$ is a monoidal model category.
	In particular, the category $D([\cc C, \ShvA])/\tau$ is closed symmetric monoidal and its tensor product  $\otimesL$ coincides with the tensor product in $D([\cc C, \ShvA])$.
\end{lem}

\begin{proof}
	We apply \cite[Theorem B]{white2014monoidal}. Cofibrant objects in $\Ch([\cc C, \ShvA])$ are flat, so the theorem is applicable.
	The domains and codomains of the generating cofibrations of $\Ch([\cc C,\ShvA])$ are of the form $\cc C(\Gmt{k},-) \otimesShv X$ for $k \in \bb N, X \in \Smk$.
	For $n \in \bb N$, let $\tau_n$ be the morphism $$\cc C(\Gmn{n+1},-) \otimesShv \Gmn{1} \overset{\tau_n}{\rightarrow} C(\Gmn{n},-).$$
	We need to show that for every $n,m,k \in \bb N, X, Z \in \Smk$ that $$(\tau_n \otimesShv Z)[m] \otimesL (\cc C(\Gmt{k},-) \otimesShv X)$$ is a $\widehat{\tau}$-local equivalence in $D([\cc C, \ShvA])$.
	
	Since all involved objects are cofibrant we have
	$$(\tau_n \otimesShv Z)[m] \otimesL (\cc C(\Gmt{k},-) \otimesShv X) \cong (\tau_n \otimesShv Z)[m] \otimesDay (\cc C(\Gmt{k},-) \otimesShv X).$$
	Also we have
	$$(\tau_n \otimesShv Z)[m] \otimesDay (\cc C(\Gmt{k},-) \otimesShv X) \cong (\tau_n \otimesDay (\cc C(\Gmt{k},-) \otimesShv (X \times Z)))[m]$$
	so it suffices to show  for every $n ,k \in \bb N, X\in \Smk$ that every shift of $\tau_n \otimesDay (\cc C(\Gmt{k},-) \otimesShv X)$
	is a $\widehat{\tau}$-local equivalence.
	This morphism is then equal to the composite
	\begin{multline*}
		(\cc C(\Gmn{n+1},-) \otimesShv \Gmn{1}) \otimesDay (\cc C(\Gmt{k},-) \otimesShv X) \cong \cc C(\Gmn{n+1} \times \Gmt{k},-) \otimesShv \Gmn{1} \otimesShv X \rightarrow\\ \rightarrow \cc C(\Gmn{n}\times\Gmt{k},-) \otimesShv X .
	\end{multline*}
	To show that it is a $\widehat{\tau}$-local equivalence, let  $F \in \Ch([\cc C, \ShvA])$ be a $\tau$-fibrant object, i.e. a functor that is locally fibrant and satisfies cancellation in the sense that
	$F(\Gmn{n}) \rightarrow F(\Gmn{n+1})(\Gmn{1}\times -)$
	is a local quasi-isomorphism. Since both sides are locally fibrant, it is also a sectionwise quasi-isomorphism.
	
	We now just need to show for all $m \in \bb Z$ that
	\begin{multline*}
		\mathrm{Hom}_{D([\cc C, \ShvA])}(\cc C(\Gmn{n}\times\Gmt{k},-) \otimesShv X , F[m])  \rightarrow\\ \rightarrow
		\mathrm{Hom}_{D([\cc C, \ShvA])}(\cc C(\Gmn{n+1} \times \Gmt{k},-) \otimesShv \Gmn{1} \otimesShv X , F[m])	
	\end{multline*}
	is an isomorphism in $\Ab$.
	
	Since $F[m]$ is locally fibrant and $\cc C(\Gmn{n}\times\Gmt{k},-) \otimesShv X$ and $\cc C(\Gmn{n+1} \times \Gmt{k},-) \otimesShv \Gmn{1} \otimesShv X $ are cofibrant, this is isomorphic to the arrow
	\begin{multline*}
		\mathrm{Hom}_{K([\cc C, \ShvA])}(\cc C(\Gmn{n}\times\Gmt{k},-) \otimesShv X , F[m])  \rightarrow\\ \rightarrow
		\mathrm{Hom}_{K([\cc C, \ShvA])}(\cc C(\Gmn{n+1} \times \Gmt{k},-) \otimesShv \Gmn{1} \otimesShv X , F[m]).	
	\end{multline*}
	And this is isomorphic to the following arrow between homology groups
	$$H_m(F(\Gmn{n}\times\Gmt{k})(X)) \rightarrow H_m(F( \Gmn{n+1} \times\Gmt{k})(X \times \Gmn{1})).$$
	So we just need to show that the following arrow is a quasi-isomorphism.
	$$F(\Gmn{n}\times\Gmt{k})(X) \rightarrow F( \Gmn{n+1} \times\Gmt{k})(X \times \Gmn{1}) $$
	Lemma \ref{binomial} implies $F(\Gmn{n} \times \Gmt{k} ) \cong \underset{i=0}{\overset{k}{\bigoplus}}\binom{k}{i} F(\Gmn{n+i})$.
	We have to show that the map
	$\underset{i=0}{\overset{k}{\bigoplus}}\binom{k}{i} F(\Gmn{n+i})(X) \rightarrow \underset{i=0}{\overset{k}{\bigoplus}}\binom{k}{i} F(\Gmn{n+1+i})(X \times \Gmn{1}) $
	is a quasi-isomorphism.
	This follows from the fact that 
	$F(\Gmn{n}) \rightarrow F(\Gmn{n+1})(\Gmn{1}\times -)$
	is a sectionwise quasi-isomorphism for any $n \in \bb Z$.
\end{proof}

\begin{lem}
	The enriched functor $\mathscr{C}(\Gmn{1},-) : \cc C \rightarrow \ShvA$ is invertible in $D([\mathscr{C}, \ShvA])/\tau$ with respect to $\otimesL$, and its inverse is $I \otimesShv \Gmn{1}$.
\end{lem}
\begin{proof}
	The enriched functor $\mathscr{C}(\Gmt{1},-)$ is cofibrant in $\Ch([\cc C, \ShvA])$, because it is representable.
	The enriched functor $\mathscr{C}(\Gmn{1},-)$ is a direct summand of $\mathscr{C}(\Gmt{1},-)$, so $\mathscr{C}(\Gmn{1},-)$ is also cofibrant.
	For every cofibrant $F \in \Ch([\cc C, \ShvA])$ we therefore have $\mathscr{C}(\Gmn{1},-) \otimesL F \cong \mathscr{C}(\Gmn{1},-) \otimesDay F. $
	Now let $F := I \otimesShv \Gmn{1}$, i.e. $F$ is the enriched functor defined by $F(X) := I(X) \otimesShv \Gmn{1}.$
	This functor $F$ is cofibrant, because it is a direct summand of  $\cc C(pt,-)\otimesShv \Gmt{1}$.
	
	We now show that there is an isomorphism
	$$\mathscr{C}(\Gmn{1},-) \otimesDay (I \otimesShv \Gmn{1}) \cong (\mathscr{C}(\Gmn{1},-) \otimesDay I) \otimesShv \Gmn{1}.$$
	It explicitly looks as follows. $\Gmn{1} \otimesShv -$ is a left adjoint, so it preserves all coends, so
	\begin{multline*}
		(\mathscr{C}(\Gmn{1},-) \otimesDay (I \otimesShv \Gmn{1}))(c) = \overset{(a,b) \in \cc C \otimes \cc C}{\int} \cc C(a \times b, c) \otimesShv  \mathscr{C}(\Gmn{1},a) \otimesShv I(b) \otimesShv \Gmn{1} \cong \\ \cong \Gmn{1} \otimesShv  \overset{(a,b) \in \cc C \otimes \cc C}{\int} \cc C(a \times b, c) \otimesShv  \mathscr{C}(\Gmn{1},a) \otimesShv I(b) \cong \Gmn{1} \otimesShv (\mathscr{C}(\Gmn{1},-) \otimesDay I)
	\end{multline*}
	
	Now $I$ is the monoidal unit of $\otimesDay$, so
	$ \mathscr{C}(\Gmn{1},-) \otimesDay F \cong \mathscr{C}(\Gmn{1},-) \otimesShv \Gmn{1}$.
	Finally, the morphism $\tau$ gives an isomorphism 
	$ \mathscr{C}(\Gmn{1},-) \otimesShv \Gmn{1} \rightarrow I$
	in the derived category  $D([\mathscr{C}, \ShvA])/\tau$.
	So we ultimately get an isomorphism 
	$\mathscr{C}(\Gmn{1},-) \otimesL F \cong I $ in $D([\mathscr{C}, \ShvA])/\tau$, which shows that $\mathscr{C}(\Gmn{1},-)$ is invertible.
\end{proof}

Since $I \otimesShv \Gmn{1}$ is invertible, we also have that $I \otimesShv \Gmn{k}$ is invertible, because due to the isomorphism $I \otimesShv \Gmn{k+1} \cong (I \otimesShv \Gmn{k}) \otimesDay (I \otimesShv \Gmn{1})$ it is a product of invertible objects.
The inverse of $I \otimesShv \Gmn{k}$ is $\mathscr{C}(\Gmn{k},-)$.
Also note that in every symmetric closed monoidal category, every $\otimes$-invertible object is strongly dualizable.
So $\mathscr{C}(\Gmn{k},-)$ is strongly dualizable in $D([\cc C, \ShvA])/\tau$.

Since finite sums of strongly dualizable objects are strongly dualizable, and since Lemma \ref{binomial} says that $\cc C(\Gmt{k},-) $ is a finite sum of $\cc C(\Gmn{i},-) $, we get the following corollary.
\begin{cor}
	For all $k \in \bb N$ the enriched functors $\cc C(\Gmt{k},-) $ and $\cc C(\Gmn{k},-)$ are strongly dualizable in $D([\mathscr{C}, \ShvA])/\tau$ with duals $I \otimesShv \Gmt{k}$ and $I \otimesShv \Gmn{k}$ respectively.
\end{cor}

The model category $\Ch([\cc C, \ShvA])$ can be Bousfield localized along the family of morphisms $\widehat{\sim_{\cc{C}}}$.
The homotopy category of this Bousfield localization is the derived category $D([\mathscr{C}, \ShvA])/\sim_{\cc{C}}$.

\begin{lem}
	The left Bousfield localization of $\Ch([\cc C, \ShvA])$ along $\widehat{\sim_{\cc{C}}}$ is a monoidal model category.
	In particular, the category $D([\cc C, \ShvA])/\sim_{\cc{C}}$ is closed symmetric monoidal and its tensor product  $\otimesL$ coincides with the tensor product in $D([\cc C, \ShvA])$.
\end{lem}
\begin{proof}
	Similarly to Lemma \ref{taumonoidal}, we apply \cite[Theorem B]{white2014monoidal}.	
	The domains and codomains of the generating cofibrations of $\Ch([\cc C,\ShvA])$ are of the form $\cc C(\Gmt{k},-) \otimesShv X$ for $k \in \bb N, X \in \Smk$.
	We need to show for $f$ in $\widehat{\sim_{\cc{C}}}$ that all $f \otimesL \cc C(\Gmt{k},-) \otimesShv X$ are $\widehat{\sim_{\cc{C}}}$-local equivalences.
	If $f \in \widehat{\tau}$, then we know this from the proof of Lemma \ref{taumonoidal}.
	So assume that $f \in \widehat{\bb A^1_1}$, so that $f$ is of the form
	$$(\cc C(U,-) \otimesShv \bb A^1) \otimes Z [n] \rightarrow \cc C(U,-) \otimes Z [n] $$
	for some $U \in \cc C$.
	Since all involved objects are cofibrant we know that $$f \otimesL \cc C(\Gmt{k},-) \otimesShv X = f \otimesDay \cc C(\Gmt{k},-) \otimesShv X.$$
	
	So $f$ is isomorphic to
	$$(\cc C(U\times \Gmt{k},-) \otimesShv \bb A^1) \otimes (X \times Z) [n] \rightarrow \cc C(U \times \Gmt{k},-) \otimes (X \times Z) [n]  $$
	and this morphism lies again in $\widehat{\bb A^1_1}$. In particular it is a $\widehat{\bb A^1_1}$-local equivalence, and therefore also a $\widehat{\sim_{\cc{C}}}$-local equivalence.
\end{proof}

\begin{lem} \label{GmtimesXcomputation}
	There is an isomorphism in $D([\mathscr{C}, \ShvA])/\sim_{\mathscr{C}}$:
	$$\mathscr{C}(\Gmn{k}, -) \otimesShv X \cong [\Gmn{k}, \cc M_{\Cor}(X)].$$
\end{lem}
\begin{proof}
	We have $\cc C(\Gmn{k},-) \otimesShv X \cong \cc C(\Gmn{k},-) \otimesL ( I\otimesShv X)$
	Since $\cc C(\Gmn{k},-)$ is strongly dual to $I\otimesShv \Gmn{k}$ with respect to $\otimesL$ in $D([\cc C, \ShvA])/\sim_{\cc{C}}$ we get that
	$$\cc C(\Gmn{k},-) \otimesL ( I\otimesShv X) \cong \inthom{D([\cc C, \ShvA])/\sim_{\cc C}}{I \otimesShv \Gmn{k}}{I\otimesShv X} .$$
	By Lemma \ref{enrichedmotivelocal} the functor $M_{\Cor}(X \times -)$ is strictly $\sim_{\cc{C}}$-local.
	Since $I \otimesShv \Gmn{k}$ is cofibrant we can therefore compute the above internal hom as
	$$\inthom{D([\cc C, \ShvA])/\sim_{\cc C}}{I \otimesShv \Gmn{k}}{I\otimesShv X} \cong \inthom{D([\cc C, \ShvA])}{I \otimesShv \Gmn{k}}{M_{\Cor}(X \times -)}.$$
	Let $M_{\Cor}(X \times -)^f$ be a pointwise local fibrant replacement of $M_{\Cor}(X \times -)$ in $\Ch([\cc C, \ShvA])$.
	Then $M_{\Cor}(X \times -)^f$ is $\sim_{\cc{C}}$-fibrant and we have an isomorphism in $\Ch([\cc C, \ShvA])$.
	$$\inthom{D([\cc C, \ShvA])}{I \otimesShv \Gmn{k}}{M_{\Cor}(X \times -)} \cong[\Gmn{k},M_{\Cor}(X \times -)^f] \cong [\Gmn{k},M_{\Cor}(X \times -)].$$
	The last isomorphism follows from the fact that due to Lemma \ref{Gmpreserveslocal} the functor $[\Gmn{k},-]$ preserves local quasi-isomorphisms between strictly $\bb A^1_1$-local objects.
\end{proof}

\begin{lem} \label{compactgeneration}
	$DM_{\Cor}[\cc C]$ is compactly generated by the set $\{[\Gmn{k}, \cc M_{\Cor}(X)] \mid k \in \mathbb{N}, X \in \Smk \}.$
\end{lem}
\begin{proof}
	Let us first show that $[\Gmn{k},\cc M_{\Cor}(X)]$ is an object of $DM_{\Cor}[\cc C]$.	
	By Lemma \ref{enrichedmotivelocal} the functor $\cc M_{\Cor}(X)$ is strictly $\bb A^1_1$-local and strictly $\tau$-local. So if $\cc M_{\Cor}(X)^f$ is a locally fibrant replacement of $\cc M_{\Cor}(X)$, then $\cc M_{\Cor}(X)^f$ is enriched $\bb A^1_1$-local and satisfies cancellation.
	Since it is enriched $\bb A^1_1$-local, for every $U \in \Smk$ the complex $M_{\Cor}(X \times U)^f$ is motivically fibrant in $\Ch(\ShvA)$. Since $\Gmn{k}$ is cofibrant in $\Ch(\ShvA)$, it follows that $[\Gmn{k},M_{\Cor}(X \times U)^f]$ is motivically fibrant in $\Ch(\ShvA)$. This then implies that $[\Gmn{k},\cc M_{\Cor}(X)^f]$ is enriched $\bb A^1_1$-local.
	Since $\cc M_{\Cor}(X)^f$ satisfies cancellation, it also follows that $[\Gmn{k},\cc M_{\Cor}(X)^f]$ satisfies cancellation.
	
	By Lemma \ref{Gmpreserveslocal} the functor $[\Gmn{k},-]$ preserves local equivalences between strictly $\bb A^1_1$-local objects. Hence it follows that $[\Gmn{k},\cc M_{\Cor}(X)^f]$ is a local fibrant replacement of $[\Gmn{k},\cc M_{\Cor}(X)]$. Thus $[\Gmn{k},\cc M_{\Cor}(X)]$ is strictly $\bb A^1_1$-local and strictly $\tau$-local. So $[\Gmn{k},\cc M_{\Cor}(X)]$ is in $DM_{\Cor}[\cc C]$.
	
	Let us now show that the objects $[\Gmn{k},\cc M_{\Cor}(X)]$ compactly generate $DM_{\Cor}[\cc C]$.
	According to \cite[Theorem 6.2]{garkusha2019derived} the category $D([\mathscr{C}, \ShvA])$ is a compactly generated triangulaged category, that is compactly generated by the set $\{\mathscr{C}(c, -) \otimesShv g_i \mid c \in \cc C, i \in I\},$ where $\{g_i \mid i \in I\}$ is a set of compact generators of 
	$D(\ShvA)$.
	
	Since $\ShvA$ is compactly generated by sheaves of the form $I(X)$ for $X \in \Smk$, we conclude that $D([\mathscr{C}, \ShvA])$, and hence also $D([\mathscr{C}, \ShvA])/\sim_{\mathscr{C}}$, are compactly generated by the set $\{\mathscr{C}(\Gmt{k}, -) \otimesShv \embed{X} \mid k \in \mathbb{N}, X \in \Smk \}$.
	By Lemma \ref{binomial} the enriched functor $\mathscr{C}(\Gmt{k}, -)$ is a direct sum of functors of the form $\mathscr{C}(\Gmn{k}, -)$. So we conclude that  
	$\{\mathscr{C}(\Gmn{k}, -) \otimesShv \embed{X} \mid k \in \mathbb{N}, X \in \Smk \}$is a set of compact generators of $D([\mathscr{C}, \ShvA])/\sim_{\mathscr{C}}$.
	
	Since $\{\mathscr{C}(\Gmn{k}, -) \otimesShv \embed{X} \mid k \in \mathbb{N}, X \in \Smk \}$ is a set of compact generators of the triangulated category $D([\mathscr{C}, \ShvA])/\sim_{\mathscr{C}}$ we now get that by Lemma \ref{GmtimesXcomputation} that $ \{[\Gmn{k}, \cc M_{\Cor}(X)] \mid k \in \mathbb{N}, X \in \Smk \}$ is a set of compact generators of $D([\mathscr{C}, \ShvA])/\sim_{\mathscr{C}}$.
	
	Now each functor $[\Gmn{k}, \cc M_{\Cor}(X)]$ is in $DM_{\Cor}[\cc C]$. 
	We remarked in Definition \ref{simcdef} that the canonical map $DM_{\Cor}[\cc C] \rightarrow D([\mathscr{C}, \ShvA])/\sim_{\mathscr{C}}$ is an equivalence. Therefore it follows that
	$\{[\Gmn{k}, \cc M_{\Cor}(X)] \mid k \in \mathbb{N}, X \in \Smk \}$ is a set of compact generators of $DM_{\Cor}[\cc C]$.
\end{proof}

\begin{lem} \label{evGmn}
	For every $k \in \bb N$ and $X \in \Smk$ the canonical map
	$$ ev_{\Gm}([\Gmn{k},\cc M_{\Cor}(X)]) \rightarrow \Omega_{\Gm}^k ev_{\Gm}(\cc M_{\Cor}(X)^f) $$
	is a levelwise local quasi-isomorphism in $\SpGm(\Ch(\ShvA))$, where $\cc M_{\Cor}(X)^f$ is a pointwise  local fibrant replacement of $\cc M_{\Cor}(X)$.
\end{lem}
\begin{proof}
	
	Let $M_{\Cor}(X \times -)^f$ be a locally fibrant replacement of $M_{\Cor}(X \times -)$.
	By Lemma \ref{enrichedmotivelocal} we know that $M_{\Cor}(X \times -)^f$ is enriched $\bb A^1_1$-local and enriched $\tau$-local.
	So $M_{\Cor}(X \times -)^f$ is pointwise $\bb A^1$-invariant and satisfies cancellation.
	Since $M_{\Cor}(X \times -)^f$ is pointwise $\bb A^1$-invariant
	it follows that $ev_{\Gm}(M_{\Cor}(X \times -)^f)$ is levelwise motivically fibrant.
	Since $M_{\Cor}(X \times -)^f$ satisfies cancellation, we see that $ev_{\Gm}(M_{\Cor}(X \times -)^f)$ is an $\Omega_{\Gm}$-spectrum.
	So $ev_{\Gm}(M_{\Cor}(X \times -)^f)$ is stably motivically fibrant in $\SpGm(\Ch(\ShvA))$, and hence $\Omega_{\Gm}^k ev_{\Gm}(M_{\Cor}(X \times -)^f)$ can be computed in weight $n$ as
	$$\Omega_{\Gm}^k ev_{\Gm}(M_{\Cor}(X \times -)^f)(n) = [\Gmn{k}, M_{\Cor}(X \times \Gmn{n})^f].$$
	But that is also the $n$-th weight of $ev_{\Gm}([\Gmn{k},M_{\Cor}(X \times -)^f])$.
	So the canonical map
	$$ ev_{\Gm}([\Gmn{k},M_{\Cor}(X \times -)]) \rightarrow \Omega_{\Gm}^k ev_{\Gm}(M_{\Cor}(X \times -)) $$
	is isomorphic to
	$$ ev_{\Gm}([\Gmn{k},M_{\Cor}(X \times -)]) \rightarrow ev_{\Gm}([\Gmn{k},M_{\Cor}(X \times -)^f]).$$
	This is a levelwise local quasi-isomorphism in $\SpGm(\Ch(\ShvA))$, because due to Lemma \ref{Gmpreserveslocal} the functor $[\Gmn{k},-]$ preserves local quasi-isomorphisms between strictly $\bb A^1_1$-local objects.
\end{proof}

To prove Theorem \ref{cthm} and show that the functor $ ev_{\Gm}: DM_{\Cor}[\cc C] \rightarrow DM_{\Cor} $ is an equivalence, we will use \cite[Lemma 4.8]{garkusha2022recollements}, which says the following:
\begin{lem} \label{garkushaLemma}
	
	Let $A$, $B$ be compactly generated triangulated categories. Let $\Sigma$ be a set of compact generators in $A$. Let $F: A \rightarrow B$ be a triangulated functor such that
	
	1. The collection $\{F(X)|X \in \Sigma\}$ is a set of compact generators in $B$
	
	2. For all $X, Y \in \Sigma$ and $n \in \mathbb{Z}$ the map
	$$F_{X,Y[n]} : \mathrm{Hom}_A(X,Y[n]) \rightarrow \mathrm{Hom}_B(F(X),F(Y)[n])$$ is an isomorphism.
	
	Then $F$ is an equivalence of triangulated categories.
\end{lem}

We are now in a position to prove the main result of this section.
\begin{proof}[Proof of Theorem \ref{cthm}]
	
	We use Lemma \ref{garkushaLemma}.
	Here $A = DM_{\Cor}[\cc C]$ and $B= DM_{\Cor}$ are in fact compactly generated triangulated categories.
	One set of compact generators of $DM_{\Cor}$ is given by $\{ \Omega_{\Gm}^kev_{\Gm}(\cc M_{\Cor}(X)^f) | k \in \mathbb{N}, X \in \Smk  \}$, where $\cc M_{\Cor}(X)^f$ is a pointwise local fibrant replacement of $\cc M_{\Cor}(X)$.
	By Lemma \ref{compactgeneration} the set $$\Sigma := \{[\Gmn{k}, \cc M_{\Cor}(X)] | k \in \mathbb{N}, X \in \Smk \}$$ is a set of compact generators of $DM_{\Cor}[\cc C]$.
	This is the set of compact generators to which we want to apply Lemma \ref{garkushaLemma}.
	We now check the two conditions of that lemma.
	
	To show the first condition we use Lemma \ref{evGmn}:
	For every $A \in \Sigma$ we have an isomorphism
	$$ev_{\Gm}(A)  = ev_{\Gm}([\Gmn{k}, \cc M_{\Cor}(X)]) \overset{\ref{evGmn}}{\cong} \Omega_{\Gm}^kev_{\Gm}(\cc M_{\Cor}(X)^f)$$
	which is one of the compact generators of $DM_{\Cor}$.
	So $$\{ev_{\Gm}(A) | A \in \Sigma \} = \{ \Omega_{\Gm}^kev_{\Gm}(\cc M_{\Cor}(X)^f) | k \in \mathbb{N}, X \in \Smk  \}$$
	which shows condition 1.
	
	Let us now check condition 2.
	Take $\mathcal{X}, \mathcal{Y} \in DM_{\Cor}[\cc C]$ and $n \in \bb Z$. We have to show that $\mathrm{Hom}_{DM_{\Cor}[\cc C]}(\mathcal{X}, \mathcal{Y}[n]) \cong  \mathrm{Hom}_{DM_{\Cor}}(ev_{\Gm}(\mathcal{X}), ev_{\Gm}(\mathcal{Y})[n]).$
	Since $\Sigma$ compactly generates
	$DM_{\Cor}[\cc C]$ it suffices to show this for the case $\mathcal{X} \in \Sigma$. So assume without loss of generality that $\mathcal{X} \in \Sigma$ is of the form $[\Gmn{k}, \cc M_{\Cor}(X)]$ for some $X \in \Smk$ and $k \in \mathbb{N}$.
	Furthermore, we may assume without loss of generality that $\mathcal{Y}$ is $\sim_{\cc C}$-fibrant. So $\mathcal{Y}$ is pointwise motivically fibrant and satisfies cancellation. Then we have with Lemma \ref{GmtimesXcomputation} that
	\begin{multline*}
		\hspace{-10pt}\mathrm{Hom}_{DM_{\Cor}[\cc C]}([\Gmn{k}, \cc M_{\Cor}(X)], \mathcal{Y}[n]) \overset{\ref{GmtimesXcomputation}}{\cong}
		\mathrm{Hom}_{D([\mathscr{C}, \ShvA]) / \sim_{\mathscr{C}}}(\mathscr{C}(\Gmn{k},-) \otimesShv X, \mathcal{Y}[n]) =\\ = \mathrm{Hom}_{D([\mathscr{C}, \ShvA]) / \sim_{\mathscr{C}}}(\mathscr{C}(\Gmn{k},-), \inthom{\ShvA}{\embed{X}}{\mathcal{Y}}[n]) = H_n(\mathcal{Y}(\Gmn{k})(X)).	
	\end{multline*}
	By Lemma \ref{evGmn} we have an isomorphism $$ ev_{\Gm}([\Gmn{k},\cc M_{\Cor}(X)]))  \overset{\ref{evGmn}}{\cong} \Omega^k_{\Gm}ev_{\Gm}(\cc M_{\Cor}(X)^f).$$
	Since $\mathcal{Y}$ satisfies cancellation, $ev_{\Gm}(\mathcal{Y})$ is an $\Omega_{\Gm}$-spectrum, hence $ev_{\Gm}(\mathcal{Y}) \cong \Omega^k_{\Gm} ev_{\Gm}(\mathcal{Y})(k).$
	Since $\mathcal{Y}$ is pointwise motivically fibrant, it follows that $ev_{\Gm}(\mathcal{Y})$ is stably motivically fibrant in $DM_{\Cor}$.
	Therefore,
	\begin{align*}
		\mathrm{Hom}_{DM_{\Cor}}(ev_{\Gm}([\Gmn{k},\cc M_{\Cor}(X)]), ev_{\Gm}(\mathcal{Y})[n]) \cong \\
		\mathrm{Hom}_{DM_{\Cor}}(\Omega^k_{\Gm}ev_{\Gm}(\cc M_{\Cor}(X)^f), ev_{\Gm}(\mathcal{Y})[n]) \cong \\
		\mathrm{Hom}_{DM_{\Cor}}(\Omega^k_{\Gm}ev_{\Gm}(\cc M_{\Cor}(X)^f), \Omega^k_{\Gm} ev_{\Gm}(\mathcal{Y})(k)[n]) \cong \\
		\mathrm{Hom}_{DM_{\Cor}}(ev_{\Gm}(\cc M_{\Cor}(X)^f),  ev_{\Gm}(\mathcal{Y})(k)[n]) \cong\\
		\mathrm{Hom}_{DM_{\Cor}}(\Sigma_{\Gm}^\infty X_+,  ev_{\Gm}(\mathcal{Y})(k)[n])
		\cong H_n(\mathcal{Y}(\Gmn{k})(X)) .
	\end{align*}
	We use here the fact that $ev_{\Gm}(\cc M_{\Cor}(X)^f)$ is a stably motivically fibrant replacement of $\Sigma_{\Gm}^\infty X_+$.
	We have verified all the conditions of Lemma \ref{garkushaLemma}.
	So $ev_{\Gm}: DM_{\Cor}[\cc C] \rightarrow DM_{\Cor}$ is an equivalence of triangulated categories.
	In particular, we have a zig-zag of equivalences
	$$D([\mathscr{C}, \ShvA]) / \sim_{\mathscr{C}} \overset{\sim}{\leftarrow} DM_{\Cor}[\cc C] \overset{\sim}{\rightarrow} DM_{\Cor} .$$ This completes the proof of Theorem \ref{cthm}.	
\end{proof}

\section{Converting motivic equivalences to local equivalences}\label{mottolocalsection}

Our next goal is to prove  Theorem \ref{mainthm}.
To this end, we prove some lemmas in this and the next section.
Theorem \ref{mottolocal} from this section will be crucial for proving Theorem \ref{RO}, and Theorem \ref{RO} will be crucial for proving Theorem \ref{mainthm}.

Let $\mathscr{M}$ be the category of motivic spaces and $f\mathscr{M}$ the category of finitely presented motivic spaces defined in \cite{dundas2003motivic}.
Then $\mathscr{M}$ has a motivic model structure, as defined in \cite[Theorem 2.12]{dundas2003motivic}. The weak equivalences in this model structure are called motivic equivalences.

Given a $\Ch(\ShvA)$-enriched functor $G: \Sm \rightarrow \Ch(\ShvA)$,
we can extend $G$ to a (non-enriched) functor $\hat{G}: f\mathscr{M} \rightarrow \Ch(\ShvA)$ in the following way.
We can apply $G$ levelwise to simplicial objects to get a functor $$G^{\Delta^{op}} : \Delta^{op}\Sm \rightarrow \Delta^{op}\Ch(\ShvA) .$$ 

For a finite pointed set $n_+=\{0,\dots,n\}$ and $U \in \Smk$ we write $n_+ \otimes U$ for the $n$-fold coproduct $\underset{i=1}{\overset{n}{\coprod}} U$.
We first extend it to $G: f\mathscr{M} \rightarrow \Delta^{op}\Ch(\ShvA)$ by
$$G(A) := \underset{(\Delta[n]\times U)_+ \rightarrow A^c}{\mathrm{colim}} G^{\Delta^{op}}(\Delta[n]_+ \otimes U),$$
where $A^c$ is a cofibrant replacement of $A$ in $f\cc M$.
We then compose it with the Dold-Kan correspondence $$DK^{-1}: \Delta^{op}\Ch(\ShvA) \rightarrow \Ch_{\geq 0}(\Ch(\ShvA))$$ and the total complex functor $$\mathrm{Tot} :  \Ch_{\geq 0}(\Ch(\ShvA)) \rightarrow \Ch(\ShvA),\quad\mathrm{Tot}(X)_n := \underset{k+l=n}{\bigoplus} X_{k,l},$$ to obtain a functor 
$$\hat{G}: f\mathscr{M} \rightarrow \Ch(\ShvA)$$
\begin{equation}\label{hatdef}
	\hat{G}(A) := \mathrm{Tot}(DK^{-1}( \underset{(\Delta[n]\times U)_+ \rightarrow A^c}{\mathrm{colim}} G^{\Delta^{op}}(\Delta[n]_+ \otimes U)  )).
\end{equation}
Note that for $U \in \Smk$ we have $\hat{G}(U_+) \cong G(U)$.

Throughout this section let $F : \Sm \rightarrow \Ch(\ShvA)$ be an enriched functor that is $\sim$-fibrant in $\Ch([\Sm,\ShvA])$.
This means that $F$ is pointwise locally fibrant, satisfies Nisnevich excision in the sense of Definition \ref{nisdef}, and for every $X \in \Smk$ there are natural quasi-isomorphisms
$F(X \times \mathbb{A}^1) \rightarrow F(X)$,
$F(\Gmn{n}) \rightarrow [\Gmn{1},F(\Gmn{n+1})]$
in $\Ch(\ShvA)$, and for every $X, U \in \Smk$ a natural quasi-isomorphism
$$F(X)(U) \rightarrow F(X)(U \times \mathbb{A}^1)$$
in $\Ch(\Ab)$.
By the above construction we can extend $F$ to a functor $\hat{F} : f\mathscr{M} \rightarrow \Ch(\ShvA)$.

In this section we prove the following theorem.

\begin{thm} \label{mottolocal}
	$\hat{F}$ sends motivic equivalences in $f\mathscr{M}$ to local quasi-isomorphisms in $\Ch(\ShvA)$.
\end{thm}

The proof is like that of \cite[Theorem 4.2]{garkusha2019framed} and requires several lemmas.

\begin{lem} \label{Fcoprod}
	Let $H : \Sm \rightarrow \ShvA$ be a $\ShvA$-enriched functor.
	Then $H(\emptyset) \cong 0$
	and for all $U, V \in \Sm$
	$H(U \coprod V) \cong H(U) \oplus H(V)$
	in $\ShvA$.
	In particular, if $G : \Sm \rightarrow \Ch(\ShvA)$ is a $\Ch(\ShvA)$-enriched functor we have $G(\emptyset) \cong 0$ and $G(U \coprod V) \cong G(U) \oplus G(V)$ in $\Ch(\ShvA)$. 
\end{lem}
\begin{proof}
	By the $\ShvA$-enriched co-Yoneda lemma we can write $H$ as the following co-end: for $U \in \Sm$ we have
	$$H(U) \cong \overset{X \in \Sm}{\int}H(X) \otimes \Sm(X,U) \cong \overset{X \in \Sm}{\int}H(X) \otimes
	\Cor(-,U)_{\nis}(X\times -).$$
	By Definition \ref{tak} Axiom (3),  we have $\Cor(-,\emptyset)_{\nis} = 0$ and for all $U, V \in \Smk$, $$\Cor(-, U \coprod V)_{\nis} \cong \Cor(-,U)_{\nis} \oplus \Cor(-,V)_{\nis} .$$
	This implies that
	$$H(\emptyset) \cong \overset{X \in \Sm}{\int}H(X) \otimes \Cor(-,\emptyset)_{\nis}(X\times -) = \overset{X \in \Sm}{\int}H(X) \otimes 0 = 0 $$
	and for all $U, V \in \Smk$,
	\begin{align*}
		H(U \coprod V) \cong \overset{X \in \Sm}{\int}H(X) \otimes \Cor(-,U \coprod V)_{\nis}(X\times -) \cong \\ \overset{X \in \Sm}{\int}H(X) \otimes (\Cor(-,U)_{\nis}(X\times -) \oplus \Cor(-,V)_{\nis}(X\times -) )\cong \\ (\overset{X \in \Sm}{\int}H(X) \otimes \Cor(-,U)_{\nis}(X\times -)) \oplus(\overset{X \in \Sm}{\int}H(X) \otimes \Cor(-,V)_{\nis}(X\times -)) \cong\\ \cong H(U) \oplus H(V)	
	\end{align*}
	as required.
\end{proof}
\begin{cor} \label{gammacorollary}
	Let $G : \Sm \rightarrow \Ch(\ShvA)$ be a $\Ch(\ShvA)$-enriched functor. Then for every $n \in \bb N, U \in \Smk$ the canonical map
	$$G(n_+ \otimes U) = G(\underset{i=1}{\overset{n}{\coprod}} 1_+ \otimes U) \rightarrow \underset{i=1}{\overset{n}{\bigoplus}} G(1_+ \otimes U) = \underset{i=1}{\overset{n}{\bigoplus}} G(U) $$ is an isomorphism.
\end{cor}

Recall that $\Delta^{op}\Ab$ is monoidal with respect to the degreewise tensor product, and $\Ch_{\geq 0}(\Ab)$ is monoidal with respect to the usual tensor product of chain complexes.
\begin{lem}\label{lemmaDoldKanTensor}
	The Dold-Kan equivalence
	$DK^{-1} : \Delta^{op}\Ab \rightarrow \Ch_{\geq 0}(\Ab)$
	preserves tensor products up to chain homotopy equivalence in the following sense.
	There are maps
	$$\nabla_{A,B} : DK^{-1}(A) \otimes DK^{-1}(B) \rightarrow DK^{-1}(A \otimes B) $$
	$$\Delta_{A,B} : DK^{-1}(A \otimes B) \rightarrow DK^{-1}(A) \otimes DK^{-1}(B) $$
	natural in $A, B$,
	such that $\Delta_{A,B} \circ \nabla_{A,B} = id_{DK^{-1}(A) \otimes DK^{-1}(B)}$, and there is a chain homotopy $\nabla_{A,B} \circ \Delta_{A,B} \sim id_{DK^{-1}(A \otimes B)}$. 
	This chain homotopy is natural in the following sense: for all maps $f : A \rightarrow A^\prime$, $g : B \rightarrow B^\prime$ the chain homotopy between the maps $DK^{-1}(f \otimes g) \circ \nabla_{A,B} \circ \Delta_{A,B} \sim DK^{-1}(f \otimes g)$ encoded by the diagram
	$$\xymatrix{DK^{-1}(A \otimes B) \ar[r]^(.4){\Delta}  \ar[d]_{DK^{-1}(f  \otimes g)} & DK^{-1}(A) \otimes DK^{-1}(B) \ar[r]^(.6){\nabla} & DK^{-1}(A \otimes B) \ar[d]^{DK^{-1}(f  \otimes g)}\\
		DK^{-1}(A^\prime \otimes B^\prime) \ar[r]^(.4){\Delta}  \ar@/_1.5pc/[rr]^{\simeq}_{id_{DK^{-1}(A^\prime \otimes B^\prime)}  }  & DK^{-1}(A^\prime) \otimes DK^{-1}(B^\prime) \ar[r]^(.6){\nabla} & DK^{-1}(A^\prime \otimes B^\prime) }$$
	is equal to the chain homotopy between the maps $DK^{-1}(f \otimes g) \circ \nabla_{A,B} \circ \Delta_{A,B} \sim DK^{-1}(f \otimes g)$ encoded by the diagram 
	\footnotesize
	$$\xymatrix{DK^{-1}(A \otimes B) \ar[r]^(.4){\Delta}  \ar@/_1.5pc/[rr]^{\simeq}_{id_{DK^{-1}(A \otimes B)} }  & DK^{-1}(A) \otimes DK^{-1}(B) \ar[r]^(.6){\nabla} & DK^{-1}(A \otimes B) \ar[rr]^{DK^{-1}(f \otimes g)} & & DK^{-1}(A^\prime \otimes B^\prime) }$$
	\normalsize
\end{lem}
\begin{proof}
	Everything except for the naturality of the chain homotopy follows from \cite{nlab:eilenberg-zilber/alexander-whitney_deformation_retraction}. The functor $DK^{-1}$ is the normalized Moore complex, the map $\Delta_{A,B}$ is the Alexander-Whitney map and $\nabla_{A,B}$ is the Eilenberg-Zilber map.
	In \cite[page 7]{gonzalez1999combinatorial} one can find explicit formulas for both of these maps, and one can also find an explicit formula for the chain homotopy $\nabla_{A,B} \otimes \Delta_{A,B} \sim id_{DK^{-1}(A \otimes B)}$, which is called the Shih operator in that paper.
	Using that explicit formula one can easily verify the naturality of the chain homotopy.
\end{proof}

Given a simplicial set $K \in \Delta^{op}\Set$ we can form the free simplicial abelian group $\bb Z^{(K)} \in \Delta^{op}\Ab$ and then apply the Dold-Kan equivalence $DK^{-1} : \Delta^{op}\Ab \rightarrow \Ch_{\geq 0}(\Ab)$ to get a chain complex which we will denote by $\bb Z[K]$:
\begin{equation} \label{ZKdefinition}
	\bb Z[K] := DK^{-1}(\bb Z^{(K)}) \in \Ch(\Ab).
\end{equation}
The chain complex $\bb Z[K]$ is degreewise free.
For example, with this notation $\bb Z[S^n]$ is the chain complex that is $\bb Z$ concentrated in homological degree $n$.

\begin{lem} \label{FKUeqKFU}
	Let $G : \Sm \rightarrow \Ch(\ShvA)$ be a $\Ch(\ShvA)$-enriched functor. For every finite simplicial set $K$ and every $A \in f\cc M$ we have a chain homotopy equivalence
	$$\hat{G}(K_+ \wedge A) \overset{\sim}{\rightarrow} \bb Z[K] \otimes \hat{G}(A)  $$
	in $\Ch(\ShvA)$ which is natural in $K$ and $A$. The chain homotopies here are also also natural in $K$ and $A$, just like the chain homotopy from Lemma \ref{lemmaDoldKanTensor}.
\end{lem}
\begin{proof}
	Since $\hat{G}(A)$ depends only on the cofibrant replacement $A^c$ of $A$, it suffices to show the claim for $A^c$.
	We can write $A^c$ as a filtered colimit of simplicial schemes
	$A^c = \colim{i \in I} X_i$
	for some $X_i \in \Delta^{op}\Smk$, and some filtered diagram $I$.
	Then also $K_+ \wedge A^c$ is cofibrant and we have $K_+ \wedge A^c = \colim{i \in I} (K_+ \wedge X_i)$.
	Let $G^{\Delta^{op}} : \Delta^{op}\Sm \rightarrow \Delta^{op}\ShvA$ be the functor that applies $G$ in each simplicial degree.
	It follows from Corollary \ref{gammacorollary} that for each $i \in I$ we have an isomorphism
	$$G^{\Delta^{op}}(K_+ \otimes X_i) \overset{\sim}{\rightarrow} \bb Z^{(K)} \otimes G^{\Delta^{op}}(X_i) $$
	in $\Delta^{op}\Ch(\ShvA)$, where $\bb Z^{(K)} \in \Delta^{op}\Ab $ is the simplicial free abelian group on $K$ and where the tensor product on the right side is degreewise the tensor product of $\Ch(\ShvA)$, i.e. for each $n \in \bb N$
	$$(\bb Z^{(K)} \otimes G^{\Delta^{op}}(X_i))_n := \bb Z^{(K)}_n \otimes G^{\Delta^{op}}(X_i)_n \in \Ch(\ShvA).$$
	
	It follows from Lemma \ref{lemmaDoldKanTensor}
	that the Dold-Kan correspondence $DK^{-1}: \Delta^{op}\Ch(\ShvA) \rightarrow \Ch_{\geq 0}(\Ch(\ShvA))$ preserves tensor products up to chain homotopy equivalence, and this chain homotopy equivalence is functorial.
	So the above isomorphism then implies that we have a natural chain homotopy equivalence
	$\hat{G}(K_+ \wedge X_{i,+}) \rightarrow \bb Z[K] \otimes \hat{G}(X_{i,+})$
	in $\Ch(\ShvA)$.
	Then we get a natural chain homotopy equivalence
	$$\hat{G}(K_+ \wedge A^c) = \colim{i \in I} \hat{G}(K_+ \wedge X_{i,+}) \rightarrow  \colim{i \in I} \bb Z[K] \otimes \hat{G}(X_{i,+}) \cong$$ $$ \cong \bb Z[K] \otimes \colim{i \in I} \hat{G}(X_{i,+}) = \bb Z[K] \otimes \hat{G}(A^c)$$
	in $\Ch(\ShvA)$.
\end{proof}

\begin{cor}\label{GABweq}
	Let $G : \Sm \rightarrow \Ch(\ShvA)$ be a $\Ch(\ShvA)$-enriched functor. Let $K$ be a finite simplicial set, and let $f : A \rightarrow B$ be a morphism in $f \cc M$ such that $\hat{G}(f)$ is a local quasi-isomorphism in $\Ch(\ShvA)$.
	Then the map $\hat{G}(K_+ \wedge f): \hat{G}(K_+ \wedge A) \rightarrow \hat{G}(K_+ \wedge B)$ is also a local quasi-isomorphism in $\Ch(\ShvA)$.
\end{cor}
\begin{proof}
	By Lemma \ref{FKUeqKFU} the map
	$\hat{G}(K_+ \wedge f): \hat{G}(K_+ \wedge A) \rightarrow \hat{G}(K_+ \wedge B)$
	is chain homotopic to the map
	$\bb Z[K] \otimes \hat{G}(f): \bb Z[K] \otimes \hat{G}(A) \rightarrow \bb Z[K] \otimes \hat{G}(B) $
	in $\Ch(\ShvA)$.
	If $\hat{G}(f)$ is also a local quasi-isomorphism, then since $\bb Z[K]$ is degreewise flat, it follows that $\bb Z[K] \otimes \hat{G}(f)$ is also a local quasi-isomorphism. So $\hat{G}(K_+ \wedge f)$ is a local quasi-isomorphism in $\Ch(\ShvA)$.
\end{proof}

\begin{lem} \label{GKLweq}
	Let $G : \Sm \rightarrow \Ch(\ShvA)$ be a $\Ch(\ShvA)$-enriched functor.
	Let $K, L$ be finite simplicial sets, $A \in f\cc M$ and let $e: K \rightarrow L$ be a weak equivalence of simplicial sets. Then 
	$\hat{G}(e_+ \wedge A) :  \hat{G}(K_+ \wedge A) \rightarrow \hat{G}(L_+ \wedge A)$ is a sectionwise quasi-isomorphism in $\Ch(\ShvA)$.
\end{lem}
\begin{proof}
	If $e: K \rightarrow L$ is a weak equivalence of simplicial sets,
	then it follows from basic properties of the Dold-Kan equivalence that $\bb Z[e] : \bb Z[K] \rightarrow \bb Z[L]$ is a quasi-isomorphism in $\Ch(\Ab)$. Let $C:= \mathrm{Cone}(\bb Z[e]) \in \Ch(\Ab)$ be the homological mapping cone of $\bb Z[e]$. Since $\bb Z[e]$ is a quasi-isomorphism, we know that $C$ is acyclic. Since $\bb Z[K]$ and $\bb Z[L]$ are degreewise free, we know that $C$ is degreewise free. So $0 \rightarrow C$ is a trivial cofibration in the projective model structure on $\Ch(\Ab)$. Since the projective model structure on $\Ch(\Ab)$ satisfies the monoid axiom, then for every $D \in \Ch(\Ab)$ the chain complex $C \otimes D$ is acyclic. Since $C \otimes D$ is the mapping cone of $\bb Z[e] \otimes D$, then for every $D \in \Ch(\Ab)$ the map
	$\bb Z[e]  \otimes D : \bb Z[K]  \otimes D \rightarrow \bb Z[L] \otimes D$
	is a quasi-isomorphism in $\Ch(\Ab)$.
	
	By Lemma \ref{FKUeqKFU} 
	$\hat{G}(e_+ \wedge A) :  \hat{G}(K_+ \wedge A) \rightarrow \hat{G}(L_+ \wedge A)$
	is chain homotopic to the map
	$\bb Z[e] \otimes \hat{G}(A) : \bb Z[K] \otimes \hat{G}(A) \rightarrow \bb Z[L] \otimes \hat{G}(A) $
	in $\Ch(\ShvA)$. But this is a sectionwise quasi-isomorphism, because for every $V \in \Smk$ the map $$\bb Z[e]  \otimes \hat{G}(A)(V) : \bb Z[K]  \otimes \hat{G}(A)(V) \rightarrow \bb Z[L] \otimes \hat{G}(A)(V)$$
	is a quasi-isomorphism in $\Ch(\Ab)$, by the above argument with $D:=\hat{G}(A)(V)$.
\end{proof}

\begin{defs}
	\begin{enumerate}
		\item
		A map $e : A\rightarrow X$ in a category $\mathscr{D}$ is called a \textit{coprojection} if it is isomorphic to the coproduct inclusion $A \rightarrow A \coprod Y$ for some $Y\in \mathscr{D}$.
		\item
		A map $e : A\rightarrow X$ in $\Delta^{op}\mathscr{D}$ is called a \textit{termwise coprojection}, if for every $n \in \mathbb{N}$, the map in the $n$-th simplicial degree $e_n : A_n \rightarrow X_n$ is a coprojection in $\mathcal{D}$.
		\item
		A pushout square in $\Delta^{op}\mathscr{D}$ $$\xymatrix{A \ar[r]^e \ar[d] & B \ar[d] \\
			C \ar[r]^{e^\prime} & D}$$  is called an \textit{elementary pushout square}, if $e$ and $e^\prime$ are termwise coprojections.
	\end{enumerate}
\end{defs}

Recall that throughout this section $F : \Sm \rightarrow \Ch(\ShvA)$ is a $\sim$-fibrant enriched functor, and that we have above constructed a non-enriched functor $\hat{F} : f\cc M \rightarrow \Ch(\ShvA)$.

\begin{lem}
	$\hat{F}$ takes elementary pushout squares in $\Delta^{op}\Sm$ to homotopy pushout squares in $\Ch(\ShvA)$.
\end{lem}
\begin{proof}
	Take a pushout square in $\Sm$, along coprojections $e, e^\prime$ : $$\xymatrix{A \ar[r]^e \ar[d] & A \coprod X \ar[d] \\
		B \ar[r]^{e^\prime} & B \coprod X}$$ 
	We can apply $F$ to get a square in $\Ch(\ShvA)$: $$\xymatrix{F(A) \ar[r] \ar[d] & F(A \coprod X) \ar[d] \\
		F(B) \ar[r] & F(B \coprod X)}$$
	According to Lemma \ref{Fcoprod} this square is isomorphic to $$\xymatrix{F(A) \ar[r] \ar[d] & F(A) \oplus F(X) \ar[d] \\
		F(B) \ar[r] & F(B) \oplus F(X)}$$
	By taking a local cofibrant replacement $F(X)^c$ of $F(X)$ we see that this square is locally equivalent to $$\xymatrix{F(A) \ar[r] \ar[d] & F(A) \oplus F(X)^c \ar[d] \\
		F(B) \ar[r] & F(B) \oplus F(X)^c}$$
	This square is a homotopy pushout, because it is a strict pushout and $F(A) \rightarrow F(A) \oplus F(X)^c$ is a cofibration. So $F$ sends pushout squares along coprojections in $\Sm$ to homotopy pushout squares in $\Ch(\ShvA)$.
	
	If we have an elementary pushout square $Q$ in $\Delta^{op}\Sm$ then in every simplicial degree it will be a pushout along coprojections. Then $F(Q)$ will be a square in $\Delta^{op}\Ch(\ShvA)$ that is in every simplicial degree a homotopy pushout. After applying the Dold-Kan correspondence we will still have a degreewise homotopy pushout, and after applying the total complex functor we obtain a single homotopy pushout square in $\Ch(\ShvA)$.
	So $\hat{F}(Q)$ is a homotopy pushout square in $\Ch(\ShvA)$.
\end{proof}

The previous lemma immediately implies the following corollary.

\begin{cor}\label{Fpush}
	If we have  an elementary pushout square in $\Delta^{op}\Sm$,
	$$\xymatrix{A \ar[r]^e \ar[d] & B \ar[d] \\
		C \ar[r]^{e^\prime} & D}$$ and $\hat{F}(e)$ is a local quasi-isomorphism, then $\hat{F}(e^\prime)$ is a local quasi-isomorphism in $\Ch(\ShvA)$.
\end{cor}

With all of these lemmas established, we can now prove the main result of this section.

\begin{proof}[Proof of Theorem \ref{mottolocal}]
	Let $Q$ be an elementary Nisnevich square of the form $$ \xymatrix{U^\prime \ar[d] \ar[r]&  X^\prime \ar[d] \\ U \ar[r]& X }$$
	In the category of pointed simplicial Nisnevich sheaves $\mathscr{M} = \Shv(\Smk, \psSet)$ we can  factor the morphism $U^\prime_+ \rightarrow X^\prime_+$ by using the mapping cylinder $C := (U^\prime_+ \times \Delta[1]_+) \underset{U^\prime_+}{\coprod} X^\prime_+ $ to get a factorization $\xymatrix{U^\prime_+ \hspace{3pt} \ar@{>->}[r]& C \ar@{->>}[r]^(0.4){\sim} & X^\prime_+ }$ where the left map is a cofibration and the right map is a simplicial homotopy equivalence. We define $s(Q) := U_+ \underset{U^\prime_+}{\coprod} C$. We can similarly take a mapping cylinder $t(Q)$ of the map $s(Q) \rightarrow X_+$ to factor it into $\xymatrix{s(Q) \hspace{3pt} \ar@{>->}[r]& t(Q) \ar@{->>}[r]^(0.4){\sim} & X_+ }$ where the left map is a cofibration and the right map a simplicial homotopy equivalence. We also take the mapping cylinder $C_X$ of $(\mathbb{A}^1\times X)_+ \rightarrow X_+$ to factor it as  $\xymatrix{(\mathbb{A}^1\times X)_+ \hspace{3pt} \ar@{>->}[r]& C_X \ar@{->>}[r]^(0.4){\sim} & X_+ }.$
	
	Let $J_{mot} = J_{proj} \cup J_{\mathbb{A}^1} \cup J_{nis} $ where
	$$J_{proj} = \{\Lambda^r[n]_+ \wedge U_+ \rightarrow \Delta[n]_+ \wedge U_+ \mid U \in \Smk, n > 0, 0 \leq r \leq n  \}$$
	$$J_{\mathbb{A}^1} = \{ \Delta[n]_+ \wedge U\times \mathbb{A}^1_+ \underset{\partial\Delta[n]_+ \wedge U \times \mathbb{A}^1}{\coprod} \partial\Delta[n]_+ \wedge C_U \rightarrow \Delta[n]_+ \wedge C_U  \mid U \in \Smk \}$$
	$$J_{nis} = \{\Delta[n]_+ \wedge s(Q) \underset{\partial\Delta[n]_+ \wedge s(Q)}{\coprod} \partial\Delta[n]_+ \wedge t(Q) \rightarrow \Delta[n]_+ \wedge t(Q) \mid Q \in \mathcal{Q} \}$$
	where $\mathcal{Q}$ is the set of elementary Nisnevich squares.
	We claim that $\hat{F}$ sends all morphisms in $J_{mot}$ to local quasi-isomorphisms.
	Since $\Lambda^r[n] \rightarrow \Delta[n]$ is a weak equivalence of simplicial sets it follows by Lemma \ref{GKLweq} that $\hat{F}(\Lambda^r[n]_+ \wedge U_+) \rightarrow \hat{F}(\Delta[n]_+ \wedge U_+)$ is a local quasi-isomorphism, so $\hat{F}$ sends $J_{proj}$ to local quasi-isomorphisms.
	
	Note that $\hat{F}$ sends simplicial homotopy equivalences to chain homotopy equivalences, because $\hat{F}(\Delta[1]_+ \otimes A^c)$ is a cylinder object for $\hat{F}(A^c)$.
	Since we have a local quasi-isomorphism $\hat{F}(X \times \mathbb{A}^1) \rightarrow \hat{F}(X)$ and a simplicial homotopy equivalence $C_X \rightarrow X_+$ we have a local quasi-isomorphism $\hat{F}(X \times \mathbb{A}^1) \rightarrow \hat{F}(C_X)$.
	
	Similarly, since $F$ satisfies Nisnevich excision we have a local quasi-iso\-mor\-phism $\hat{F}(s(Q)) \rightarrow \hat{F}(t(Q))$.
	Let $f: A \rightarrow B$ be a morphism either of the form $s(Q) \rightarrow t(Q)$ or $(X \times \mathbb{A}^1)_+ \rightarrow C_X$, and let $e: K \rightarrow L$ be a cofibration of simplicial sets. Then $e$ is a termwise coprojection and $\hat{F}(f)$ is a local quasi-iso\-mor\-phism.
	Consider the diagram 
	$$\xymatrix{K_+ \wedge A \ar[r] \ar[d]^{a_0}&  L_+ \wedge A \ar[d]^{a_2} \ar@/^2.0pc/[ddr]^{a_1}\\
		K_+ \wedge B \ar[r] \ar[drr] & K_+ \wedge B \underset{K_+ \wedge A}{\coprod} L_+ \wedge A \ar[dr]^(0.6){a_3} \\
		& & L_+ \wedge B }$$
	Since $\hat{F}(f)$ is a local quasi-isomorphism, by Lemma \ref{GABweq} also the maps $\hat{F}(a_0) = \hat{F}(K_+ \wedge f)$ and $\hat{F}(a_1) = \hat{F}(L_+ \wedge f)$ are local quasi-isomorphisms. By Corollary \ref{Fpush} also $\hat{F}(a_2)$ is a local quasi-isomorphism. By the 
	$2$-out-of-$3$-property this then implies that also $\hat{F}(a_3)$ is a local quasi-isomorphism.
	So $\hat{F}$ sends all morphisms from $J_{mot}$ to local quasi-isomorphisms.
	Theorem \ref{mottolocal} now follows by a simple small object argument, exactly like in the proof of Theorem 4.2 from \cite{garkusha2019framed}. 
\end{proof}

\section{The Röndigs--\O stv\ae r Theorem} \label{ROSection}
Recall that the category of motivic spaces $\cc M = \Shv(\Smk, \psSet)$ is equipped with a projective motivic model structure. See \cite[Theorem 2.12]{dundas2003motivic} for details.
This model structure induces a stable motivic model structure on the category of $(S^1,\Gm)$-bispectra of motivic spaces $\Sp_{S^1,\Gm}(\cc M)$.
We also have a motivic model structure on $\Ch(\ShvA)$, given by taking the left Bousfield localization of the local model structure on $\Ch(\ShvA)$ along the motivic equivalences from Definition \ref{motivicequivdef}.
This motivic model structure induces a stable motivic model sturcture on the category of $\Gm$-spectra of chain complexes $\Sp_{\Gm}(\Ch(\ShvA))$.
The homotopy category of $\Sp_{S^1,\Gm}(\cc M)$ is $SH(k)$.
The homotopy category of $\Sp_{\Gm}(\Ch(\ShvA))$ is $DM_{\Cor}$.

There is a forgetful functor $\mathcal{U} : DM_{\Cor} \rightarrow SH(k)$
with a left adjoint $\mathcal{L} : SH(k) \rightarrow DM_{\Cor}.$
It can be described as follows.	
The functor $\mathcal{U}$ is the derived functor of the right Quillen functor
$$Sp_{\Gm}(\Ch(\ShvA)) \overset{J}{\rightarrow}
Sp_{\Gm,S^1}(\mathrm{Ch_{\geq0}}(\ShvA)) \overset{DK}{\rightarrow}$$ $$ Sp_{\Gm,S^1}(\Delta^{op}\ShvA) \overset{U}{\rightarrow} Sp_{\Gm,S^1}(\cc M).$$
	Here $J: \Ch(\ShvA) \rightarrow \Sp_{S^1}(\mathrm{Ch_{\geq0}}(\ShvA))$ is the right Quillen equivalence that is called $T$ in \cite[Section 3]{jardine2003presheaves}.
	If $\tau_{\geq 0} : \Ch(\ShvA) \rightarrow \mathrm{Ch_{\geq0}}(\ShvA)$ is the good truncation functor sending $A \in \Ch(\ShvA)$ to
	$$\dots \rightarrow A_2 \rightarrow A_1 \rightarrow \ker(A_0 \overset{\partial^0_A}{\rightarrow} A_{-1})$$
	in $\Ch_{\geq 0}(\ShvA)$,
	then $J$ is defined on $A$ by $J(A) = (\tau_{\geq 0}(A[n]))_{n \in \bb N} \in \Sp_{S^1}(\Ch_{\geq 0}(\ShvA))$.
	
	The functor $DK: \mathrm{Ch_{\geq0}}(\ShvA) \rightarrow \Delta^{op}\ShvA$ is the Dold Kan equivalence, whose $n$-simplices are given by $$DK(X)_n = \underset{\underset{\text{surjective}}{[n] \rightarrow [k]} }{\bigoplus} X_k.$$
	$U : \ShvA \rightarrow \cc M$ is the functor that forgets transfers and the abelian group structure.
	We define $\hat{U} := U \circ DK\circ J $, so that $\mathcal{U}$ is the right derived functor of $\hat{U}$.

	We write $\mathcal{L}: SH(k) \rightarrow DM_{\Cor}$ for the left adjoint of $\mathcal{U}$.
	The adjunction $\mathcal{L}:SH(k) \rightleftarrows DM_{\Cor}: \mathcal{U} $ is a monoidal adjunction, so that $\mathcal{U}$ is lax monoidal and $\mathcal{L}$ is strong monoidal.
Furthermore $\mathcal{U}$ is a conservative functor. This means that if $f$ is a morphism in $DM_{\Cor}$ such that $\mathcal{U}(f)$ is an isomorphism in $SH(k)$, then $f$ is an isomorphism in $DM_{\Cor}$.

In this section we prove the following theorem, which is reminiscient of the Röndigs-\O stv\ae r theorem \cite[Corollary 56]{rondigs2008modules}.
\begin{thm} \label{RO}
Let $F : \Sm \rightarrow \Ch(\ShvA)$ be an enriched functor that is $\sim$-fibrant in $\Ch([\Sm,\ShvA])/\sim$.
Then for every $X \in \Smk$, the canonical morphism
$$ ev_{\Gm}(F) \otimes \motive{X} \rightarrow ev_{\Gm} (F(X \times -))$$ is an isomorphism in $DM_{\Cor}[1/p]$, which is natural in $X$.
\end{thm}

To prove \ref{RO} we will need several lemmas.
The most important lemma we will need is the following one from \cite[Corollary 56]{rondigs2008modules}:

\begin{lem}\label{C56}
	Let $X : f\mathscr{M} \rightarrow \mathscr{M}$ be a motivic functor that sends motivic equivalences between cofibrant objects to motivic equivalences. Let $B$ be a strongly dualizable object in $SH(k)[1/p]$. Then the canonical map of $(S^1,\Gm)$-bispectra $$ev_{S^1,\Gm}(X \wedge B) \rightarrow ev_{S^1,\Gm}(X \circ (-\wedge B))$$ is an isomoprhism in $SH(k)[1/p]$.
\end{lem}


The following theorem by Riou can be found in \cite[Appendix B, Corollary B.2]{levine2019algebraic}.
\begin{thm}\label{Riou}
If $U \in \Smk$, then $\Sigma^{\infty}_{S^1,\Gm}U_+$ is strongly dualizable in $SH(k)[1/p]$.
\end{thm}

To apply Lemma \ref{C56} in our situation, we  
have to convert $\Ch(\ShvA)$-enriched functors into motivic functors in the sense of \cite{dundas2003motivic}. We will now discuss how to do this.

We can consider the category of motivic spaces $\cc M$, the category of finitely presented motivic spaces $f\cc M$, the category of pointed smooth schemes $\Smkplus$ and the category of $S^1$-spectra of motivic spaces $\Sp_{S^1}(\cc M)$ to all be $\cc M$-enriched categories.
In the $\cc M$-enriched category $\Sp_{S^1}(\cc M)$ the morphism objects $\mathrm{Map}_{\Sp_{S^1}(\cc M)}(A,B) \in \cc M$, are defined for $A, B \in \Sp_{S^1}(\cc M)$ via an equalizer diagram, like in \cite[page 101]{hovey2001spectra}.
So we have an equalizer diagram:
\begin{equation} \label{S1equalizer}
	\xymatrix{ \mathrm{Map}_{\Sp_{S^1}(\cc M)}(A,B) \ar[r] & \underset{n \in \bb N}{\prod} \inthom{\cc M}{A_n}{B_n} \ar@<-.5ex>[r] \ar@<.5ex>[r] & \underset{n \in \bb N}{\prod} \inthom{\cc M}{S^1 \wedge A_n}{B_{n+1}} }.
\end{equation}
This makes $\Sp_{S^1}(\cc M)$ into an $\cc M$-enriched category.

In order to relate $\cc M$-enriched categories and $\Ch(\ShvA)$-enriched categories, 
we need some lax monoidal functors between $\cc M$ and $\Ch(\ShvA)$.
We have a non-enriched forgetful functor $\hat{U} : \Ch(\ShvA) \rightarrow \Sp_{S^1}(\cc M)$, and we have a functor $ev_0: \Sp_{S^1}(\cc M) \rightarrow \cc M$ taking the $0$-th weight of a $S^1$-spectrum.
The functor $ev_0 \circ \hat{U} : \Ch(\ShvA) \rightarrow \cc M$ has a left adjoint $L: \cc M \rightarrow \Ch(\ShvA)$.

\begin{lem}
	The functor $ev_0 \circ \hat{U} : \Ch(\ShvA) \rightarrow \cc M$ and its left adjoint $L : \cc M \rightarrow \Ch(\ShvA)$ are both lax monoidal functors.
\end{lem}
\begin{proof}
	The functor $\hat{U}$ is the composite
	$$\Ch(\ShvA) \overset{J}{\rightarrow} \Sp_{S^1}(\Ch_{\geq 0}(\ShvA)) \overset{DK}{\rightarrow} \Sp_{S^1}(\Delta^{op}(\ShvA)) \overset{U}{\rightarrow}  \Sp_{S^1}(\cc M).$$
	
	Let $\tau_{\geq 0} : \Ch(\ShvA) \rightarrow \Ch_{\geq 0}(\ShvA)$ be the good truncation functor sending $A \in \Ch(\ShvA)$ to
	$\dots \rightarrow A_2 \rightarrow A_1 \rightarrow \ker(A_0 \overset{\partial^0_A}{\rightarrow} A_{-1})$
	in $\Ch_{\geq 0}(\ShvA)$. 
	Then the following diagram commutes
	$$\xymatrix{\Ch(\ShvA) \ar[dr]_{\tau_{\geq 0}} \ar[r]^(0.4)J & \ar[d]^{ev_0} \Sp_{S^1}(\Ch_{\geq 0}(\ShvA)) \ar[r]^{DK} & \ar[d]^{ev_0} \Sp_{S^1}(\Delta^{op}\ShvA) \ar[r]^(0.6)U &  \ar[d]^{ev_0} \Sp_{S^1}(\cc M)  \\
	& \Ch_{\geq 0}(\ShvA) \ar[r]^{DK} & \Delta^{op}\ShvA \ar[r]^U & \cc M }.$$
To show that $ev_0\circ \hat{U}$ is lax monoidal, we just have to show that $U$, $DK$ and $\tau_{\geq 0}$ are lax monoidal, and to show that $L$ is lax monoidal we just have to show that each of the left adjoints of $U$, $DK$ and $\tau_{\geq 0}$ respectively is lax monoidal.

The left adjoint of $\tau_{\geq 0}$ is the inclusion functor $\Ch_{\geq 0}(\ShvA) \rightarrow  \Ch(\ShvA)$. This inclusion is obviously strong monoidal. This then implies that $\tau_{\geq 0}$ is lax monoidal. See \cite[Proposition 2.1]{nlab:oplax_monoidal_functor} or \cite[Theorem 1.2]{kelly2006doctrinal}.

The quasi-inverse of the Dold--Kan correspondence $DK^{-1} : \Delta^{op}(\ShvA) \rightarrow \Ch_{\geq 0}(\ShvA)$ is the normalized Moore complex functor.
It has a lax monoidal structure given by the Eilenberg--Zilber map
and it has an oplax monoidal structure given by the Alexander--Whitney map. See  \cite{nlab:eilenberg-zilber/alexander-whitney_deformation_retraction} or \cite[Definition 29.7]{may1992simplicial}.
	Since $DK^{-1}$ has an oplax monoidal structure it follows from \cite[Proposition 2.1]{nlab:oplax_monoidal_functor} that $DK$ has a lax monoidal structure.
	
	Finally, the forgetful functor $U: \Delta^{op}\ShvA \rightarrow \cc M$ is clearly lax monoidal as its left adjoint is strong monoidal.
	So $ev_0 \circ \hat{U} : \Ch(\ShvA) \rightarrow \cc M$ and its left adjoint $L : \cc M \rightarrow \Ch(\ShvA)$ are both lax monoidal functors.
\end{proof}

Let $F : \Sm \rightarrow \Ch(\ShvA)$ be a $\Ch(\ShvA)$-enriched functor.
We want to associate to $F$ an $\cc M$-enriched functor $$F^{\cc M} : f\cc M \rightarrow \Sp_{S^1}(\cc M).$$

To do this we will first construct a $\cc M$-enriched functor $\Smkplus \rightarrow \Sp_{S^1}(\cc M)$ and then Kan extend it to $f \cc M$.

The $\cc M$-enriched functor $\Smkplus \rightarrow \Sp_{S^1}(\cc M)$ is constructed as follows.
On objects it sends $X_+ \in \Smkplus$ to $\hat{U}(F(X)) \in \Sp_{S^1}(\cc M)$.
To define it on morphisms we now need to define for each $X, Y \in \Smk$ a map in $\cc M$:
$$\Smkplus(X_+,Y_+) \rightarrow \mathrm{Map}_{\Sp_{S^1}(\cc M)}(\hat{U}FX, \hat{U}FY). $$

This map is constructed in three steps.
In the following construction $X, Y\in \Smk$ are smooth schemes. Recall that $L : \cc M  \rightarrow \Ch(\ShvA)$ is the left adjoint of $ev_0 \circ \hat{U} : \Ch(\ShvA) \rightarrow \cc M$.

\begin{enumerate}
	\item 
	Since $L : \cc M \rightarrow \Ch(\ShvA)$ is lax monoidal, we have a map
	$$L \inthom{\cc M}{X_+}{Y_+} \rightarrow \inthom{\Ch(\ShvA)}{L(X_+)}{L(Y_+)}$$ in $\Ch(\ShvA)$. See \cite[Example 3.1]{nlab:closed_functor} for the construction of this map.
	By adjunction we get a map
	$$\inthom{\cc M}{X_+}{Y_+} \rightarrow ev_0\hat{U}\inthom{\Ch(\ShvA)}{L(X_+)}{L(Y_+)}$$ in $\cc M$.
	By construction, we have an isomorphism $L(X_+) \cong \Cor(-,X)_{\nis}$. Therefore $\inthom{\Ch(\ShvA)}{L(X_+)}{L(Y_+)} \cong \Sm(X,Y)$. Furthermore $\Smkplus(X_+,Y_+) = \inthom{\cc M}{X_+}{Y_+}$. We therefore get a map in $\cc M$.
	$$\Smkplus(X_+,Y_+) \rightarrow ev_0\hat{U}\Sm(X,Y).$$
	
	\item
	Since $F : \Sm \rightarrow \Ch(\ShvA)$ is a $\Ch(\ShvA)$-enriched functor we have a map $\Sm(X,Y) \rightarrow \inthom{\Ch(\ShvA)}{FX}{FY}$ in $\Ch(\ShvA)$.
	We thus also get a map in $\cc M$.:
	$$ev_0\hat{U}\Sm(X,Y) \rightarrow ev_0\hat{U}\inthom{\Ch(\ShvA)}{FX}{FY}.$$
	\item
	For every $n \in \bb N$, and every $A, B \in \Ch(\ShvA)$ the chain complex shift functor $[n]$ gives us an isomorphism
	$$\inthom{\Ch(\ShvA)}{A}{B} \overset{\sim}{\rightarrow} \inthom{\Ch(\ShvA)}{A[n]}{B[n]}.$$
	Since $ev_0\hat{U}$ is lax monoidal, we can use \cite[Example 3.1]{nlab:closed_functor} to get a canonical map
	$$ev_0\hat{U}\inthom{\Ch(\ShvA)}{A[n]}{B[n]} \rightarrow \inthom{\cc M}{ev_0\hat{U}A[n]}{ev_0\hat{U}B[n]} =\inthom{\cc M}{(\hat{U}A)_n}{(\hat{U}B)_n}.$$
\end{enumerate}
	All these maps $ev_0\hat{U}\inthom{\Ch(\ShvA)}{A}{B} \rightarrow \inthom{\cc M}{(\hat{U}A)_n}{(\hat{U}B)_n}$
	yield a map 
	$$ev_0\hat{U}\inthom{\Ch(\ShvA)}{A}{B} \rightarrow \underset{n \in \bb N}{\prod} \inthom{\cc M}{(\hat{U}A)_n}{(\hat{U}B)_n}.$$
	We want to show that it factors over $\mathrm{Map}_{\Sp_{S^1}(\cc M)}(\hat{U}A,\hat{U}B) $.
	Since $\hat{U}(A)$ is a $S^1$-spectrum we have for every $A \in \Ch(\ShvA)$ a map
	$$S^1 \wedge ev_0\hat{U}(A[n]) \rightarrow ev_0\hat{U}(A[n+1])$$
	in $\cc M$. Since $\hat{U}$ is a functor, this map is natural in $A$.
	Using this naturality one can check that for all $A,B \in \Ch(\ShvA)$ the following diagram commutes:
	\footnotesize
	$$\xymatrix{ ev_0\hat{U}\inthom{\Ch(\ShvA)}{A[n]}{B[n]} \ar[r] \ar[d]^\sim & \inthom{\cc M}{(\hat{U}A)_n }{(\hat{U}B)_n} \ar[r]^(0.4){S^1\wedge -} & \inthom{\cc M}{S^1 \wedge (\hat{U}A)_n }{ S^1 \wedge (\hat{U}B)_n} \ar[d] \\
	ev_0\hat{U}\inthom{\Ch(\ShvA)}{A[n+1]}{B[n+1]}\ar[r] & \inthom{\cc M}{(\hat{U}A)_{n+1} }{(\hat{U}B)_{n+1}} \ar[r]  &\inthom{\cc M}{S^1 \wedge (\hat{U}A)_n }{(\hat{U}B)_{n+1}}   } $$\normalsize
	
	By the equalizer universal property of $\mathrm{Map}_{\Sp_{S^1}(\cc M)}(\hat{U}A,\hat{U}B) $ from diagram (\ref{S1equalizer}) we get a dotted map like in the following diagram
	$$\xymatrix{& ev_0\hat{U}\inthom{\Ch(\ShvA)}{A}{B} \ar[d] \ar@{..>}[dl] \\
	 \mathrm{Map}_{\Sp_{S^1}(\cc M)}(\hat{U}A,\hat{U}B) \ar[r] & \underset{n \in \bb N}{\prod} \inthom{\cc M}{(\hat{U}A)_n}{(\hat{U}B)_n} \ar@<-.5ex>[r] \ar@<.5ex>[r] & \underset{n \in \bb N}{\prod} \inthom{\cc M}{S^1 \wedge (\hat{U}A)_n}{(\hat{U}B)_{n+1}} }$$
	In particular, we have a map
	$$ev_0\hat{U}\inthom{\Ch(\ShvA)}{FX}{FY} \rightarrow \mathrm{Map}_{\Sp_{S^1}(\cc M)}(\hat{U}FX, \hat{U}FY).$$
And then we have maps
$$\Smkplus(X_+,Y_+) \rightarrow ev_0\hat{U}\Sm(X,Y) \rightarrow ev_0\hat{U}\inthom{\Ch(\ShvA)}{FX}{FY} \rightarrow \mathrm{Map}_{\Sp_{S^1}(\cc M)}(\hat{U}FX, \hat{U}FY).$$
By composing these three steps together we get a map
$$\Smkplus(X,Y) \rightarrow \mathrm{Map}_{\Sp_{S^1}(\cc M)}(\hat{U}FX, \hat{U}FY) $$
in $\cc M$. This map preserves identity morphisms and is compatible with composition, so we get an $\cc M$-enriched functor  $\Smkplus \rightarrow \Sp_{S^1}(\cc M)$, sending $X$ to $\hat{U}FX$.

We now define $F^{\cc M} : f\cc M \rightarrow \Sp_{S^1}(\cc M)$ to be the $\cc M$-enriched Left Kan extension of this $\cc M$-enriched functor $\Smkplus \rightarrow \Sp_{S^1}(\cc M)$
along the $\cc M$-enriched inclusion functor $\Smkplus \rightarrow f\cc M$.
$$\xymatrix{ \Smkplus \ar[r] \ar[d]  & \Sp_{S^1}(\cc M)\\
f\cc M \ar@{..>}[ru]_{F^{\cc M}} } $$
The functor $F^{\cc M}$ can be explicitly computed on $A \in f\cc M$ as
$$F^{\cc M}(A) = \overset{X_+ \in \Smkplus}{\int} \hat{U}(F(X)) \wedge \inthom{\cc M}{X_+}{A} .$$
Note that $F^{\cc M}$ respects filtered colimits, because $X_+ \in\Smkplus$ is finitely presented in $\cc M$.

\begin{lem}\label{FmvshatF}
	Let $F: \Sm \rightarrow \Ch(\ShvA)$ be a $\sim$-fibrant functor.
	For every finitely presented motivic space $A \in f\cc M$ with cofibrant replacement $A^c$, we have a natural isomorphism
	$(\hat{U} \circ \hat{F})(A) \cong F^{\cc M}(A^c)$
	in $\Sp_{S^1}(\cc M)$. Here $\hat{U} : \Ch(\ShvA) \rightarrow \Sp_{S^1}(\cc M)$ is the forgetful functor and $\hat{F} : f\cc M \rightarrow \Ch(\ShvA)$ is the extension of $F$ defined by equation \eqref{hatdef} in Section \ref{mottolocalsection}.
\end{lem}
\begin{proof}
	If $A= X_+$ for some $X \in \Smk$ we have
	$\hat{U}(\hat{F}(X_+)) \cong \hat{U}(F(X)) $
	and by the $\cc M$-enriched co-Yoneda lemma we have
	$$ \hat{U}(F(X)) \cong   \overset{Y_+ \in \Smkplus}{\int} \hat{U}(F(Y)) \wedge \inthom{\cc M}{Y_+}{X_+} = F^{\cc M}(X_+) .$$
	So the claim is true for $A=X_+$.
	The claim then also follows for all other objects $A$ in $f\cc M$, because $A^c$ is a filtered colimit of simplicial schemes, and $F^{\cc M}$ respects filtered colimits.
\end{proof}

\begin{lem} \label{Fhatfibrant}
	Let $F: \Sm \rightarrow \Ch(\ShvA)$ be a pointwise locally fibrant functor, and let $A \in f\cc M$ be a finitely presented motivic space.
	Then $\hat{F}(A)$ is locally fibrant in $\Ch(\ShvA)$.
\end{lem}
\begin{proof}
	For every scheme $X$ we know that $F(X)$ is locally fibrant in $\Ch(\ShvA)$.
	If $A$ is a finitely presented motivic space, then $A^c$ is a filtered colimit of simplicial schemes.
	$A^c = \colim{i \in I} X_i$
	for some $X_i \in \Delta^{op}\Smk$ and filtered diagram $I$,
	and we have 
	$\hat{F}(A) = \colim{i \in I} \hat{F}(X_i).$
	The fact that $F$ is pointwise locally fibrant implies for each $i \in I$ that $\hat{F}(X_i)$ is locally fibrant in $\Ch(\ShvA)$.
	By Lemma \ref{transfereddweaklyfingen} the model category $\Ch(\ShvA)$ is weakly finitely generated, so it follows by \cite[Lemma 3.5]{dundas2003enriched} that filtered colimits of fibrant objects are fibrant. So $\hat{F}(A)$ is locally fibrant in $\Ch(\ShvA)$.
\end{proof}

For every $n \in \bb N$ we can take the $n$-th level of the functor $F^{\cc M} : f\cc M \rightarrow \Sp_{S^1}(\cc M)$
to get an $\cc M$-enriched motivic functor $$F^{\cc M}_n : f\cc M \rightarrow \cc M.$$
The functor $F^{\cc M}_n$ is then a motivic functor as defined in \cite{dundas2003motivic}.

\begin{lem}\label{ROApplicable}
	Let $F: \Sm \rightarrow \Ch(\ShvA)$ be a $\sim$-fibrant enriched functor.
	For every $n \in \bb N$ the motivic functor $F^{\cc M}_n : f\cc M \rightarrow \cc M$ sends motivic equivalences between cofibrant objects to local equivalences.
\end{lem}
\begin{proof}
	By Theorem \ref{mottolocal} we know that $\hat{F} : f\cc M \rightarrow \Ch(\ShvA)$ sends motivic equivalences to local quasi-isomorphisms. By Lemma \ref{Fhatfibrant} we know that $\hat{F}$ sends all objects of $f \cc M$ to locally fibrant objects. With respect to the $S^1$-stable local model structure on $\Sp_{S^1}(\cc M)$ and the local model structure on $\Ch(\ShvA)$, the functor $\hat{U} : \Ch(\ShvA) \rightarrow \Sp_{S^1}(\cc M)$ is a right Quillen functor, so it preserves weak equivalences between fibrant objects. It then follows that $\hat{U} \circ \hat{F} : f \cc M \rightarrow \Sp_{S^1}(\cc M)$ sends motivic equivalences to stable local equivalences between locally fibrant $S^1$-spectra in $\Sp_{S^1}(\cc M)$. Hence  $\hat{U} \circ \hat{F}$ sends motivic equivalences to levelwise local equivalences.
	By Lemma \ref{FmvshatF} this then means that $F^{\cc M} : f \cc M \rightarrow \Sp_{S^1}(\cc M)$ sends motivic equivalences between cofibrant objects to levelwise local equivalences in $\Sp_{S^1}(\cc M)$.
	So for every $n \in \bb N$ the motivic functor $F^{\cc M}_n : f\cc M \rightarrow \cc M$ sends motivic equivalences between cofibrant objects to local equivalences.
\end{proof}

Before proving the main theorem of this section, we need an additional lemma about $(S^1,S^1,\Gm)$-trispectra.
To avoid confusion between the two $S^1$-directions we now introduce an extra notation. We write $S^1_1$ for the first $S^1$-direction and we write $S^1_2$ for the second $S^1$-direction. Therefore, whenever we discuss $(S^1,S^1,\Gm)$-spectra, we deal with $(S^1_1,S^1_2,\Gm)$-spectra following this notation.
For every $F : \Sm \rightarrow \Ch(\ShvA)$  we consider $F^{\cc M} : f\cc M \rightarrow \Sp_{S^1_2}(\cc M)$ to be a functor landing in $S^1_2$-spectra.

Given a $\Gm$-spectrum of chain complexes $A \in \SpGm(\Ch(\ShvA))$ we let  $\bb Z[\bb S] \boxtimes A \in \Sp_{S^1_1,\Gm}(\Ch(\ShvA))$ refer to the $(S^1_1,\Gm)$-bispectrum of chain complexes that is given in $S^1_1$-weight $n$ by
$$(\bb Z[\bb S] \boxtimes A)_n := \bb Z[S^n] \otimes A \in \SpGm(\Ch(\ShvA)) .$$
The definition of $\bb Z[S^n]$ is in Section \ref{mottolocalsection}, equation \eqref{ZKdefinition}. It is the chain complex that is $\bb Z$ concentrated in homological degree $n$.

The functor $\hat{U} : \Sp_{\Gm}(\Ch(\ShvA)) \rightarrow \Sp_{S^1_2,\Gm}(\cc M)$ can naively be extended to a functor denoted by the same letter $$\hat{U} : \Sp_{S^1_1,\Gm}(\Ch(\ShvA)) \rightarrow  \Sp_{S^1_1,S^1_2,\Gm}(\cc M)$$ by applying it $S^1_1$-levelwise.

\begin{lem} \label{evComparison}
	Let $F: \Sm \rightarrow \Ch(\ShvA)$ be a $\sim$-fibrant functor.
	For every $X \in \Smk$ we have a natural map of $(S^1_1,S^1_2,\Gm)$-trispectra
	$$ev_{S^1_1,\Gm}(F^{\cc M}(- \times X)) \rightarrow \hat{U}( \bb Z[\bb S] \boxtimes   ev_{\Gm}(F(- \times X)))$$
	in $\Sp_{S^1_1,S^1_2,\Gm}(\cc M)$.
	This map is a $S^1_1$-levelwise $(S^1_2,\Gm)$-stable motivic equivalence.
\end{lem}
\begin{proof}
	Since we are only evaluating $F^{\cc M}$ on simplicial schemes, by Lemma \ref{FmvshatF} we just need to show that there is a $S^1_1$-levelwise $(S^1_2,\Gm)$-stable motivic equivalence
	$$ev_{S^1_1,\Gm}((\hat{U}\circ \hat{F})(- \times X)) \rightarrow \hat{U}( \bb Z[\bb S] \boxtimes   ev_{\Gm}(F(- \times X))).$$
	And this follows from Lemma \ref{FKUeqKFU}.
\end{proof}

We are now in a position to prove the main theorem of this section.
\begin{proof}[Proof of Theorem \ref{RO}]
Let $F : \Sm \rightarrow \Ch(\ShvA)$ be a $\sim$-fibrant functor.
Due to Lemma \ref{ROApplicable} and Lemma \ref{Riou} we can apply Lemma \ref{C56} to get an isomorphism
$$ ev_{S^1_1,\Gm}(F^{\cc M}_n) \wedge \Sigma^{\infty}_{S^1,\Gm}X_+ \overset{\sim}{\rightarrow} ev_{S^1_1,\Gm}(F^{\cc M}_n (- \times X))$$ in $SH(k)[1/p]$.
These combine into a $S^1_2$-levelwise $(S^1_1,\Gm)$-stable motivic equivalence of $(S^1_1,S^1_2,\Gm)$-trispectra
$$ ev_{S^1_1,\Gm}(F^{\cc M}) \wedge \Sigma^{\infty}_{S^1,\Gm}X_+ \overset{\sim}{\rightarrow} ev_{S^1_1,\Gm}(F^{\cc M} (- \times X))$$
in $\Sp_{S^1_1,S^1_2,\Gm}(\cc M)[1/p]$.
By Lemma \ref{evComparison} 
we have a commutative diagram
$$\xymatrix{ ev_{S^1_1,\Gm}(F^{\cc M}) \wedge \Sigma^{\infty}_{S^1,\Gm}X_+ \ar[r]^\sim \ar[d]^\sim & ev_{S^1_1,\Gm}(F^{\cc M} (- \times X)) \ar[d]^\sim \\
\hat{U}(\bb Z[\bb S] \boxtimes ev_{\Gm}(F)) \wedge \Sigma^{\infty}_{S^1,\Gm}X_+ \ar[r]^\sim & \hat{U}(\bb Z[\bb S] \boxtimes ev_{\Gm}(F (- \times X))) }$$
where the vertical maps are  $S^1_1$-levelwise $(S^1_2,\Gm)$-stable equivalences.
It follows that the bottom horizontal map is a $(S^1_1,S^1_2,\Gm)$-stable equivalence.
By Lemma \ref{Riou} we know that $\Sigma^{\infty}_{S^1,\Gm} X_+$ is strongly dualizable in $SH(k)[1/p]$. Since $\mathcal{L}$ and $\mathcal{U}$ are a monoidal adjunction, we can apply \cite[Chapter 7, Lemma 4.6]{bachmann2020milnor} to get for every $n \in \bb N$ that
$$  \mathcal{U}(\bb Z[S^n] \otimes ev_{\Gm}(F)) \wedge \Sigma^{\infty}_{S^1,\Gm}X_+ \cong \mathcal{U}(\bb Z[S^n] \otimes ev_{\Gm}(F) \otimes \motive{X})  $$
in $SH(k)[1/p]$.
These assemble into a $S^1_1$-levelwise $(S^1_2,\Gm)$-stable equivalence of trispectra
$$  \hat{U}(\bb Z[\bb S] \boxtimes ev_{\Gm}(F)) \wedge \Sigma^{\infty}_{S^1,\Gm}X_+ \rightarrow \hat{U}(\bb Z[\bb S] \boxtimes ev_{\Gm}(F) \otimes \motive{X})   .$$

We then have a commutative diagram
$$\xymatrix{ \hat{U}(\bb Z[\bb S] \boxtimes ev_{\Gm}(F) \otimes \motive{X}) \ar[r] & \hat{U}(\bb Z[\bb S] \boxtimes ev_{\Gm}(F (- \times X))) \\
\hat{U}(\bb Z[\bb S] \boxtimes ev_{\Gm}(F)) \wedge \Sigma^{\infty}_{S^1,\Gm}X_+ \ar[u]^{\sim} \ar[ur]_{\sim}  }$$
in $\Sp_{S^1_1,S^1_2,\Gm}(\cc M)[1/p]$,
where the two lower maps are $(S^1_1,S^1_2,\Gm)$-stable motivic equivalences.
It follows that the upper horizontal map is a $(S^1_1,S^1_2,\Gm)$-stable motivic equivalence in $\Sp_{S^1_1,S^1_2,\Gm}(\cc M)[1/p]$.

Since $\mathcal{U}:DM_{\Cor}[1/p] \rightarrow SH(k)[1/p]$ is conservative, we then get a $(S^1_1,\Gm)$-stable motivic equivalence
$$\bb Z[\bb S] \boxtimes ev_{\Gm}(F) \otimes \motive{X} \overset{\sim}{\rightarrow} \bb Z[\bb S] \boxtimes ev_{\Gm}(F (- \times X))$$
in $\Sp_{S^1_1,\Gm}(\Ch(\ShvA))[1/p]$.
Since the functor $$\bb Z[S^1] \otimes - : \Sp_{\Gm}(\Ch(\ShvA))[1/p] \rightarrow \Sp_{\Gm}(\Ch(\ShvA))[1/p]$$ is an auto-equivalence, it follows from \cite[Theorem 5.1]{hovey2001spectra} that $$\bb Z[\bb S] \boxtimes - : \Sp_{\Gm}(\Ch(\ShvA))[1/p] \rightarrow \Sp_{S^1_1,\Gm}(\Ch(\ShvA))[1/p]$$
is a Quillen equivalence, where $\Sp_{S^1_1,\Gm}(\Ch(\ShvA))$ is equipped with the stable model structure of $\bb Z[S^1]$-spectra in $\Sp_{\Gm}(\Ch(\ShvA))$.
Since $\bb Z[\bb S] \boxtimes -$ preserves weak equivalences between all objects from $\Sp_{\Gm}(\Ch(\ShvA))[1/p]$, this then implies that $$ev_{\Gm}(F) \otimes \motive{X} \overset{\sim}{\rightarrow} ev_{\Gm}(F (- \times X))$$ is an isomorphism in $DM_{\Cor}[1/p]$.
\end{proof}

\section{Proof of Theorem \ref{mainthm}} \label{SecondStep}

In this section we will prove Theorem \ref{mainthm}, but we first need a few lemmas.

\begin{lem}\label{compactgenerationLemma}
	The category $D([\Sm,\ShvA])/\sim[1/p]$ is compactly generated by the set
	$$\{ [\Gmn{n},I(-)]\otimes Z \mid n \in \bb N, Z \in \Smk  \}  .$$
\end{lem}
\begin{proof}
	The objects $[\Gmn{n},I(-)]\otimes Z$ are compact by \cite[Theorem 6.2]{garkusha2019derived}.
	Let $F\in D([\Sm,\ShvA])/\sim[1/p]$ be an enriched functor such that
	for all $n\in \bb N, Z \in \Smk$
	$$\mathrm{Hom}_{D([\Sm,\ShvA])/\sim[1/p]}([\Gmn{n},I(-)]\otimes Z,F) = 0.$$
	Without loss of generality, $F$ is $\sim$-fibrant.
	Then we get for all $n\in \bb N, Z \in \Smk$ that
	$F(\Gmn{n})(Z) \cong 0$
	in $D(\Ab)[1/p]$.
	This implies that
	$ev_{\Gm}(F) \cong 0$
	in $DM_{\Cor}[1/p]$.
	It follows Theorem \ref{RO} that for every $U \in \Smk$
	$$ev_{\Gm}(F(U \times -))\cong ev_{\Gm}(F) \otimes \mathcal{L}(\Sigma^{\infty}_{S^1,\Gm}U_+) \cong 0 $$
	in $DM_{\Cor}[1/p]$.
	Since $F(U \times -)$ is $\sim$-fibrant, the $\Gm$-spectrum $ev_{\Gm}(F(U \times -))$ is motivically fibrant in $DM_{\Cor}[1/p]$.
	Then
	$$F(U) \cong F(U \times pt) = ev_{\Gm}(F(U \times -))(0) \cong 0$$
	in $D(\ShvA)[1/p]$.
	This means that
	$F \cong 0$ in
	$D([\Sm,\ShvA])/\sim[1/p]$.
	So 
$$\{ [\Gmn{n},I(-)]\otimes Z \mid n \in \bb N, Z \in \Smk  \}$$ is a set of compact generators for $D([\Sm,\ShvA])/\sim[1/p]$.
\end{proof}

\begin{lem} \label{withoutZfibrant}
The enriched functor $[\Gmn{n},M_{\Cor}(-)] : \Sm \rightarrow \Ch(\ShvA)$ satisfies Nisnevich excision in the sense of Definition \ref{nisdef}.
\end{lem}
\begin{proof}
Take an elementary Nisnevich square: $$\xymatrix{ U^\prime \ar[r]_{\beta} \ar[d]^{\alpha} & X^\prime \ar[d]_{\gamma}\\
	U \ar[r]^{\delta} & X }$$
From Definition \ref{tak} it follows that there is an exact sequence
$$0 \rightarrow \Cor(-,U^\prime)_{nis} \rightarrow \Cor(-,U)_{nis} \oplus \Cor(-,X^\prime)_{nis}  \rightarrow \Cor(-,X)_{nis} \rightarrow 0 .$$
Since $\Cor$ is a strict $V$-category of correspondences, by applying $C_*$ we get a triangle 
$$M_{\Cor}(U^\prime) \rightarrow M_{\Cor}(U) \oplus M_{\Cor}(X^\prime) \rightarrow M_{\Cor}(X) \rightarrow \Sigma M_{\Cor}(U^\prime)$$
in $D(\ShvA)$.
We can take local fibrant replacements $M_{\Cor}(X)^f$  of each of these terms $M_{\Cor}(X)$, and then apply $\Omega_{\Gm}^n$ to get a triangle of locally fibrant complexes in $D(\ShvA)$
$$\Omega_{\Gm}^n(M_{\Cor}(U^\prime)^f) \rightarrow \Omega_{\Gm}^n(M_{\Cor}(U)^f) \oplus \Omega_{\Gm}^n(M_{\Cor}(X^\prime)^f) \rightarrow \Omega_{\Gm}^n(M_{\Cor}(X)^f) \rightarrow \Sigma\Omega_{\Gm}^n(M_{\Cor}(U^\prime)^f).$$

Lemma \ref{Gmpreserveslocal} says that $\inthom{\Ch(\ShvA)}{\Gmn{1}}{-} : \Ch(\ShvA) \rightarrow \Ch(\ShvA)$ preserves local equivalences between $\bb A^1$-local complexes. This implies that   $[\Gmn{n},M_{\Cor}(X)]$ is locally equivalent to $[\Gmn{n},M_{\Cor}(X)^f] = \Omega_{\Gm}^n(M_{\Cor}(X)^f)$. So we ultimately get a triangle in $D(\ShvA)$
$$[\Gmn{n},M_{\Cor}(U^\prime)] \rightarrow [\Gmn{n},M_{\Cor}(U)] \oplus [\Gmn{n},M_{\Cor}(X^\prime)] \rightarrow [\Gmn{n},M_{\Cor}(X)]\rightarrow \Sigma [\Gmn{n},M_{\Cor}(U^\prime)].$$
This means that $$\xymatrix{ [\Gmn{n},M_{\Cor}(U^\prime)] \ar[r]_{\beta_*} \ar[d]^{\alpha_*} & [\Gmn{n},M_{\Cor}(X^\prime)] \ar[d]_{\gamma_*}\\
	[\Gmn{n},M_{\Cor}(U)] \ar[r]^{\delta_*} & [\Gmn{n},M_{\Cor}(X)] }$$ is homotopy cartesian, so $[\Gmn{n},M_{\Cor}(-)] : \Sm \rightarrow \Ch(\ShvA)$ satisfies Nisnevich excision.
\end{proof}

\begin{lem} \label{fibrant}
	For every $Z \in \Smk$ the enriched functor $[\Gmn{n},M_{\Cor}(- \times Z)] : \Sm \rightarrow \Ch(\ShvA)$ satisfies Nisnevich excision in the sense of Definition \ref{nisdef}.
\end{lem}
\begin{proof}
	Take an elementary Nisnevich square $$\xymatrix{ U^\prime \ar[r]_{\beta} \ar[d]^{\alpha} & X^\prime \ar[d]_{\gamma}\\
		U \ar[r]^{\delta} & X }$$
	Then the square $$\xymatrix{ U^\prime \times Z \ar[r]_{\beta \times 1} \ar[d]^{\alpha \times 1} & X^\prime \times Z \ar[d]_{\gamma \times 1}\\
		U \times Z \ar[r]^{\delta \times 1} & X \times Z }$$ is again an elementary Nisnevich square.
	The result now follows from Lemma \ref{withoutZfibrant}.
\end{proof}

\begin{proof}[Proof of Theorem \ref{mainthm}]
	Let $$T_{\mathscr{C}} := \langle [\Gmt{n},-]\otimes X \mid n \in \mathbb{N}, X\in \Smk \rangle$$ be the full triangulated subcategory of $D([\Sm,\ShvA])$ that is compactly generated by $[\Gmt{n},-]\otimes X$ for all $n \in \mathbb{N}$ and $X \in \Smk$.
According to \cite[Lemma 4.10]{garkusha2022recollements} the composite 
	$$T_{\mathscr{C}} \rightarrow D([\Sm,\ShvA]) \overset{res}{\rightarrow} D([\mathscr{C},\ShvA])$$ 
	is an equivalence of triangulated categories, where the first map is the inclusion map and the second map is the map restricting functors from $\Sm$ to $\cc C$.
	
	Let $\widehat{\sim_{\cc{C}}}$ be the set of morphisms, following the notation from Lemma \ref{localobjectcomparison} by
	$$\widehat{\sim_{\cc{C}}} := \{ (f \otimes Z)[n] \mid f \in \sim_{\cc C}, Z \in \Smk, n \in \bb N\}.$$
		Here $\sim_{\cc{C}}$ is defined in Section \ref{simCdef} on page \pageref{simCdef}.
	We can consider $\widehat{\sim_{\cc{C}}}$ to be a set of morphisms in $T_{\cc C}$. We write $T_{\mathscr{C}}/\sim_{\cc{C}}$ for the localization of $T_{\cc C}$ along the set of morphisms $\widehat{\sim_{\cc{C}}}$ between compact objects.
	
	The equivalence $T_{\cc C} \rightarrow D([\mathscr{C},\ShvA])$ then induces an equivalence of compactly generated triangulated categories $$T_{\mathscr{C}}/\sim_{\cc{C}} \rightarrow D([\mathscr{C},\ShvA])/\sim_{\cc{C}}.$$
	By Theorem \ref{cthm} we have that $$ev_{\Gm}: D([\mathscr{C}, \ShvA])/\sim_{\cc{C}} \rightarrow DM_{\Cor}$$ is an equivalence of compactly generated triangulated categories.
	So we have an equivalence of compactly generated triangulated categories
	$$ev_{\Gm} : T_{\mathscr{C}}/\sim_{\cc{C}} \rightarrow DM_{\Cor}.$$
	Next, the inclusion $T_{\cc C} \rightarrow D([\Sm,\ShvA])$ induces a triangulated functor
	$$\Phi: T_{\cc C}/\sim_{\cc{C}}  \rightarrow D([\Sm,\ShvA])/\sim .$$
	We will now use Lemma \ref{garkushaLemma} to show that $$\Phi[1/p] : T_{\cc C}/\sim_{\cc{C}}[1/p]  \rightarrow D([\Sm,\ShvA])/\sim [1/p] $$ is an equivalence of triangulated categories.
	Following the notation of Lemma \ref{garkushaLemma}, here $A= T_{\cc C}/\sim_{\cc{C}}[1/p] $ and $B=D([\Sm,\ShvA])/\sim [1/p]$ are compactly generated triangulated categories.
	
	Due to Lemma \ref{binomial} and the definition of $T_{\cc C}$, the set $$\Sigma:= \{ [\Gmn{n},I(-)] \otimes X \mid n \in \bb N, X \in \Smk   \} $$is a set of compact generators for $T_{\cc C}/\sim_{\cc{C}}[1/p]$. This is the set of compact generators to which we apply Lemma \ref{garkushaLemma}.
	Due to Lemma \ref{compactgenerationLemma}, the functor $\Phi[1/p]$ sends $\Sigma$ to a set of compact generators for $D([\Sm,\ShvA])/\sim [1/p]$, so the first condition of Lemma \ref{garkushaLemma} is satisfied.
	
	Let us check the second condition.
	Since $T_{\mathscr{C}}/\sim_{\cc{C}}$ is equivalent to $D([\mathscr{C}, \ShvA])/\sim_{\cc{C}}$,
	by Lemma \ref{GmtimesXcomputation} we have an isomorphism 
	$$[\Gmt{n},I(-)]\otimesShv X \cong [\Gmt{n},\cc M_{\Cor}(X)] $$
	in $T_{\mathscr{C}}/\sim_{\cc{C}}$.
	From Lemma \ref{fibrant} it follows that the enriched functor $[\Gmt{n},\cc M_{\Cor}(X)]: \Sm \rightarrow \Ch(\ShvA)$ satisfies Nisnevich excision. Similarly to Lemma \ref{enrichedmotivelocal}, it is also strictly local with respect to the relations $\bb A^1_1$, $\tau$. The definitions of these relations is in Section \ref{simdef}, page \pageref{simdef}. Since the map
	$M_{\Cor}(X \times \bb A^1) \rightarrow M_{\Cor}(X)$
	is an isomorphism in $DM_{\Cor}^{\eff}$ between $\bb A^1$-local complexes, so it is also a local quasi-isomorphism. Since $[\Gmn{n},-]$ preserves local quasi-isomorphisms between $\bb A^1$-local objects, it follows that $[\Gmt{n},\cc M_{\Cor}(X)]$ is strictly local with respect to $\bb A^1_2$.
	So the enriched functor
	$[\Gmt{n},\cc M_{\Cor}(X)]: \Sm \rightarrow \Ch(\ShvA)$ is strictly $\sim$-local.
	Also for every $d \in \bb N$ the shifted functor
	$[\Gmt{n},\cc M_{\Cor}(X)][d]: \Sm \rightarrow \Ch(\ShvA)$ is strictly $\sim$-local.
	
	The functor $\Phi : T_{\cc C}/\sim_{\cc{C}} \rightarrow D([\Sm,\ShvA])/\sim$ is by construction fully faithful on strictly $\sim$-local objects, in the sense that if $A, B \in T_{\cc C}/\sim_{\cc{C}}$ are strictly $\sim$-local then the map
	$$\mathrm{Hom}_{T_{\cc C}/\sim_{\cc{C}}}(A, B) \rightarrow \mathrm{Hom}_{D([\Sm,\ShvA])/\sim}(\Phi(A),\Phi(B)) $$
	is a bijection of abelian groups.
	In particular $\Phi$ is fully faithful on all shifts of objects of the form $[\Gmt{n},\cc M_{\Cor}(X)]$, where $n \in \bb N$, $X \in \Smk$.
	Since the objects $[\Gmt{n},\cc M_{\Cor}(X)]$ are isomorphic to the objects $[\Gmt{n},I(-)] \otimes X$ in $T_{\cc C}/\sim_{\cc{C}}$, it follows that $\Phi$ is fully faithful on all shifts of objects from the set of compact generators $\Sigma$.
	
	This verifies the second condition from Lemma \ref{garkushaLemma}.
	It now follows that 
	$$\Phi: T_{\cc C}/\sim_{\cc{C}}[1/p] \rightarrow D([\Sm,\ShvA])/\sim[1/p]$$
	is an equivalence of triangulated categories.
	Recall that by Lemma \ref{DMSmcomparison} we have a canonical equivalence of triangulated categories
	$$DM_{\Cor}[\Sm] \rightarrow D([\Sm,\ShvA])/\sim.$$
	We then have a commutative diagram
	$$\xymatrix{DM_{\Cor}[\Sm][1/p] \ar[r]^(0.4)\sim & D([\Sm,\ShvA])/\sim[1/p] \ar[dr]^{ev_{\Gm}} \\
	& T_{\cc C}/\sim_{\cc{C}}[1/p] \ar[r]^\sim_{ev_{\Gm}} \ar[u]_\sim^{\Phi} & DM_{\Cor}[1/p]  }$$
	which shows that the evaluation functor
	$$ev_{\Gm}: DM_{\Cor}[\Sm][1/p] \rightarrow DM_{\Cor}[1/p]$$
	is an equivalence of categories.
	This completes the proof of Theorem \ref{mainthm}.
\end{proof}

\end{document}